\documentclass[11pt,a4paper]{amsart}

\usepackage[utf8]{inputenc}
\usepackage[T1]{fontenc}
\usepackage[margin=1.1in]{geometry}
\usepackage{mathtools,amssymb,mathrsfs}
\usepackage{needspace}
\usepackage[dvipsnames]{xcolor}
\definecolor{linkblue}{RGB}{56,84,124}
\usepackage{hyperref}
\hypersetup{
	colorlinks=true,
	linkcolor=linkblue,
	citecolor=linkblue,
	urlcolor=linkblue
}

\mathtoolsset{showonlyrefs}
\numberwithin{equation}{section}

\newcommand{\IR}{\mathbb R}
\newcommand{\br}{{\normalfont{\boldsymbol r}}}
\newcommand{\dd}{\,\mathrm d}
\newcommand{\dual}[3]{\left\langle #1,#2\right\rangle_{#3}}
\newcommand{\norm}[1]{\left\|#1\right\|}
\newcommand{\divGh}{\operatorname{div}_{\widehat\Gamma}}
\newcommand{\curl}{\operatorname{\mathbf{curl}}}
\newcommand{\ii}{\imath}

\theoremstyle{plain}
\newtheorem{theorem}{Theorem}[section]
\newtheorem{assumption}{Assumption}[section]
\newtheorem{lemma}{Lemma}[section]
\newtheorem{proposition}{Proposition}[section]
\newtheorem{corollary}{Corollary}[section]
\theoremstyle{remark}
\newtheorem{remark}{Remark}[section]

\title[Shape Holomorphy of the Helmholtz BIOs and the Maxwell EFIE]{Shape Holomorphy of the Helmholtz Calder\'on Projector and the Maxwell EFIE on $\mathscr C^{1,1}$ Surfaces and Screens}
\author[Fernando Henr\'iquez]{Fernando Henr\'iquez}
\date{}

\makeatletter
\let\ams@setauthors\@setauthors
\def\@setauthors{%
  \ams@setauthors
  \begin{center}
    \vspace{7pt}
    \normalfont\footnotesize
    Institute for Analysis and Scientific Computing,\\
    Technische Universit\"{a}t Wien\\
    Wiedner Hauptstrasse 8-10, A-1040 Wien, Austria\\[2pt]
    \texttt{fernando.henriquez@asc.tuwien.ac.at}
  \end{center}%
}
\makeatother

\begin{document}
\begin{abstract}
We establish the holomorphic dependence of the Helmholtz Calder\'on
projector on closed $\mathscr C^{1,1}$ surfaces and of the operator
associated with the Maxwell electric field integral equation (EFIE)
on closed and open surfaces of the same regularity.
More precisely, after transporting the boundary integral operators (BIOs)
to fixed reference spaces, we prove that the resulting maps admit bounded
holomorphic extensions to complex neighborhoods of compact sets of
admissible boundary transformations, a property referred to as
\emph{shape holomorphy}.
A key step in our proof consists in extending the 
results in [Lanza de Cristoforis and Rossi,
\emph{Analytic Methods of Analysis and Differential Equations: AMADE 2006}
(2008), pp.~193--220]
and [Dalla Riva, Luzzini, and Musolino,
\emph{Inverse Problems} 38 (2022), 055004]
from Schauder to fractional Sobolev spaces.

For the EFIE, the single layer representation and the surface Piola
transformation give the result in Maxwell trace spaces, whereas the result for 
open surfaces or \emph{screens} is obtained by introducing suitable restriction and extension by zero
operators. As opposed to existing results, we do not rely on an affine-parametric representation
of the boundary deformations, thus allowing for more general parametric shape variations. 

Finally, we establish shape holomorphy
of the \emph{shape-to-solution} map for the EFIE, subject to interior electric
Maxwell nonresonance on closed surfaces but with no such condition on screens.
Under suitable summability and stability assumptions, these shape holomorphy
results provide an analytic basis for the dimension-robust convergence of
interpolation and quadrature methods used in computational uncertainty
quantification (UQ) of scattering problems with
high-dimensional, shape-parametric uncertain geometry.
\end{abstract}
\maketitle

\section{Introduction}

Many computational problems in science and engineering require the repeated
solution of boundary value problems on domains with varying geometry.
Such problems arise in shape optimization, inverse shape identification,
engineering design, and uncertainty quantification (UQ).
In these settings, one is interested in the dependence of the solution,
or of a quantity of interest (QoI), on the shape of the domain.
In many query and real time applications, solving the full boundary value
problem for every admissible geometry may be computationally prohibitive.

Following a parametric approach, families of admissible domains can be
represented through transformations of a fixed reference domain.
These transformations may depend on finitely many parameters, as is often
the case in geometric design, or on countably many parameters, as in models
of spatially distributed boundary uncertainty.
For well posed problems, transporting the solutions to the reference
domain yields a solution map with values in a fixed function space.
Approximating this map efficiently becomes challenging when the number of
parameters is large. For tensor product methods, for instance, the
computational effort grows exponentially with the number of parameters,
leading to the so called curse of dimensionality in the parameter space.

A strategy to overcome this difficulty relies on extending the parametric
maps holomorphically to products of complex domains whose sizes reflect
the influence of the individual parameters \cite{Mujica1986}.
Uniform bounds on these extensions, together with suitable summability
assumptions on the parametrization, allow us to exploit this anisotropic
dependence and obtain approximation rates that are independent of the
parametric dimension \cite{ChkifaCohenSchwab2015}.
This form of parametric holomorphy is used in the analysis of sparse
polynomial interpolation \cite{ChkifaCohenSchwab2014} and higher order
quasi Monte Carlo integration \cite{DickLeGiaSchwab2016}.
It also underlies the analysis of artificial neural network surrogates
for parametric solution maps \cite{SchwabZech2019} and reduced basis
methods \cite{CohenDeVore2016}.

To apply these approximation results to a boundary value problem, one must
establish the required holomorphic dependence on the geometric parameters.
For problems admitting a boundary integral formulation, this leads to the
study of the corresponding boundary integral operators (BIOs)
\cite{ColtonKress2013}.
In this work, we consider the complete Calder\'on projector for the Helmholtz
equation on closed $\mathscr C^{1,1}$ surfaces and the Maxwell electric field
integral equation (EFIE) operator on closed surfaces and screens of the same
regularity. After transport to a fixed reference boundary, these operators
act between fixed fractional Sobolev and Maxwell trace spaces.
The purpose of this work is to establish a holomorphic extension of the
transported operators, in operator norm, to a complex neighborhood of the
admissible boundary transformations. We refer to this property as
\emph{shape holomorphy}. For boundary integral equations (BIEs) that are well
posed, holomorphy of the transported operators and right hand sides implies
holomorphic dependence of the solutions.

\subsection{Previous Work}

The dependence of BIOs on the boundary has been studied extensively
through shape derivatives.
For smooth boundaries and smooth deformations, electromagnetic BIOs
transported to fixed Maxwell trace spaces are infinitely G\^ateaux
differentiable \cite[Theorems~6.2 and~6.3]{CostabelLeLouer2012ShapeII}.
Such differentiability with respect to real boundary perturbations
does not by itself imply the complex Fr\'echet differentiability
required for shape holomorphy.

A different approach starts from the real analytic dependence of
layer potentials on their supporting boundary and density
\cite{LanzaRossi2004}.
For the Helmholtz equation, this dependence is established in
\cite{LanzaRossiHelmholtz2008}
and \cite[Lemma~2.5]{DallaRivaLuzziniMusolino2022}.
For the Laplace equation,
\cite[Corollaries~3.3 and~3.4]{DallaRivaLuzziniMusolino2022ESAIM}
establishes real analyticity and a holomorphic extension of the complete
Calder\'on projector in H\"older spaces on $\mathscr C^{1,\alpha}$
boundaries.
The deformation variable in these analyses is an embedding of a fixed
reference boundary and is not restricted to an affine parametric
representation.
The Helmholtz results cited above are the starting point for our
analysis. They are formulated in H\"older spaces, and their transfer to
the fractional Sobolev trace spaces is a central part of the present work.

Shape holomorphy in fractional Sobolev spaces has also been studied in
two dimensions, for both closed curves and open arcs.
For smooth closed curves, shape holomorphy of the Laplace Calder\'on
projector is established in \cite{HS21}, and the corresponding analysis
for the Helmholtz BIOs is developed in \cite{Henriquez2021}.
For collections of open arcs, \cite{PHJ24} establishes shape holomorphy
in weighted Sobolev spaces.
For parametrically varying Lipschitz boundaries in three dimensions,
\cite{DolzHenriquez2024} establishes shape holomorphy in an $L^2$
setting, allowing polyhedral boundaries.
That analysis does not establish holomorphy of the complete Calder\'on
projector on the fractional Sobolev spaces of Dirichlet and Neumann
traces.
Shape holomorphy has also been established for volume integral equations
in \cite{HenriquezLabarca2026}.

For the Maxwell EFIE on closed $\mathscr C^{1,1}$ surfaces,
\cite[Section~2.2 and Theorem~6.1]{EscapilJerez2026EFIE} establishes
parametric holomorphy of the transported operator on its natural energy
space for affine parametric boundary transformations with countably
many parameters.
Under interior electric nonresonance, the same work establishes
holomorphy of the solution of the EFIE and the far-field, and derives
sparse approximation estimates
\cite[Theorems~6.2, 7.2, and~7.3]{EscapilJerez2026EFIE}.

\Needspace{8\baselineskip}
\subsection{Contributions}

We establish shape holomorphy of the complete Helmholtz Calder\'on
projector on closed $\mathscr C^{1,1}$ surfaces
(Theorem~\ref{thm:c11-helmholtz-calderon-holomorphy}), and of the
Maxwell EFIE operator on closed
$\mathscr C^{1,1}$ surfaces (Theorem~\ref{thm:c11-efie-holomorphy})
and on screens of the same regularity
(Theorem~\ref{thm:c11-screen-holomorphy}).
For each compact family of admissible boundary transformations, the
transported operators admit bounded holomorphic extensions to a complex
neighborhood of that family, with values in the corresponding spaces of
bounded operators on the reference boundary.

Our proof first transfers the existing H\"older analyticity results
for the scalar Helmholtz operators to fractional Sobolev spaces.
The transfer uses Lax's lemma in the form given in
\cite[Theorem~4.13]{Kress2014}, based on
\cite{Lax1954}.
The single layer representation of the EFIE and the surface Piola
transformation then yield shape holomorphy on closed surfaces.
For screens, fixed restriction and extension by zero operators reduce
the holomorphy argument to the closed surface result.

The boundary transformation itself is the variable in our analysis,
rather than the coefficients of a prescribed affine parametrization
as in \cite[Section~2.2 and Theorem~6.1]{EscapilJerez2026EFIE}.
This separates the regularity of the boundary integral operators with
respect to the geometry from the choice of a parametrization.
The same operator result can therefore be used for different
parametrizations, including nonlinear ones, without repeating the
holomorphy argument for each representation.
For analytic parametrizations, parametric holomorphy follows by
composition. More generally, one may consider nonlinear
parametrizations of Gevrey regularity, as in
\cite[Section~2]{DjurdjevacKaarniojaSchillingsZepernick2026}.
If their ranges lie in a compact family of admissible transformations
and their parametric derivative bounds hold in the
$\mathscr C^{1,1}(\widehat\Gamma;\mathbb R^3)$ norm, composition with
the holomorphic operator maps preserves the Gevrey order
\cite[Section~2.4]{HarbrechtSchmidlinSchwab2024}.

We also establish shape holomorphy of the EFIE solutions on closed
surfaces and screens, provided that the data transported to the reference
boundary admit bounded holomorphic extensions
(Corollaries~\ref{cor:c11-closed-solution-holomorphy}
and~\ref{cor:c11-screen-solution-holomorphy}).
For closed surfaces, we exclude interior electric Maxwell eigenvalues;
no such restriction is needed for screens.
In the closed-surface setting, we further establish bounded holomorphic
extensions of the electric far-field
(Corollary~\ref{cor:c11-closed-far-field-holomorphy}),
its pointwise squared magnitude, and its squared $L^2$ norm
(Corollary~\ref{cor:c11-squared-far-field-holomorphy}).

\subsection{Outline}

This work is organized as follows.
Section~\ref{sec:c11-preliminaries} collects the geometric and functional
preliminaries, including the extension of operators from H\"older to
fractional Sobolev spaces.
Section~\ref{sec:c11-helmholtz-holomorphy} establishes shape holomorphy
of the four scalar Helmholtz BIOs and the corresponding Calder\'on projector.
Section~\ref{sec:c11-efie} introduces the Maxwell trace spaces and the
surface Piola transformation, and treats the EFIE on closed surfaces,
its solution, and the electric far-field.
Section~\ref{sec:c11-screens} treats the screen EFIE and its solution,
including the screen spaces and Piola transformation.
Section~\ref{sec:c11-conclusions} contains the concluding remarks.
Technical results and deferred proofs are collected in the appendices.

\subsection{Notation}
\label{sec:c11-notation}
Function and distribution spaces are over $\mathbb C$ unless a real
target space is explicitly specified. In particular, $L^p(\Sigma)$,
$H^s(\Sigma)$, and $\mathscr C^{k,\alpha}(\Sigma)$ consist of
complex-valued functions or distributions; their vector and matrix
counterparts are likewise complex unless stated otherwise.
The notation $\mathscr C^{k,\alpha}(\Sigma;\mathbb R^3)$ and
$\mathscr C^{k,\alpha}(\Sigma;\mathbb C^3)$ specifies spaces of
real and complex vector-valued functions on $\Sigma$, respectively.

For complex Banach spaces $E$ and $F$, $\mathscr L(E,F)$ denotes the
space of bounded complex linear operators from $E$ to $F$, equipped with the
operator norm
\begin{equation}
	\|T\|_{\mathscr L(E,F)}
	\coloneqq \sup_{u\in E\setminus\{0\}}\frac{\|Tu\|_F}{\|u\|_E},
	\quad \forall T\in\mathscr L(E,F).
\end{equation}
We write $\mathscr L_{\text{\normalfont iso}}(E,F)$ for the set of bounded linear
isomorphisms from $E$ to $F$ with bounded inverses.
We write $E'\coloneqq \mathscr L(E,\mathbb C)$ for the continuous
dual. The evaluation pairing
$\langle\cdot,\cdot\rangle_{E'\times E}:E'\times E\to\mathbb C$
is defined by
\begin{equation}
	\dual{f}{u}{E'\times E}\coloneqq f(u),
	\quad \forall f\in E',\quad\forall u\in E.
\end{equation}
This pairing is complex bilinear, with no complex conjugation.
The Sobolev and Maxwell duality pairings use this convention.
Inner products are Hermitian, linear in the first argument and
conjugate linear in the second.
For all $T\in\mathscr L(E,F)$, its dual operator
$T':F'\to E'$ is defined by
\begin{equation}
	\dual{T'f}{u}{E'\times E}
	=\dual{f}{Tu}{F'\times F},
	\quad \forall f\in F',\quad\forall u\in E.
\end{equation}
The map $T\mapsto T'$ from $\mathscr L(E,F)$ to $\mathscr L(F',E')$
is bounded and complex linear.


\section{Preliminaries}
\label{sec:c11-preliminaries}

\subsection{Boundary Transformations}
Throughout this work, let $\widehat{\Gamma}$ be the boundary
of a bounded, simply connected $\mathscr{C}^{1,1}$ domain $\widehat{\text{D}}$ with connected
boundary, i.e.,
$\widehat{\Gamma} \coloneqq \partial \widehat{\text{D}}$, which we
refer to as the reference boundary.
We consider a collection $\{\Gamma_\br\}_{\br \in \mathfrak{T}}$
of boundaries obtained as images of the reference boundary
under $\br \in \mathfrak{T}$, where $\mathfrak{T} \subset\mathscr C^{1,1}(\widehat\Gamma;\IR^3)$
is the set of boundary transformations.
We impose the following assumptions.

\begin{assumption}[Admissible boundary transformations]
\label{ass:c11-deformations}
Let $\mathfrak T\subset\mathscr C^{1,1}(\widehat\Gamma;\IR^3)$ satisfy
the following conditions.
\begin{itemize}
 \item[\normalfont{(i)}] The set $\mathfrak T$ is compact.
 \item[\normalfont{(ii)}] For each $\br\in\mathfrak T$, the image
 $\Gamma_\br\coloneqq \br(\widehat\Gamma)$ is the boundary of a bounded, simply connected
 $\mathscr C^{1,1}$ domain $\normalfont\text{D}_\br$ with connected boundary.
 \item[\normalfont{(iii)}] For each $\br\in\mathfrak T$, the map
 $\br:\widehat\Gamma\to\Gamma_\br\subset\mathbb R^3$ is a $\mathscr C^{1,1}$
 diffeomorphism.
\end{itemize}
\end{assumption}


Let $\widehat{\boldsymbol n}$ and $\boldsymbol n_\br$ denote the
outward unit normals on $\widehat\Gamma$ and $\Gamma_\br$, respectively,
and let $\mathsf I_3$ denote the identity matrix. We set
$\boldsymbol N_\br\coloneqq\boldsymbol n_\br\circ\br$.
At each $\widehat{\mathbf x}\in\widehat\Gamma$, the matrix
$\widehat P(\widehat{\mathbf x})\coloneqq
\mathsf I_3-\widehat{\boldsymbol n}(\widehat{\mathbf x})
\widehat{\boldsymbol n}(\widehat{\mathbf x})^\top$
is the orthogonal projection onto the tangent plane
$T_{\widehat{\mathbf x}}\widehat\Gamma$.
Its action on a vector $\boldsymbol v\in\mathbb C^3$ is
$\widehat P\boldsymbol v=\boldsymbol v-
\widehat{\boldsymbol n}(\widehat{\boldsymbol n}\cdot\boldsymbol v)$.

For each $\br\in\mathfrak T$, we define the surface gradients by
\begin{equation}
	\nabla_{\widehat\Gamma}u
	\coloneqq\widehat P\,(\nabla\widetilde u)|_{\widehat\Gamma}
	\quad\text{and}\quad
	\nabla_{\Gamma_\br}v
	\coloneqq(\mathsf I_3-\boldsymbol n_\br\boldsymbol n_\br^\top)
	             (\nabla\widetilde v)|_{\Gamma_\br},
\end{equation}
for all $u\in\mathscr C^1(\widehat\Gamma)$ and
$v\in\mathscr C^1(\Gamma_\br)$.
Here $\widetilde u$ and $\widetilde v$ are arbitrary local
$\mathscr C^1$ extensions of $u$ and $v$, respectively.
These definitions are independent of the chosen extensions, and the
surface gradients admit bounded extensions
\cite[Section~2 and Proposition~3.6]{BCS2002}
\begin{equation}
	\nabla_\Sigma\in\mathscr L\bigl(
	H^{\frac12}(\Sigma),H^{-\frac12}(\Sigma;\mathbb C^3)\bigr),
	\quad\Sigma\in\{\widehat\Gamma,\Gamma_\br\}.
\end{equation}

For each $\br\in\mathfrak T$, we define its tangential differential by
\begin{equation}
	F_\br\coloneqq D_{\widehat\Gamma}\br
	=D\widetilde\br|_{\widehat\Gamma}\,\widehat P
	\in\mathscr C^{0,1}(\widehat\Gamma;\IR^{3\times3}).
	\label{eq:c11-tangential-differential}
\end{equation}
Here $\widetilde\br\in\mathscr C^{1,1}(U;\mathbb R^3)$ is any local
extension of $\br$, where $U\subset\mathbb R^3$ is open.
The tangential differential is independent of the extension and satisfies
$F_\br\widehat{\boldsymbol n}=0$.
Since $\br$ is a $\mathscr{C}^{1,1}$ diffeomorphism, the restriction of
$F_\br(\widehat{\mathbf x})$ to $T_{\widehat{\mathbf x}}\widehat\Gamma$
is an isomorphism onto $T_{\br(\widehat{\mathbf x})}\Gamma_\br$.
We write $F_\br^{-1}(\widehat{\mathbf x})$ for the inverse from
$T_{\br(\widehat{\mathbf x})}\Gamma_\br$ to
$T_{\widehat{\mathbf x}}\widehat\Gamma$, extended by zero on
$\operatorname{span}\{\boldsymbol N_\br(\widehat{\mathbf x})\}$. Thus
\begin{equation}
	F_\br^{-1}F_\br=\widehat P
	\quad\text{and}\quad
	F_\br^{-1}\boldsymbol N_\br=0.
	\label{eq:c11-tangential-inverse-convention}
\end{equation}
The product $F_\br F_\br^{-1}$ is the projection
$\mathsf I_3-\boldsymbol N_\br\boldsymbol N_\br^\top$ onto
$T_{\br(\widehat{\mathbf x})}\Gamma_\br$ at each
$\widehat{\mathbf x}\in\widehat\Gamma$.

For $q\in\mathscr C^{0,1}(\widehat\Gamma;E)$, with
$E\in\{\mathbb R,\mathbb R^3,\mathbb R^{3\times3},
       \mathbb C,\mathbb C^3,\mathbb C^{3\times3}\}$, we
introduce the Lipschitz seminorm as
\begin{equation}
	|q|_{\mathscr C^{0,1}(\widehat\Gamma;E)}
	\coloneqq \sup_{\substack{\widehat{\mathbf{x}},\widehat{\mathbf{y}}
	                     \in\widehat\Gamma\\
	                   \widehat{\mathbf{x}}\ne\widehat{\mathbf{y}}}}
	  \frac{\|q(\widehat{\mathbf{x}})-q(\widehat{\mathbf{y}})\|_E}
	       {\|\widehat{\mathbf{x}}-\widehat{\mathbf{y}}\|}.
\end{equation}
Here $\|\cdot\|_E$ is the absolute value, Euclidean norm, or Frobenius
norm, respectively.
For all $\boldsymbol u\in\mathscr C^{1,1}(\widehat\Gamma;\mathbb R^3)$
and $\boldsymbol v\in\mathscr C^{1,1}(\widehat\Gamma;\mathbb C^3)$,
we use the norms
\begin{equation}
	\begin{aligned}
	\|\boldsymbol u\|_{\mathscr C^{1,1}(\widehat\Gamma;\mathbb R^3)}
	&\coloneqq \|\boldsymbol u\|_{\mathscr C^1(\widehat\Gamma;\mathbb R^3)}
	  +|D_{\widehat\Gamma}\boldsymbol u|_
	      {\mathscr C^{0,1}(\widehat\Gamma;\mathbb R^{3\times3})},\\
	\|\boldsymbol v\|_{\mathscr C^{1,1}(\widehat\Gamma;\mathbb C^3)}
	&\coloneqq \|\boldsymbol v\|_{\mathscr C^1(\widehat\Gamma;\mathbb C^3)}
	  +|D_{\widehat\Gamma}\boldsymbol v|_
	      {\mathscr C^{0,1}(\widehat\Gamma;\mathbb C^{3\times3})}.
	\end{aligned}
\end{equation}
Here the tangential differentials have the same matrix representation
as in~\eqref{eq:c11-tangential-differential}.
At each point of the reference boundary, let
$\boldsymbol e_1,\boldsymbol e_2$ be orthonormal tangent vectors
with $\boldsymbol e_1\times\boldsymbol e_2=\widehat{\boldsymbol n}$.
The identity
\begin{equation}
	\operatorname{cof}(F_\br)\widehat{\boldsymbol n}
	=(F_\br\boldsymbol e_1)\times(F_\br\boldsymbol e_2)
\end{equation}
shows that the surface Jacobian is
\begin{equation}
	J_{\text S,\br}
	=\bigl\|\operatorname{cof}(F_\br)\widehat{\boldsymbol n}\bigr\|\ne0.
	\label{eq:c11-surface-jacobian}
\end{equation}
The surface Jacobian is positive because $F_\br$ is invertible between
the tangent planes.

\begin{lemma}
\label{lem:c11-jacobian-multiplier}
Let $\mathfrak T$ satisfy Assumption~\ref{ass:c11-deformations}.
For each $\br\in\mathfrak T$, the functions $J_{\text S,\br}$ and
$J_{\text S,\br}^{-1}$ belong to $\mathscr C^{0,1}(\widehat\Gamma)$.
Multiplication by either function is bounded on
$\mathscr C^{0,\alpha}(\widehat\Gamma)$ for $0<\alpha<1$,
and on $H^s(\widehat\Gamma)$ for $-1\leq s\leq1$.
\end{lemma}

\begin{proof}
The proof is given in Appendix~\ref{sec:proof_c11-jacobian-multiplier}.
\end{proof}

To study complex Fr\'echet differentiability with respect to the
boundary transformation, we introduce, for $\delta>0$, the complex
neighborhood
\begin{equation}
	\mathfrak T_\delta
	\coloneqq \bigl\{\br\in\mathscr C^{1,1}(\widehat\Gamma;\mathbb C^3):
    \exists \widetilde\br \in \mathfrak{T} \text{ such that }
	\|\br-\widetilde\br\|_{\mathscr C^{1,1}(\widehat\Gamma;\mathbb C^3)}
	<\delta\bigr\}.
\end{equation}

The surface Jacobian, its reciprocal, the transported normal, and the
tangential inverse admit bounded holomorphic extensions to a common
complex neighborhood of $\mathfrak T$.
\begin{lemma}
\label{lem:c11-geometric-holomorphy}
Let $\mathfrak T$ be a set of boundary transformations
satisfying Assumption~\ref{ass:c11-deformations}.
Then there exists $\delta>0$ such that the following assertions hold.
\begin{itemize}
    \item[(i)]
    The maps
    $\mathfrak T\ni\br\mapsto J_{\text S,\br} \in \mathscr C^{0,1}(\widehat\Gamma)$ and
    $\mathfrak T\ni\br\mapsto J_{\text S,\br}^{-1} \in \mathscr C^{0,1}(\widehat\Gamma)$
    admit bounded holomorphic extensions to $\mathfrak T_\delta$.
    \item[(ii)]
    The maps
    $\mathfrak T\ni\br\mapsto\boldsymbol N_\br
      \in\mathscr C^{0,1}(\widehat\Gamma;\mathbb C^3)$
    and
    $\mathfrak T\ni\br\mapsto F_\br^{-1}
      \in\mathscr C^{0,1}(\widehat\Gamma;\mathbb C^{3\times3})$
    admit bounded holomorphic extensions to $\mathfrak T_\delta$.
\end{itemize}
\end{lemma}

\begin{proof}
The proof is given in Appendix~\ref{sec:proof_c11-geometric-holomorphy}.
\end{proof}

\subsection{Pullback Operator}
\label{sec:c11-pullbacks}
For each $\br\in\mathfrak T$, with $\mathfrak{T}$
satisfying Assumption~\ref{ass:c11-deformations}, we define
\begin{equation}
	\begin{aligned}
	\tau_\br&:\mathscr C^0(\Gamma_\br)\to\mathscr C^0(\widehat\Gamma):
	&\tau_\br f&\coloneqq f\circ\br,\\
	\tau_\br^{-1}&:\mathscr C^0(\widehat\Gamma)\to\mathscr C^0(\Gamma_\br):
	&\tau_\br^{-1}\widehat f&\coloneqq \widehat f\circ\br^{-1}.
	\end{aligned}
\end{equation}

For each $\br\in\mathfrak T$, the surface change of variables gives
\begin{equation}\label{eq:c11-surface-change-of-variables}
	\int_{\Gamma_\br}f(\mathbf{x})\,\dd s_{\mathbf{x}}
	=
    \int_{\widehat\Gamma}f\bigl(\br(\widehat{\mathbf{x}})\bigr)
	J_{\text S,\br}(\widehat{\mathbf{x}})\,
    \dd s_{\widehat{\mathbf{x}}},
\end{equation}
for all $f\in L^1(\Gamma_\br)$.

For $\Sigma \in \{\Gamma_\br,\widehat\Gamma\}$ and $0\leq s\leq1$, we denote by
$\langle\cdot,\cdot\rangle_{\Sigma}$ the complex bilinear duality pairing
between $H^{-s}(\Sigma)$ and $H^s(\Sigma)$,
with $L^2(\Sigma)$ as pivot space.
For regular functions, it agrees with $\int_\Sigma f\phi\,\dd s$,
without complex conjugation.
For vector fields, we use its component-wise extension from
$(H^{-s}(\Sigma))^3\times(H^s(\Sigma))^3$ to $\mathbb C$.

Lemma~\ref{lem:c11-pullback-isomorphisms} shows that, for
$0\leq s\leq1$, the pullback operators $\tau_\br$ and
$\tau_\br^{-1}$ extend uniquely to bounded operators between
$H^s(\Gamma_\br)$ and $H^s(\widehat\Gamma)$.
For $-1\leq s<0$, $f\in H^s(\Gamma_\br)$, and
$\widehat f\in H^s(\widehat\Gamma)$, we define
\begin{equation}
    \begin{aligned}
	\langle\tau_\br f,\widehat\phi\rangle_{\widehat\Gamma}
	&\coloneqq \langle f,\tau_\br^{-1}
	            (J_{\text S,\br}^{-1}\widehat\phi)\rangle_{\Gamma_\br},\\
	\langle\tau_\br^{-1}\widehat f,\phi\rangle_{\Gamma_\br}
	&\coloneqq \langle\widehat f,J_{\text S,\br}\tau_\br\phi
	      \rangle_{\widehat\Gamma},
    \end{aligned}
	\label{eq:c11-distribution-pullbacks}
\end{equation}
for all $\widehat\phi\in H^{-s}(\widehat\Gamma)$ and
$\phi\in H^{-s}(\Gamma_\br)$, respectively.
By the surface change of variables \eqref{eq:c11-surface-change-of-variables}, these definitions agree with
composition for regular densities.

\begin{lemma}
\label{lem:c11-pullback-isomorphisms}
Let $\mathfrak T$ satisfy Assumption~\ref{ass:c11-deformations}.
For each $\br\in\mathfrak T$, the following assertions hold.
\begin{itemize}
 \item[\normalfont{(i)}] For $-1\leq s\leq1$,
 \begin{equation}
 	\tau_\br\in\mathscr L_{\text{\normalfont iso}}
 	\bigl(H^s(\Gamma_\br),H^s(\widehat\Gamma)\bigr)
 	\quad\text{and}\quad
 	\tau_\br^{-1}\in\mathscr L_{\text{\normalfont iso}}
 	\bigl(H^s(\widehat\Gamma),H^s(\Gamma_\br)\bigr).
 \end{equation}
 \item[\normalfont{(ii)}] For $k=0,1$ and $0<\alpha<1$,
 \begin{equation}
 	\tau_\br\in\mathscr L_{\text{\normalfont iso}}
 	\bigl(\mathscr C^{k,\alpha}(\Gamma_\br),
 	      \mathscr C^{k,\alpha}(\widehat\Gamma)\bigr)
 	\quad\text{and}\quad
 	\tau_\br^{-1}\in\mathscr L_{\text{\normalfont iso}}
 	\bigl(\mathscr C^{k,\alpha}(\widehat\Gamma),
 	      \mathscr C^{k,\alpha}(\Gamma_\br)\bigr).
 \end{equation}
\end{itemize}
\end{lemma}

\begin{proof}
The proof is given in Appendix~\ref{app:c11-pullbacks}.
\end{proof}

\subsection{Extension from H\"older to Sobolev Spaces}
\label{sec:c11-sobolev-extension}
We transfer the Helmholtz analyticity results in
\cite{LanzaRossiHelmholtz2008} and
\cite{DallaRivaLuzziniMusolino2022}
from H\"older to fractional Sobolev operator spaces.
In particular, \cite[Lemma~2.5, items~(i)--(iii)]{DallaRivaLuzziniMusolino2022} establishes
real-analytic dependence on the boundary transformation of the single
layer, double layer, and adjoint double layer BIOs transported to the
reference boundary.
Following the symmetry argument discussed in
\cite[Preface to the Classics Edition, p.~xi]{ColtonKress2013},
we construct a bounded, complex linear extension map using symmetry
with respect to the bilinear pairing
$\langle\cdot,\cdot\rangle_{\widehat\Gamma}$.
Lemma~\ref{lem:c11-sobolev-extension} gives the extension estimate
that transfers holomorphy and uniform bounds to the Sobolev operator norm.

\begin{lemma}
\label{lem:c11-sobolev-extension}
For a given $0<\alpha<1$, let
$\widehat{\mathsf T}\in\mathscr L(\mathscr C^{0,\alpha}(\widehat\Gamma),
\allowbreak\mathscr C^{1,\alpha}(\widehat\Gamma))$ satisfy
\begin{equation}\label{eq:symmetry_Lax}
	\int_{\widehat\Gamma}(\widehat{\mathsf T}u)(\widehat{\mathbf{x}})
	    v(\widehat{\mathbf{x}})\,\dd s_{\widehat{\mathbf{x}}}
	=
	\int_{\widehat\Gamma}u(\widehat{\mathbf{x}})
	    (\widehat{\mathsf T}v)(\widehat{\mathbf{x}})\,
	    \dd s_{\widehat{\mathbf{x}}},
\end{equation}
for all $u,v\in\mathscr C^{0,\alpha}(\widehat\Gamma)$.
Then the following assertions hold.
\begin{itemize}
 \item[\normalfont{(i)}] The operator $\widehat{\mathsf T}$ has a unique
 bounded extension
 \begin{equation}
 	\mathcal E(\widehat{\mathsf T}):
 	H^{-\frac{1}{2}}(\widehat\Gamma)\to H^{\frac{1}{2}}(\widehat\Gamma).
 \end{equation}
 \item[\normalfont{(ii)}] For operators satisfying \eqref{eq:symmetry_Lax}, the extension
 operation $\mathcal E$ is complex linear and satisfies
 \begin{equation}
 	\|\mathcal E(\widehat{\mathsf T})\|_
 	    {\mathscr L(H^{-\frac{1}{2}}(\widehat\Gamma),
 	                H^{\frac{1}{2}}(\widehat\Gamma))}
 	\leq C
 	\|\widehat{\mathsf T}\|_
 	    {\mathscr L(\mathscr C^{0,\alpha}(\widehat\Gamma),
 	                \mathscr C^{1,\alpha}(\widehat\Gamma))},
 \end{equation}
 where $C=C(\widehat\Gamma,\alpha)$ is given explicitly
 in~\eqref{eq:c11-sobolev-extension-constant}.
 \item[\normalfont{(iii)}] 
For each $u\in\mathscr C^{0,\alpha}(\widehat\Gamma)$, it holds that
 \begin{equation}
 \label{eq:c11-sobolev-extension-agreement}
 	\mathcal E(\widehat{\mathsf T})u=\widehat{\mathsf T}u.
 \end{equation}
\end{itemize}
\end{lemma}

\begin{proof}
We first prove a Sobolev estimate for $\widehat{\mathsf T}$ on its
original domain $\mathscr C^{0,\alpha}(\widehat\Gamma)$, and then
extend it by density. For this estimate, we use the Laplace single
layer BIO to define an equivalent norm on
$H^{-\frac12}(\widehat\Gamma)$.
Let $\widehat{\mathsf V}_0$ denote the Laplace single layer BIO on
$\widehat\Gamma$, defined, for $u\in\mathscr C^0(\widehat\Gamma)$, by 
\begin{equation}
	(\widehat{\mathsf V}_0u)(\widehat{\mathbf{x}})
	\coloneqq
	\int_{\widehat\Gamma}
	\frac{u(\widehat{\mathbf{y}})}
	{4\pi\|\widehat{\mathbf{x}}-\widehat{\mathbf{y}}\|}
	\,\dd s_{\widehat{\mathbf{y}}},
	\quad\forall\widehat{\mathbf{x}}\in\widehat\Gamma,
\end{equation}
It is a symmetric isomorphism from
$\mathscr C^{0,\alpha}(\widehat\Gamma)$
to $\mathscr C^{1,\alpha}(\widehat\Gamma)$ for $0<\alpha<1$
\cite[Lemma~2.5]{BonnaillieDallaRivaDambrineMusolino2018}.
Its extension to the Sobolev spaces is a bounded, positive and coercive
isomorphism from $H^{-\frac12}(\widehat\Gamma)$ to
$H^{\frac12}(\widehat\Gamma)$, see e.g.,
\cite[Corollary~8.13]{McLean2000}.
This extension agrees with $\widehat{\mathsf V}_0$ on
$\mathscr C^{0,\alpha}(\widehat\Gamma)$.

The operator $\widehat{\mathsf V}_0$ is symmetric and coercive,
thus implying that
\begin{equation}
	(u,v)_0\coloneqq
	\langle u,\widehat{\mathsf V}_0\overline v\rangle_{\widehat\Gamma},
\end{equation}
for all $u,v\in H^{-\frac12}(\widehat\Gamma)$, defines an inner product in $ H^{-\frac{1}{2}}(\widehat\Gamma)$
and an induced norm $\|u\|_0\coloneqq(u,u)_0^{1/2}$ that is equivalent to
$\|u\|_{H^{-\frac{1}{2}}(\widehat\Gamma)}$.
The embedding of $\mathscr C^{0,\alpha}(\widehat\Gamma)$ in
$H^{-\frac12}(\widehat\Gamma)$ is continuous.

Set $\widehat{\mathsf K}\coloneqq
\widehat{\mathsf V}_0^{-1}\widehat{\mathsf T}
\in\mathscr L(\mathscr C^{0,\alpha}(\widehat\Gamma),
              \mathscr C^{0,\alpha}(\widehat\Gamma))$, and let
$\widehat{\mathsf J}$ be conjugation on
$\mathscr C^{0,\alpha}(\widehat\Gamma)$.
Symmetry of $\widehat{\mathsf T}$ and the real symmetric kernel of
$\widehat{\mathsf V}_0$ give
\begin{equation}
	(\widehat{\mathsf K}u,v)_0
	=\langle\overline v,\widehat{\mathsf T}u\rangle_{\widehat\Gamma}
	=\langle u,\widehat{\mathsf T}\overline v\rangle_{\widehat\Gamma}
	=(u,\widehat{\mathsf J}\widehat{\mathsf K}\widehat{\mathsf J}v)_0,
\end{equation}
for all $u,v\in\mathscr C^{0,\alpha}(\widehat\Gamma)$.
Conjugation is isometric on $\mathscr C^{0,\alpha}(\widehat\Gamma)$.
Hence $\widehat{\mathsf J}\widehat{\mathsf K}\widehat{\mathsf J}$
is bounded and complex linear, with the same operator norm as
$\widehat{\mathsf K}$.
We apply Lax's lemma (Theorem~\ref{thm:c11-operator-extension}) with
$X=Y=\mathscr C^{0,\alpha}(\widehat\Gamma)$, the inner product
$(\cdot,\cdot)_0$, $A=\widehat{\mathsf K}$, and
$B=\widehat{\mathsf J}\widehat{\mathsf K}\widehat{\mathsf J}$,
to obtain
\begin{equation}
	\|\widehat{\mathsf K}u\|_0
	\leq\|\widehat{\mathsf K}\|_
	 {\mathscr L(\mathscr C^{0,\alpha}(\widehat\Gamma),
	             \mathscr C^{0,\alpha}(\widehat\Gamma))}\|u\|_0,
	\quad \forall u\in\mathscr C^{0,\alpha}(\widehat\Gamma).
\end{equation}
Since $\widehat{\mathsf T}u=\widehat{\mathsf V}_0\widehat{\mathsf K}u$
for all $u\in\mathscr C^{0,\alpha}(\widehat\Gamma)$, boundedness of
$\widehat{\mathsf V}_0:H^{-\frac12}(\widehat\Gamma)\to
H^{\frac12}(\widehat\Gamma)$ and norm equivalence give
\begin{equation}
\label{eq:c11-original-operator-sobolev-bound}
	\begin{aligned}
	\|\widehat{\mathsf T}u\|_{H^{\frac12}(\widehat\Gamma)}
	&=\|\widehat{\mathsf V}_0\widehat{\mathsf K}u\|_
	    {H^{\frac12}(\widehat\Gamma)}\\
	&\leq C\|\widehat{\mathsf K}u\|_0
	 \leq C\|\widehat{\mathsf K}\|_
	 {\mathscr L(\mathscr C^{0,\alpha}(\widehat\Gamma),
	             \mathscr C^{0,\alpha}(\widehat\Gamma))}\|u\|_0\\
	&\leq C(\widehat\Gamma,\alpha)\|\widehat{\mathsf T}\|_
	 {\mathscr L(\mathscr C^{0,\alpha}(\widehat\Gamma),
	             \mathscr C^{1,\alpha}(\widehat\Gamma))}
	 \|u\|_{H^{-\frac12}(\widehat\Gamma)}.
	\end{aligned}
\end{equation}
The last inequality also uses boundedness of
$\widehat{\mathsf V}_0^{-1}:
\mathscr C^{1,\alpha}(\widehat\Gamma)\to
\mathscr C^{0,\alpha}(\widehat\Gamma)$.
The constant in the last line of
\eqref{eq:c11-original-operator-sobolev-bound} reads
\begin{equation}
\label{eq:c11-sobolev-extension-constant}
\begin{aligned}
	C(\widehat\Gamma,\alpha)\coloneqq
	&
	\|\widehat{\mathsf V}_0\|^{\frac32}_{\mathscr L(H^{-\frac12}(\widehat\Gamma),H^{\frac12}(\widehat\Gamma))}
	\|\widehat{\mathsf V}_0^{-1}\|^{\frac12}_{\mathscr L(H^{\frac12}(\widehat\Gamma),H^{-\frac12}(\widehat\Gamma))}
	\|\widehat{\mathsf V}_0^{-1}\|_
	{\mathscr L(\mathscr C^{1,\alpha}(\widehat\Gamma),
     \mathscr C^{0,\alpha}(\widehat\Gamma))},
\end{aligned}
\end{equation}
which only depends on $\widehat\Gamma$ and $\alpha$.

Since $\widehat{\mathsf V}_0$ is the Laplace single layer BIO on the
fixed reference boundary, this constant depends only on $\alpha$ and
$\widehat\Gamma$, and is independent of $\widehat{\mathsf T}$ and $u$.
Therefore, \eqref{eq:c11-original-operator-sobolev-bound} bounds the original
operator $\widehat{\mathsf T}$ in the fractional Sobolev norms,
with the constant given in \eqref{eq:c11-sobolev-extension-constant}.

By density of $\mathscr C^{0,\alpha}(\widehat\Gamma)$ in
$H^{-\frac12}(\widehat\Gamma)$,
estimate~\eqref{eq:c11-original-operator-sobolev-bound} gives a unique
bounded extension
$\mathcal E(\widehat{\mathsf T}):H^{-\frac12}(\widehat\Gamma)
\to H^{\frac12}(\widehat\Gamma)$
satisfying~\eqref{eq:c11-sobolev-extension-agreement}, with the same norm bound.
Uniqueness also implies complex linearity of $\mathcal E$.
\end{proof}

\begin{remark}
A similar extension argument was also used in \cite[Lemma~1.2 and Theorem~5.15]{Henriquez2021} to transfer
bounds for periodic Helmholtz BIOs in two dimensions from spaces of
continuous functions to $L^2$.
\end{remark}

\section{Shape Holomorphy of the Helmholtz Calder\'on Projector}
\label{sec:c11-helmholtz-holomorphy}

We establish shape holomorphy of the  
Helmholtz Calder\'on projectors with respect to $\mathscr C^{1,1}$ boundary
transformations. We transport the four scalar BIOs to the fixed reference
boundary $\widehat\Gamma$ and prove their holomorphic dependence upon domain
deformations as bounded linear operators acting on fractional Sobolev
spaces. We use Lemma~\ref{lem:c11-sobolev-extension} to pass from H\"older
to fractional Sobolev spaces. We obtain shape holomorphy of the adjoint
double-layer operator by duality, and of the hypersingular operator
through its representation in terms of the single-layer operator and
surface differential operators. The single layer result also
provides the main ingredient for the Maxwell EFIE analysis on closed and
open surfaces in Sections~\ref{sec:c11-efie} and~\ref{sec:c11-screens},
respectively.

\subsection{The Helmholtz Calder\'on Projector}
\label{sec:c11-scalar-calderon-setting}

For an open set $\Omega\subset\mathbb R^3$, we write
$u\in H^1_{\mathrm{loc}}(\overline\Omega)$ when
$\chi u\in H^1(\Omega)$ for every $\chi\in\mathscr C_c^\infty(\mathbb R^3)$.
We use the same cutoff convention for $L^2_{\mathrm{loc}}(\overline\Omega)$;
the closure includes regularity up to the boundary.

Throughout this work, the wavenumber $\kappa>0$ is fixed.
The outgoing fundamental solution of
$-\Delta-\kappa^2$ is
\begin{equation}
	G^{(\kappa)}(\mathbf{x},\mathbf{y})
	\coloneqq 
	\frac{\exp(\ii\kappa\|\mathbf{x}-\mathbf{y}\|)}
	     {4\pi\|\mathbf{x}-\mathbf{y}\|},
	\quad \mathbf{x},\mathbf{y}\in\mathbb R^3
	\text{ with }\mathbf{x}\ne\mathbf{y}.
\end{equation}

For each $\br\in\mathfrak T$, the interior Dirichlet and Neumann traces
are bounded linear operators
$\gamma_0^-:H^1(\text{D}_\br)\to H^{\frac12}(\Gamma_\br)$ and
$\gamma_1^-:H^1(\text{D}_\br;\Delta)\to H^{-\frac12}(\Gamma_\br)$,
respectively. The Dirichlet trace is surjective. For their definitions
and mapping properties, see \cite[Theorem~3.37 and Lemma~4.3]{McLean2000}.

The exterior traces on $\Gamma_\br$ are continuous linear maps
\begin{equation}
	\gamma_0^+:H^1_{\mathrm{loc}}(\mathbb R^3\setminus\text{D}_\br)
	\to H^{\frac12}(\Gamma_\br)
	\quad\text{and}\quad
	\gamma_1^+:H^1_{\mathrm{loc}}(\mathbb R^3\setminus\text{D}_\br;\Delta)
	\to H^{-\frac12}(\Gamma_\br).
\end{equation}
The exterior Dirichlet trace is also surjective. Both Neumann traces
use the normal $\boldsymbol n_\br$
pointing out of $\text{D}_\br$.

Using the kernel $G^{(\kappa)}$ above, we define the scalar potentials by
\begin{equation}
	\begin{aligned}
	(\mathsf S_{\Gamma_\br}^{(\kappa)}\mu)(\mathbf x)
	&\coloneqq\int_{\Gamma_\br} G^{(\kappa)}(\mathbf x,\mathbf y)\mu(\mathbf y)\,\dd s_{\mathbf y},\\
	(\mathsf D_{\Gamma_\br}^{(\kappa)}u)(\mathbf x)
	&\coloneqq\int_{\Gamma_\br} \partial_{\boldsymbol n_\br(\mathbf y)}
	    G^{(\kappa)}(\mathbf x,\mathbf y)u(\mathbf y)\,\dd s_{\mathbf y},
	\end{aligned}
	\label{eq:c11-scalar-potentials}
\end{equation}
for all $\mathbf x\notin\Gamma_\br$,
$\mu\in H^{-\frac12}(\Gamma_\br)$, and $u\in H^{\frac12}(\Gamma_\br)$.
For nonsmooth densities, these integrals denote the corresponding
Sobolev duality pairings. The potentials solve $-\Delta u-\kappa^2u=0$
in $\mathbb R^3\setminus\Gamma_\br$. Their traces define the four scalar BIOs
\begin{equation}
	\begin{aligned}
	\mathsf V_{\Gamma_\br}^{(\kappa)}
	&=\gamma_0^\pm\mathsf S_{\Gamma_\br}^{(\kappa)}
		:H^{-\frac12}(\Gamma_\br)\to H^{\frac12}(\Gamma_\br),\\
	\mathsf K_{\Gamma_\br}^{(\kappa)}
	&=\gamma_0^\pm\mathsf D_{\Gamma_\br}^{(\kappa)}\mp\tfrac12\mathsf I
		:H^{\frac12}(\Gamma_\br)\to H^{\frac12}(\Gamma_\br),\\
	\mathsf K_{\Gamma_\br}^{(\kappa)\prime}
	&=\gamma_1^\pm\mathsf S_{\Gamma_\br}^{(\kappa)}\pm\tfrac12\mathsf I
		:H^{-\frac12}(\Gamma_\br)\to H^{-\frac12}(\Gamma_\br),\\
	\mathsf W_{\Gamma_\br}^{(\kappa)}
	&=-\gamma_1^\pm\mathsf D_{\Gamma_\br}^{(\kappa)}
		:H^{\frac12}(\Gamma_\br)\to H^{-\frac12}(\Gamma_\br).
	\end{aligned}
	\label{eq:c11-scalar-bios}
\end{equation}
These operators are bounded. The jump relations and the representation
formula are recalled in \cite[Chapters~6 and~7]{McLean2000}
and \cite[Sections~6.6 and~6.9]{Steinbach2008}.
The prime in $\mathsf K_{\Gamma_\br}^{(\kappa)\prime}$ denotes the
dual of $\mathsf K_{\Gamma_\br}^{(\kappa)}$ at the same wavenumber.

If $u\in H^1(\text{D}_\br)$ solves the Helmholtz equation weakly in $\text{D}_\br$,
its representation is
$u=\mathsf S_{\Gamma_\br}^{(\kappa)}\gamma_1^-u
-\mathsf D_{\Gamma_\br}^{(\kappa)}\gamma_0^-u$ in $\text{D}_\br$.
Taking the two traces gives the interior Calder\'on projector
\begin{equation}
	\mathsf C_{\Gamma_\br}^{(\kappa),-}
	\coloneqq
	\begin{pmatrix}
	\tfrac12\mathsf I-\mathsf K_{\Gamma_\br}^{(\kappa)}
		&\mathsf V_{\Gamma_\br}^{(\kappa)}\\
	\mathsf W_{\Gamma_\br}^{(\kappa)}
		&\tfrac12\mathsf I+\mathsf K_{\Gamma_\br}^{(\kappa)\prime}
	\end{pmatrix}
	\quad\text{and}\quad
	\mathsf C_{\Gamma_\br}^{(\kappa),+}
	\coloneqq\mathsf I-\mathsf C_{\Gamma_\br}^{(\kappa),-}.
	\label{eq:c11-physical-helmholtz-projectors}
\end{equation}
They are complementary bounded projectors on
$H^{\frac12}(\Gamma_\br)\times H^{-\frac12}(\Gamma_\br)$.
Their ranges consist, respectively, of the Cauchy data of interior
Helmholtz solutions in $H^1(\text{D}_\br)$ and exterior Helmholtz solutions
that satisfy the Sommerfeld radiation condition and belong to
$H^1_{\mathrm{loc}}(\mathbb R^3\setminus\text{D}_\br)$.

For each $\br\in\mathfrak T$, we transport the BIOs to the reference
boundary $\widehat\Gamma$ using the pullback $\tau_\br$ introduced in
Subsection~\ref{sec:c11-pullbacks}. The transported operators are
\begin{equation}
	\begin{aligned}
	\widehat{\mathsf V}_\br^{(\kappa)}
	&\coloneqq\tau_\br\mathsf V_{\Gamma_\br}^{(\kappa)}\tau_\br^{-1}
	:H^{-\frac12}(\widehat\Gamma)\to H^{\frac12}(\widehat\Gamma),\\
	\widehat{\mathsf K}_\br^{(\kappa)}
	&\coloneqq\tau_\br\mathsf K_{\Gamma_\br}^{(\kappa)}\tau_\br^{-1}
	:H^{\frac12}(\widehat\Gamma)\to H^{\frac12}(\widehat\Gamma),\\
	\widehat{\mathsf K}_\br^{(\kappa)\prime}
	&\coloneqq\tau_\br\mathsf K_{\Gamma_\br}^{(\kappa)\prime}\tau_\br^{-1}
	:H^{-\frac12}(\widehat\Gamma)\to H^{-\frac12}(\widehat\Gamma),\\
	\widehat{\mathsf W}_\br^{(\kappa)}
	&\coloneqq\tau_\br\mathsf W_{\Gamma_\br}^{(\kappa)}\tau_\br^{-1}
	:H^{\frac12}(\widehat\Gamma)\to H^{-\frac12}(\widehat\Gamma).
	\end{aligned}
	\label{eq:c11-transported-scalar-bios}
\end{equation}
The dual operator on the reference boundary is
\begin{equation}
	\bigl(\widehat{\mathsf K}_\br^{(\kappa)}\bigr)'
	=J_{\text S,\br}
	  \widehat{\mathsf K}_\br^{(\kappa)\prime}
	  J_{\text S,\br}^{-1}.
	\label{eq:c11-scalar-pullback-duality}
\end{equation}
Here $\bigl(\widehat{\mathsf K}_\br^{(\kappa)}\bigr)'$ denotes the
dual with respect to
$\langle\cdot,\cdot\rangle_{\widehat\Gamma}$.
The Jacobian factors follow from the surface change of variables.
On $H^{\frac12}(\widehat\Gamma)\times H^{-\frac12}(\widehat\Gamma)$,
the transported projectors are
\begin{equation}
	\widehat{\mathsf C}_\br^{(\kappa),-}
	=\begin{pmatrix}
	\tfrac12\mathsf I_{\widehat\Gamma}-\widehat{\mathsf K}_\br^{(\kappa)}
		&\widehat{\mathsf V}_\br^{(\kappa)}\\
	\widehat{\mathsf W}_\br^{(\kappa)}
		&\tfrac12\mathsf I_{\widehat\Gamma}+\widehat{\mathsf K}_\br^{(\kappa)\prime}
	\end{pmatrix}
	\quad\text{and}\quad
	\widehat{\mathsf C}_\br^{(\kappa),+}
	=\mathsf I_{\widehat\Gamma}-\widehat{\mathsf C}_\br^{(\kappa),-}.
	\label{eq:c11-transported-helmholtz-projectors}
\end{equation}
For each $\br \in \mathfrak{T}$, the transported and physical projectors
satisfy
\begin{equation}
	\widehat{\mathsf C}_\br^{(\kappa),\pm}
	=\begin{pmatrix}\tau_\br&0\\0&\tau_\br\end{pmatrix}
	\mathsf C_{\Gamma_\br}^{(\kappa),\pm}
\begin{pmatrix}\tau_\br^{-1}&0\\0&\tau_\br^{-1}\end{pmatrix}.
\end{equation}

\begin{theorem}[Shape holomorphy of the Helmholtz Calder\'on projector]
\label{thm:c11-helmholtz-calderon-holomorphy}
Let $\mathfrak T$ satisfy Assumption~\ref{ass:c11-deformations}.
Then there exists $\delta>0$ such that the maps
\begin{equation}
	\widehat{\mathcal C}^{(\kappa),\pm}
	:\mathfrak T\to
	\mathscr L\Bigl(
	H^{\frac12}(\widehat\Gamma)\times H^{-\frac12}(\widehat\Gamma),
	H^{\frac12}(\widehat\Gamma)\times H^{-\frac12}(\widehat\Gamma)
	\Bigr):\br\mapsto\widehat{\mathsf C}_\br^{(\kappa),\pm}
\end{equation}
admit bounded holomorphic extensions to $\mathfrak T_\delta$.
\end{theorem}

\begin{proof}
Propositions~\ref{prop:c11-singlelayer-holomorphy},
\ref{prop:c11-double-layer-holomorphy},
\ref{prop:c11-adjoint-double-layer-holomorphy}, and
\ref{prop:c11-hypersingular-holomorphy},
together with the block formula
\eqref{eq:c11-transported-helmholtz-projectors}, show that the maps
$\mathfrak T\ni\br\mapsto\widehat{\mathsf C}_\br^{(\kappa),\pm}$,
with values in
$\mathscr L(H^{\frac12}(\widehat\Gamma)\times H^{-\frac12}(\widehat\Gamma),
H^{\frac12}(\widehat\Gamma)\times H^{-\frac12}(\widehat\Gamma))$,
admit bounded holomorphic extensions to $\mathfrak T_\delta$.
\end{proof}

\subsection{The Single Layer BIO}
For each $\br\in\mathfrak T$ and $0<\alpha<1$,
the Helmholtz single layer BIO in~\eqref{eq:c11-scalar-bios}
has, for $\mu\in\mathscr C^{0,\alpha}(\Gamma_\br)$, the explicit expression
\begin{equation}
	(\mathsf V_{\Gamma_\br}^{(\kappa)}\mu)(\mathbf{x})
	=
	\int_{\Gamma_\br}G^{(\kappa)}(\mathbf{x},\mathbf{y})
	                  \mu(\mathbf{y})\,\dd s_{\mathbf{y}},
	\quad \forall\mathbf{x}\in\Gamma_\br,
\end{equation}
and $\mathsf V_{\Gamma_\br}^{(\kappa)} \in \mathscr{L}\left(\mathscr C^{0,\alpha}(\Gamma_\br),
\mathscr C^{1,\alpha}(\Gamma_\br)\right)$, see, e.g.,
\cite[Lemma~2.5, item~(i)]{DallaRivaLuzziniMusolino2022}.

By Lemma~\ref{lem:c11-pullback-isomorphisms}, its pullback
$\widehat{\mathsf V}_\br^{(\kappa)}
=\tau_\br\mathsf V_{\Gamma_\br}^{(\kappa)}\tau_\br^{-1}$
has the same H\"older and Sobolev mapping properties on
$\widehat\Gamma$.
To apply Lemma~\ref{lem:c11-sobolev-extension}, we introduce
the auxiliary single layer BIO
$\widehat{\mathsf U}_\br^{(\kappa)}
\coloneqq\widehat{\mathsf V}_\br^{(\kappa)}J_{\text S,\br}^{-1}$,
where $J_{\text S,\br}^{-1}$ acts by multiplication on the density.
Since $J_{\text S,\br}^{-1}\in\mathscr C^{0,1}(\widehat\Gamma)$,
this multiplication is bounded on
$\mathscr C^{0,\alpha}(\widehat\Gamma)$ and
$H^{-\frac12}(\widehat\Gamma)$ by
Lemma~\ref{lem:c11-jacobian-multiplier}.
Lemma~\ref{lem:c11-jacobian-multiplier} therefore gives the same mapping
properties as $\widehat{\mathsf V}_\br^{(\kappa)}$.
For each $\br \in\mathfrak{T}$, the surface change of variables
\eqref{eq:c11-surface-change-of-variables} gives for 
$\mu\in\mathscr C^{0,\alpha}(\widehat\Gamma)$
\begin{equation}
	(\widehat{\mathsf U}_\br^{(\kappa)}\mu)(\widehat{\mathbf{x}})
	=
	\int_{\widehat\Gamma}
	G^{(\kappa)}\bigl(\br(\widehat{\mathbf{x}}),
	             \br(\widehat{\mathbf{y}})\bigr)
	\mu(\widehat{\mathbf{y}})\,\dd s_{\widehat{\mathbf{y}}},
	\quad
	\forall\widehat{\mathbf{x}}\in\widehat\Gamma.
	\label{eq:c11-symmetric-layer}
\end{equation}
By \eqref{eq:c11-symmetric-layer},
$\widehat{\mathsf U}_\br^{(\kappa)}$ is symmetric with respect to the
bilinear pairing
$\langle\cdot,\cdot\rangle_{\widehat\Gamma}$.
We use this symmetry and Lemma~\ref{lem:c11-sobolev-extension}
to transfer holomorphy to the fractional Sobolev operator norm.

\begin{proposition}[Shape holomorphy of the single layer BIO]
\label{prop:c11-singlelayer-holomorphy}
Let $\mathfrak T$ satisfy Assumption~\ref{ass:c11-deformations}.
Then there exists $\delta>0$ such that the maps
\begin{equation}
	\mathfrak T\to
	\mathscr L\bigl(H^{-\frac{1}{2}}(\widehat\Gamma),H^{\frac{1}{2}}(\widehat\Gamma)\bigr):
	\br\mapsto\widehat{\mathsf V}_\br^{(\kappa)},
\end{equation}
and
\begin{equation}
	\mathfrak T\to
	\mathscr L\bigl(H^{-\frac{1}{2}}(\widehat\Gamma),H^{\frac{1}{2}}(\widehat\Gamma)\bigr):
	\br\mapsto\widehat{\mathsf U}_\br^{(\kappa)},
\end{equation}
admit bounded holomorphic extensions
to $\mathfrak T_\delta$.
\end{proposition}

\begin{proof}
Fix $0<\alpha<1$.
By \cite[Lemma~2.5, item~(i)]{DallaRivaLuzziniMusolino2022}, the map
$\mathfrak T\to\mathscr L(\mathscr C^{0,\alpha}(\widehat\Gamma),
                         \mathscr C^{1,\alpha}(\widehat\Gamma)):$
$\br\mapsto\widehat{\mathsf V}_\br^{(\kappa)}$ is real analytic.
For each $\br_0\in\mathfrak T$, real analyticity gives a holomorphic
extension to a complex $\mathscr C^{1,\alpha}$ neighborhood of $\br_0$
\cite[Sections~5 and~13]{Mujica1986}.
Since the inclusion
$\mathscr C^{1,1}(\widehat\Gamma;\mathbb C^3)
\hookrightarrow\mathscr C^{1,\alpha}(\widehat\Gamma;\mathbb C^3)$
is continuous, this extension restricts to a complex
$\mathscr C^{1,1}$ neighborhood of $\br_0$.
Choose this neighborhood to be a ball and shrink it so that the extension
is bounded there and every real transformation in the ball maps
$\widehat\Gamma$ onto the boundary of a $\mathscr C^{1,1}$ domain.

If two such balls overlap, their intersection is connected and
its real part contains a relatively open set. Both extensions coincide
there with the original BIO. Since
$\mathscr C^{1,1}(\widehat\Gamma;\mathbb C^3)$ is the complexification of
$\mathscr C^{1,1}(\widehat\Gamma;\mathbb R^3)$, uniqueness of analytic
continuation in Banach spaces
\cite[Theorem~5.34]{DudleyNorvaisa2011} shows that the extensions coincide
throughout the overlap.
A finite covering argument and the compactness of $\mathfrak T$
show that there exists $\delta>0$ such that the map
$\mathfrak T\ni\br\mapsto\widehat{\mathsf V}_\br^{(\kappa)}
\in\mathscr L(\mathscr C^{0,\alpha}(\widehat\Gamma),
\mathscr C^{1,\alpha}(\widehat\Gamma))$
admits a bounded holomorphic extension to $\mathfrak T_\delta$.
By Lemmas~\ref{lem:c11-jacobian-multiplier}
and~\ref{lem:c11-geometric-holomorphy}, after decreasing $\delta$ if
necessary, the formula
$\widehat{\mathsf U}_\br^{(\kappa)}
=\widehat{\mathsf V}_\br^{(\kappa)}J_{\text S,\br}^{-1}$
gives a bounded holomorphic extension to $\mathfrak T_\delta$ of the map
$\mathfrak T\ni\br\mapsto\widehat{\mathsf U}_\br^{(\kappa)}
\in\mathscr L(\mathscr C^{0,\alpha}(\widehat\Gamma),
\mathscr C^{1,\alpha}(\widehat\Gamma))$.
Both extensions are complex Fr\'echet differentiable on $\mathfrak T_\delta$.
We use the same notation for these extensions.

For each $\br\in\mathfrak T_\delta$, define the bounded bilinear form
\[
 \mathsf b_\br(u,v)\coloneqq\int_{\widehat\Gamma}
 (\widehat{\mathsf U}_\br^{(\kappa)}u)v\,\dd s,
 \quad u,v\in\mathscr C^{0,\alpha}(\widehat\Gamma).
\]
The map $\br\mapsto\mathsf b_\br$ is holomorphic with values in the
Banach space of bounded bilinear forms on
$\mathscr C^{0,\alpha}(\widehat\Gamma)\times
\mathscr C^{0,\alpha}(\widehat\Gamma)$.
Its antisymmetric part, i.e., the bilinear form
$(u,v)\mapsto\mathsf b_\br(u,v)-\mathsf b_\br(v,u)$,
depends holomorphically on $\br\in\mathfrak T_\delta$ in the same Banach space.
For real $\br\in\mathfrak T_\delta$, the surface change of variables
gives~\eqref{eq:c11-symmetric-layer}, so the antisymmetric part vanishes.
Each ball $B(\br_0,\delta)$, $\br_0\in\mathfrak T$, is contained in
$\mathfrak T_\delta$, is connected, and has a nonempty relatively
open real part. By uniqueness of analytic continuation for Banach-valued
holomorphic maps \cite[Theorem~5.34]{DudleyNorvaisa2011}, the
antisymmetric part vanishes on each ball. After decreasing $\delta$ if
necessary, these balls cover $\mathfrak T_\delta$.
Thus $\mathsf b_\br$ is symmetric and
$\widehat{\mathsf U}_\br^{(\kappa)}$ satisfies~\eqref{eq:symmetry_Lax}
in Lemma~\ref{lem:c11-sobolev-extension} for every
$\br\in\mathfrak T_\delta$.

By Lemma~\ref{lem:c11-sobolev-extension}, item~(ii), the extension
operation $\mathcal E$ is bounded and complex linear on the space of
H\"older operators satisfying its hypothesis~\eqref{eq:symmetry_Lax}.
Let $\br\in\mathfrak T_\delta$ and let
$\boldsymbol h\in\mathscr C^{1,1}(\widehat\Gamma;\mathbb C^3)$ with
$\norm{\boldsymbol h}_{\mathscr C^{1,1}(\widehat\Gamma;\mathbb C^3)}\to0$ and
$\br+\boldsymbol h\in\mathfrak T_\delta$.
Since $\br\mapsto\mathsf b_\br$ takes values in the closed subspace of
symmetric bilinear forms, its Fr\'echet derivative is symmetric in its
two arguments. Thus the operator
$\mathrm D_\br\widehat{\mathsf U}_\br^{(\kappa)}[\boldsymbol h]$
satisfies the hypothesis~\eqref{eq:symmetry_Lax} of
Lemma~\ref{lem:c11-sobolev-extension}.
Since this hypothesis is linear in the operator, it also holds for the
Taylor remainder
\begin{equation}
 \mathsf R_\br(\boldsymbol h)
 \coloneqq
 \widehat{\mathsf U}_{\br+\boldsymbol h}^{(\kappa)}
 -\widehat{\mathsf U}_\br^{(\kappa)}
 -\mathrm D_\br\widehat{\mathsf U}_\br^{(\kappa)}[\boldsymbol h].
\end{equation}
The extension estimate in Lemma~\ref{lem:c11-sobolev-extension},
item~(ii), with $C(\widehat\Gamma,\alpha)$ given
in~\eqref{eq:c11-sobolev-extension-constant}, gives
\begin{equation}
 \begin{aligned}
 \|\mathcal E(\mathsf R_\br(\boldsymbol h))\|_
 {\mathscr L(H^{-\frac12}(\widehat\Gamma),H^{\frac12}(\widehat\Gamma))}
 &\leq C(\widehat\Gamma,\alpha)\|\mathsf R_\br(\boldsymbol h)\|_
 {\mathscr L(\mathscr C^{0,\alpha}(\widehat\Gamma),
             \mathscr C^{1,\alpha}(\widehat\Gamma))}\\
 &=o\bigl(\norm{\boldsymbol h}_
 {\mathscr C^{1,1}(\widehat\Gamma;\mathbb C^3)}\bigr).
 \end{aligned}
\end{equation}
Moreover, the map
$\boldsymbol h\mapsto\mathcal E(
\mathrm D_\br\widehat{\mathsf U}_\br^{(\kappa)}[\boldsymbol h])$
is bounded and complex linear. Hence the map
$\mathfrak T_\delta\to
\mathscr L(H^{-\frac12}(\widehat\Gamma),H^{\frac12}(\widehat\Gamma)):$
$\br\mapsto\mathcal E(\widehat{\mathsf U}_\br^{(\kappa)})$ is complex
Fr\'echet differentiable, with
\begin{equation}
 \mathrm D_\br\bigl[\mathcal E(
 \widehat{\mathsf U}_\br^{(\kappa)})\bigr][\boldsymbol h]
 =\mathcal E\bigl(
 \mathrm D_\br\widehat{\mathsf U}_\br^{(\kappa)}[\boldsymbol h]\bigr).
\end{equation}
It is therefore bounded and holomorphic on $\mathfrak T_\delta$.

For each $\br\in\mathfrak T_\delta$,
Lemma~\ref{lem:c11-sobolev-extension}, item~(iii), gives
\begin{equation}
	\mathcal E(\widehat{\mathsf U}_\br^{(\kappa)})u
	=\widehat{\mathsf U}_\br^{(\kappa)}u,
	\quad\forall u\in\mathscr C^{0,\alpha}(\widehat\Gamma).
\end{equation}
Since $J_{\text S,\br}u\in\mathscr C^{0,\alpha}(\widehat\Gamma)$,
the same identity gives
$\mathcal E(\widehat{\mathsf U}_\br^{(\kappa)})J_{\text S,\br}u
=\widehat{\mathsf U}_\br^{(\kappa)}J_{\text S,\br}u
=\widehat{\mathsf V}_\br^{(\kappa)}u$.
Hence both extensions to the Sobolev spaces agree with their H\"older extensions on
$\mathscr C^{0,\alpha}(\widehat\Gamma)$ for each $\br\in\mathfrak T_\delta$.
For real $\br\in\mathfrak T$, the extensions and the original transported
BIOs are bounded from $H^{-\frac12}(\widehat\Gamma)$ to
$H^{\frac12}(\widehat\Gamma)$, so density of
$\mathscr C^{0,\alpha}(\widehat\Gamma)$ in $H^{-\frac12}(\widehat\Gamma)$
identifies them on $H^{-\frac12}(\widehat\Gamma)$.
By the preceding extension and Lemma~\ref{lem:c11-geometric-holomorphy},
\[
 \mathfrak T\to
 \mathscr L\bigl(H^{-\frac12}(\widehat\Gamma),
                H^{\frac12}(\widehat\Gamma)\bigr):
 \br\mapsto\widehat{\mathsf V}_\br^{(\kappa)}
\]
admits a bounded holomorphic extension to $\mathfrak T_\delta$,
given by
$\br\mapsto\mathcal E(\widehat{\mathsf U}_\br^{(\kappa)})J_{\text S,\br}$.
\end{proof}

\begin{remark}
The scalar single layer operator denoted by $\mathsf V_{\mathbf z}$ in
\cite[Theorem~5.1]{EscapilJerez2026EFIE} corresponds to
$\widehat{\mathsf U}_\br^{(\kappa)}$ in our notation.
The cited result establishes its
mapping properties and holomorphic dependence in fractional Sobolev
spaces, i.e., as a map $H^{-\frac12}(\widehat\Gamma)\to H^{\frac12}(\widehat\Gamma)$, on the parameters of
affine-parametric families of boundary transformations in
$\mathscr C^{1,1}(\widehat\Gamma;\mathbb R^3)$.
Proposition~\ref{prop:c11-singlelayer-holomorphy} establishes
holomorphy with respect to the boundary transformation itself for
the compact families in Assumption~\ref{ass:c11-deformations},
without requiring an affine-paramteric representation of the boundary.
\end{remark}

\subsection{The Double Layer BIO}
\label{sec:c11-double-layer-holomorphy}

For each $\br\in\mathfrak T$ and $0<\alpha<1$,
the Helmholtz double layer BIO in~\eqref{eq:c11-scalar-bios}
has the explicit expression
\begin{equation}
	(\mathsf K_{\Gamma_\br}^{(\kappa)}u)(\mathbf{x})
	=
	\int_{\Gamma_\br}
	\partial_{\boldsymbol n_\br(\mathbf{y})}
	G^{(\kappa)}(\mathbf{x},\mathbf{y})u(\mathbf{y})\,\dd s_{\mathbf{y}},
	\quad\forall\mathbf{x}\in\Gamma_\br,
\end{equation}
for all $u\in\mathscr C^{1,\alpha}(\Gamma_\br)$.
The operator is bounded on $\mathscr C^{1,\alpha}(\Gamma_\br)$
\cite[Lemma~2.5, item~(ii)]{DallaRivaLuzziniMusolino2022}
and on $H^{\frac12}(\Gamma_\br)$ by~\eqref{eq:c11-scalar-bios}.
By Lemma~\ref{lem:c11-pullback-isomorphisms}, these bounds carry over
to its pullback
$\widehat{\mathsf K}_\br^{(\kappa)}
=\tau_\br\mathsf K_{\Gamma_\br}^{(\kappa)}\tau_\br^{-1}$
on the reference boundary $\widehat\Gamma$.

To apply Lemma~\ref{lem:c11-sobolev-extension}, we introduce a
symmetric block operator containing the double layer BIO and its adjoint.
We use the fixed reference Laplace single layer BIO
$\widehat{\mathsf V}_0$ introduced in the proof of
Lemma~\ref{lem:c11-sobolev-extension} to define
\begin{equation}
	\mathsf B_\br\coloneqq
	\begin{pmatrix}
	0&\widehat{\mathsf V}_0\bigl(\widehat{\mathsf K}_\br^{(\kappa)}\bigr)'\\
	\widehat{\mathsf K}_\br^{(\kappa)}\widehat{\mathsf V}_0&0
	\end{pmatrix}
	\in\mathscr L\Bigl(
	\mathscr C^{0,\alpha}(\widehat\Gamma)\times
	\mathscr C^{0,\alpha}(\widehat\Gamma),
	\mathscr C^{1,\alpha}(\widehat\Gamma)\times
	\mathscr C^{1,\alpha}(\widehat\Gamma)\Bigr).
	\label{eq:c11-double-layer-symmetric-block}
\end{equation}
For $\br\in\mathfrak T$, the symmetry of
$\widehat{\mathsf V}_0$ yields
$\bigl(\widehat{\mathsf K}_\br^{(\kappa)}
\widehat{\mathsf V}_0\bigr)'
=\widehat{\mathsf V}_0\bigl(\widehat{\mathsf K}_\br^{(\kappa)}\bigr)'$.
Hence the off-diagonal entries are dual to one another, and $\mathsf B_\br$
is symmetric with respect to the component-wise bilinear pairing
$\langle\cdot,\cdot\rangle_{\widehat\Gamma}$.

\Needspace{9\baselineskip}
\begin{proposition}[Shape holomorphy of the double layer BIO]
\label{prop:c11-double-layer-holomorphy}
Let $\mathfrak T$ satisfy Assumption~\ref{ass:c11-deformations}.
Then there exists $\delta>0$ such that the map
\begin{equation}
	\mathfrak T\to
	\mathscr L\bigl(H^{\frac12}(\widehat\Gamma),H^{\frac12}(\widehat\Gamma)\bigr):
	\br\mapsto\widehat{\mathsf K}_\br^{(\kappa)},
\end{equation}
admits a bounded holomorphic extension to $\mathfrak T_\delta$.
\end{proposition}

\begin{proof}
Fix $0<\alpha<1$.
By \cite[Lemma~2.5, items~(ii) and~(iii)]{DallaRivaLuzziniMusolino2022},
the maps
$\mathfrak T\ni\br\mapsto\widehat{\mathsf K}_\br^{(\kappa)}
\in\mathscr L(\mathscr C^{1,\alpha}(\widehat\Gamma),
              \mathscr C^{1,\alpha}(\widehat\Gamma))$ and
$\mathfrak T\ni\br\mapsto\widehat{\mathsf K}_\br^{(\kappa)\prime}
\in\mathscr L(\mathscr C^{0,\alpha}(\widehat\Gamma),
              \mathscr C^{0,\alpha}(\widehat\Gamma))$ are real analytic.
Equation~\eqref{eq:c11-scalar-pullback-duality} and
Lemmas~\ref{lem:c11-jacobian-multiplier}
and~\ref{lem:c11-geometric-holomorphy} then show that
$\mathfrak T\ni\br\mapsto(\widehat{\mathsf K}_\br^{(\kappa)})'
\in\mathscr L(\mathscr C^{0,\alpha}(\widehat\Gamma),
              \mathscr C^{0,\alpha}(\widehat\Gamma))$ is real analytic.
Consequently, the map
$\mathfrak T\ni\br\mapsto\mathsf B_\br
\in\mathscr L((\mathscr C^{0,\alpha}(\widehat\Gamma))^2,
              (\mathscr C^{1,\alpha}(\widehat\Gamma))^2)$
introduced in~\eqref{eq:c11-double-layer-symmetric-block} is real analytic.
Arguing as in the proof of
Proposition~\ref{prop:c11-singlelayer-holomorphy}, these maps admit
bounded holomorphic extensions to $\mathfrak T_\delta$ for some
$\delta>0$ in their respective H\"older operator spaces.
We use the same notation for the extensions.

For each $\br\in\mathfrak T_\delta$, define the bounded bilinear form
\[
 \mathsf b_\br(u,v)\coloneqq\int_{\widehat\Gamma}
 (\mathsf B_\br u)\cdot v\,\dd s,
 \quad u,v\in\mathscr C^{0,\alpha}(\widehat\Gamma)\times
 \mathscr C^{0,\alpha}(\widehat\Gamma),
\]
where the dot product is bilinear, without complex conjugation.
The map $\br\mapsto\mathsf b_\br$ is holomorphic with values in the
Banach space of bounded bilinear forms on
$(\mathscr C^{0,\alpha}(\widehat\Gamma))^2\times
(\mathscr C^{0,\alpha}(\widehat\Gamma))^2$.
For real $\br\in\mathfrak T_\delta$, the block
formula~\eqref{eq:c11-double-layer-symmetric-block} and symmetry of
$\widehat{\mathsf V}_0$ imply that
$\mathsf b_\br(u,v)=\mathsf b_\br(v,u)$ for all
$u,v\in(\mathscr C^{0,\alpha}(\widehat\Gamma))^2$.
Its antisymmetric part, i.e., the bilinear form
$(u,v)\mapsto\mathsf b_\br(u,v)-\mathsf b_\br(v,u)$,
depends holomorphically on $\br\in\mathfrak T_\delta$ in the same Banach space and is zero
for real $\br\in\mathfrak T_\delta$.
Uniqueness of analytic continuation in Banach spaces
\cite[Theorem~5.34]{DudleyNorvaisa2011} therefore shows that this
bilinear form vanishes
on each ball
$B(\br_0,\delta)$, $\br_0\in\mathfrak T$, and hence on
$\mathfrak T_\delta$.
Thus $\mathsf B_\br$ satisfies the hypothesis~\eqref{eq:symmetry_Lax}
of Lemma~\ref{lem:c11-sobolev-extension}, with the product of scalar
functions replaced by the bilinear dot product, for every
$\br\in\mathfrak T_\delta$.

To extend $\mathsf B_\br$ to the Sobolev spaces, we apply the proof of
Lemma~\ref{lem:c11-sobolev-extension} to two component fields, with
$\widehat{\mathsf V}_0$ acting diagonally.
By Lemma~\ref{lem:c11-sobolev-extension}, item~(ii), $\mathcal E$
is a bounded, complex linear map from the operators in
$\mathscr L((\mathscr C^{0,\alpha}(\widehat\Gamma))^2,
            (\mathscr C^{1,\alpha}(\widehat\Gamma))^2)$
satisfying the symmetry hypothesis~\eqref{eq:symmetry_Lax}, with the
bilinear dot product, to
$\mathscr L((H^{-\frac12}(\widehat\Gamma))^2,
            (H^{\frac12}(\widehat\Gamma))^2)$.
Let $\br\in\mathfrak T_\delta$ and let
$\boldsymbol h\in\mathscr C^{1,1}(\widehat\Gamma;\mathbb C^3)$ with
$\norm{\boldsymbol h}_{\mathscr C^{1,1}(\widehat\Gamma;\mathbb C^3)}\to0$ and
$\br+\boldsymbol h\in\mathfrak T_\delta$.
Here $\mathrm D_\br\mathsf B_\br[\boldsymbol h]$ denotes the complex
Fr\'echet derivative of the map $\br\mapsto\mathsf B_\br$ at $\br$,
evaluated in the direction $\boldsymbol h$.
Differentiating the symmetry of $\mathsf b_\br$ with respect to $\br$
shows that $\mathrm D_\br\mathsf B_\br[\boldsymbol h]$ satisfies
the hypothesis~\eqref{eq:symmetry_Lax} of
Lemma~\ref{lem:c11-sobolev-extension}, with the bilinear dot product.
This condition is preserved by linear combinations, so it also holds
for
\begin{equation}
 \mathsf R_\br(\boldsymbol h)
 \coloneqq\mathsf B_{\br+\boldsymbol h}-\mathsf B_\br
 -\mathrm D_\br\mathsf B_\br[\boldsymbol h].
\end{equation}
Complex Fr\'echet differentiability of $\br\mapsto\mathsf B_\br$
gives the remainder estimate in the H\"older operator norm.
Lemma~\ref{lem:c11-sobolev-extension}, item~(ii), transfers this
estimate to the Sobolev operator norm:
\begin{equation}
 \begin{aligned}
 \|\mathcal E(\mathsf R_\br(\boldsymbol h))\|_
 {\mathscr L((H^{-\frac12}(\widehat\Gamma))^2,
             (H^{\frac12}(\widehat\Gamma))^2)}
 &\leq C(\widehat\Gamma,\alpha)\|\mathsf R_\br(\boldsymbol h)\|_
 {\mathscr L((\mathscr C^{0,\alpha}(\widehat\Gamma))^2,
             (\mathscr C^{1,\alpha}(\widehat\Gamma))^2)}\\
 &=o\bigl(\norm{\boldsymbol h}_
 {\mathscr C^{1,1}(\widehat\Gamma;\mathbb C^3)}\bigr).
 \end{aligned}
\end{equation}
Boundedness and complex linearity of $\mathcal E$ imply that the map
$\boldsymbol h\mapsto\mathcal E(\mathrm D_\br\mathsf B_\br[\boldsymbol h])$
is bounded and complex linear in $\boldsymbol h$.
Together with the preceding remainder estimate, this proves that the map
$\mathfrak T_\delta\to
\mathscr L((H^{-\frac12}(\widehat\Gamma))^2,
           (H^{\frac12}(\widehat\Gamma))^2):
\br\mapsto\mathcal E(\mathsf B_\br)$
is complex Fr\'echet differentiable, with derivative
\begin{equation}
 \mathrm D_\br[\mathcal E(\mathsf B_\br)][\boldsymbol h]
 =\mathcal E(\mathrm D_\br\mathsf B_\br[\boldsymbol h]).
\end{equation}
It is therefore bounded and holomorphic on $\mathfrak T_\delta$.

For each $\br\in\mathfrak T_\delta$, define the extension to the Sobolev spaces by
\begin{equation}
 \widehat{\mathsf K}_\br^{(\kappa)}
 \coloneqq(\mathcal E(\mathsf B_\br))_{21}\widehat{\mathsf V}_0^{-1},
\end{equation}
where $(\cdot)_{21}$ denotes the block in row $2$ and column $1$, and
$\widehat{\mathsf V}_0^{-1}\in\mathscr L(
H^{\frac12}(\widehat\Gamma),H^{-\frac12}(\widehat\Gamma))$.
Taking the block in row $2$ and column $1$ and composing it with the
fixed operator $\widehat{\mathsf V}_0^{-1}$ are bounded, complex linear
operations. Hence the map
$\mathfrak T_\delta\to
\mathscr L(H^{\frac12}(\widehat\Gamma),H^{\frac12}(\widehat\Gamma)):
\br\mapsto\widehat{\mathsf K}_\br^{(\kappa)}$
is complex Fr\'echet differentiable,
with derivative
\begin{equation}
 \mathrm D_\br\widehat{\mathsf K}_\br^{(\kappa)}[\boldsymbol h]
 =\bigl(\mathcal E(\mathrm D_\br\mathsf B_\br[\boldsymbol h])\bigr)_{21}
 \widehat{\mathsf V}_0^{-1}.
\end{equation}
For each $\br\in\mathfrak T_\delta$ and
$u\in\mathscr C^{1,\alpha}(\widehat\Gamma)$, we have
$\widehat{\mathsf V}_0^{-1}u\in\mathscr C^{0,\alpha}(\widehat\Gamma)$.
Applying Lemma~\ref{lem:c11-sobolev-extension}, item~(iii), to
$(\widehat{\mathsf V}_0^{-1}u,0)$ gives
\[
 (\mathcal E(\mathsf B_\br))_{21}\widehat{\mathsf V}_0^{-1}u
 =(\mathsf B_\br)_{21}\widehat{\mathsf V}_0^{-1}u
 =\widehat{\mathsf K}_\br^{(\kappa)}u.
\]
Hence the extension to the Sobolev spaces agrees with the H\"older extension on
$\mathscr C^{1,\alpha}(\widehat\Gamma)$ for each $\br\in\mathfrak T_\delta$.
For $\br\in\mathfrak T$, both this extension and the original BIO
are bounded on $H^{\frac12}(\widehat\Gamma)$ and agree on
$\mathscr C^{1,\alpha}(\widehat\Gamma)$.
Density of
$\mathscr C^{1,\alpha}(\widehat\Gamma)$ in
$H^{\frac12}(\widehat\Gamma)$ therefore shows that they coincide on
$H^{\frac12}(\widehat\Gamma)$.
The estimate in Lemma~\ref{lem:c11-sobolev-extension}, item~(ii),
and boundedness of $\widehat{\mathsf V}_0^{-1}$ give uniform boundedness
on $\mathfrak T_\delta$.
Hence the map
$\mathfrak T\ni\br\mapsto\widehat{\mathsf K}_\br^{(\kappa)}
\in\mathscr L(H^{\frac12}(\widehat\Gamma),H^{\frac12}(\widehat\Gamma))$
admits a bounded holomorphic extension to $\mathfrak T_\delta$.
\end{proof}

\subsection{The Adjoint Double Layer BIO}
\label{sec:c11-adjoint-double-layer-holomorphy}

For each $\br\in\mathfrak T$ and $0<\alpha<1$,
the Helmholtz adjoint double layer BIO in~\eqref{eq:c11-scalar-bios}
has the explicit expression
\begin{equation}
	(\mathsf K_{\Gamma_\br}^{(\kappa)\prime}\mu)(\mathbf{x})
	=
	\int_{\Gamma_\br}
	\partial_{\boldsymbol n_\br(\mathbf{x})}
	G^{(\kappa)}(\mathbf{x},\mathbf{y})\mu(\mathbf{y})\,\dd s_{\mathbf{y}},
	\quad\forall\mathbf{x}\in\Gamma_\br,
\end{equation}
for all $\mu\in\mathscr C^{0,\alpha}(\Gamma_\br)$.
It is bounded on $\mathscr C^{0,\alpha}(\Gamma_\br)$
\cite[Lemma~2.5, item~(iii)]{DallaRivaLuzziniMusolino2022}.
Its extension to the Sobolev spaces is bounded on $H^{-\frac12}(\Gamma_\br)$
by~\eqref{eq:c11-scalar-bios}.
By Lemma~\ref{lem:c11-pullback-isomorphisms}, its pullback
$\widehat{\mathsf K}_\br^{(\kappa)\prime}
=\tau_\br\mathsf K_{\Gamma_\br}^{(\kappa)\prime}\tau_\br^{-1}$
has the same H\"older and Sobolev mapping properties on $\widehat\Gamma$.

Its shape holomorphy follows from
Proposition~\ref{prop:c11-double-layer-holomorphy} and the duality
relation~\eqref{eq:c11-scalar-pullback-duality}.

\begin{proposition}[Shape holomorphy of the adjoint double layer BIO]
\label{prop:c11-adjoint-double-layer-holomorphy}
Let $\mathfrak T$ satisfy Assumption~\ref{ass:c11-deformations}.
Then there exists $\delta>0$ such that the map
\begin{equation}
	\widehat{\mathcal K}^{(\kappa)\prime}:
	\mathfrak T\to
	\mathscr L\bigl(H^{-\frac12}(\widehat\Gamma),H^{-\frac12}(\widehat\Gamma)\bigr):
	\br\mapsto\widehat{\mathsf K}_\br^{(\kappa)\prime},
\end{equation}
admits a bounded holomorphic extension to $\mathfrak T_\delta$.
\end{proposition}

\begin{proof}
There exists $\delta>0$ such that the bounded holomorphic extensions in
Proposition~\ref{prop:c11-double-layer-holomorphy} and
Lemma~\ref{lem:c11-geometric-holomorphy} are defined on $\mathfrak T_\delta$.
For each $\br\in\mathfrak T_\delta$, we define
\begin{equation}\label{eq:adjoint_DL_ref}
	\widehat{\mathsf K}_\br^{(\kappa)\prime}
	\coloneqq J_{\text S,\br}^{-1}
	\bigl(\widehat{\mathsf K}_\br^{(\kappa)}\bigr)'
	J_{\text S,\br},
	\quad\forall\br\in\mathfrak T_\delta.
\end{equation}
Here the prime outside the parentheses denotes the reference dual,
which is defined using the bilinear pairing, without complex conjugation.
Because the pairing is bilinear, taking the dual is a bounded complex
linear operation from
$\mathscr L(H^{\frac12}(\widehat\Gamma),H^{\frac12}(\widehat\Gamma))$
to $\mathscr L(H^{-\frac12}(\widehat\Gamma),H^{-\frac12}(\widehat\Gamma))$.
By Lemmas~\ref{lem:c11-jacobian-multiplier}
and~\ref{lem:c11-geometric-holomorphy}, the maps
$\mathfrak T\to
\mathscr L(H^{-\frac12}(\widehat\Gamma),H^{-\frac12}(\widehat\Gamma)):
\br\mapsto\bigl(u\mapsto J_{\text S,\br}u\bigr)$ and
$\mathfrak T\to
\mathscr L(H^{-\frac12}(\widehat\Gamma),H^{-\frac12}(\widehat\Gamma)):
\br\mapsto\bigl(u\mapsto J_{\text S,\br}^{-1}u\bigr)$
admit bounded holomorphic extensions to $\mathfrak T_\delta$.
Since operator composition is continuous and multilinear,
\eqref{eq:adjoint_DL_ref} defines the bounded holomorphic map
$\mathfrak T_\delta\to
\mathscr L(H^{-\frac12}(\widehat\Gamma),H^{-\frac12}(\widehat\Gamma)):
\br\mapsto\widehat{\mathsf K}_\br^{(\kappa)\prime}$.
For each $\br\in\mathfrak T_\delta$, $\mu\in H^{-\frac12}(\widehat\Gamma)$, and
$v\in H^{\frac12}(\widehat\Gamma)$, its defining pairing is
\[
 \begin{aligned}
 \langle\widehat{\mathsf K}_\br^{(\kappa)\prime}\mu,v\rangle_{\widehat\Gamma}
 &=\langle J_{\text S,\br}\mu,
 \widehat{\mathsf K}_\br^{(\kappa)}
 (J_{\text S,\br}^{-1}v)\rangle_{\widehat\Gamma}.
 \end{aligned}
\]
For $\br\in\mathfrak T$, relation~\eqref{eq:c11-scalar-pullback-duality}
identifies the right-hand side with the pairing of the transported
adjoint double layer BIO with $v$.
Since this holds for every $\mu$ and $v$ in these Sobolev spaces,
the map
$\mathfrak T\ni\br\mapsto\widehat{\mathsf K}_\br^{(\kappa)\prime}
\in\mathscr L(H^{-\frac12}(\widehat\Gamma),H^{-\frac12}(\widehat\Gamma))$
admits a bounded holomorphic extension to $\mathfrak T_\delta$ defined
by~\eqref{eq:adjoint_DL_ref}.
\end{proof}

\subsection{The Hypersingular BIO}
\label{sec:c11-hypersingular-holomorphy}

For each $\br\in\mathfrak T$, we define the vector surface curl by
$\curl_{\Gamma_\br} u\coloneqq\boldsymbol n_\br\times\nabla_{\Gamma_\br} u$.
It is bounded from $H^{\frac12}(\Gamma_\br)$ to
$\boldsymbol H^{-\frac12}(\Gamma_\br)$.
Here $\boldsymbol H^s(\Gamma_\br)\coloneqq (H^s(\Gamma_\br))^3$.
The single layer BIOs $\mathsf V_{\Gamma_\br}^{(\kappa)}$ and
$\widehat{\mathsf U}_\br^{(\kappa)}$ act componentwise on vector fields.
Vector duality pairings are sums of the three scalar pairings.
The regularized representation of the hypersingular BIO is
\begin{equation}
    \begin{aligned}
	\langle\mathsf W_{\Gamma_\br}^{(\kappa)}u,v\rangle_{\Gamma_\br}
	&=\langle\curl_{\Gamma_\br} v,
	   \mathsf V_{\Gamma_\br}^{(\kappa)}\curl_{\Gamma_\br} u\rangle_{\Gamma_\br}\\
	&\quad-\kappa^2\langle\boldsymbol n_\br v,
	   \mathsf V_{\Gamma_\br}^{(\kappa)}(\boldsymbol n_\br u)\rangle_{\Gamma_\br},
    \end{aligned}
	\label{eq:c11-hypersingular-regularization}
\end{equation}
for all $u,v\in H^{\frac12}(\Gamma_\br)$.
See \cite[Section~6.9, Eq.~(6.68)]{Steinbach2008}.

We write $F_\br^{-\top}\coloneqq(F_\br^{-1})^\top$.
For each real-valued admissible boundary transformation, the tangential
chain rule gives
\begin{equation}
	\tau_\br\curl_{\Gamma_\br}\tau_\br^{-1}\widehat u
	=\boldsymbol N_\br\times
	 F_\br^{-\top}\nabla_{\widehat\Gamma}\widehat u
	\quad\text{and}\quad
	\tau_\br(\boldsymbol n_\br\tau_\br^{-1}\widehat u)
	=\boldsymbol N_\br\widehat u.
	\label{eq:c11-surface-curl-transport}
\end{equation}
Both identities hold in $\boldsymbol H^{-\frac12}(\widehat\Gamma)$.

For each $\br\in\mathfrak T$, we define the auxiliary maps
$\mathsf C_\br,\mathsf N_\br:
H^{\frac12}(\widehat\Gamma)\to\boldsymbol H^{-\frac12}(\widehat\Gamma)$ by
\begin{equation}
	\mathsf C_\br\widehat u\coloneqq
	J_{\text S,\br}\boldsymbol N_\br\times
	F_\br^{-\top}\nabla_{\widehat\Gamma}\widehat u
	\quad\text{and}\quad
	\mathsf N_\br\widehat u\coloneqq
	J_{\text S,\br}\boldsymbol N_\br\widehat u.
\end{equation}
Both maps are bounded by the surface gradient and Lipschitz
multiplier estimates. The distributional pullback
\eqref{eq:c11-distribution-pullbacks} and the change of variables in
\eqref{eq:c11-hypersingular-regularization} give
\begingroup
\small
\begin{equation}
	\langle\widehat{\mathsf W}_\br^{(\kappa)}\widehat u,\widehat v\rangle_{\widehat\Gamma}
	=
	\langle\mathsf C_\br(J_{\text S,\br}^{-1}\widehat v),
	  \widehat{\mathsf U}_\br^{(\kappa)}
	   \mathsf C_\br\widehat u\rangle_{\widehat\Gamma}
	-\kappa^2\langle\mathsf N_\br(J_{\text S,\br}^{-1}\widehat v),
	  \widehat{\mathsf U}_\br^{(\kappa)}
	   \mathsf N_\br\widehat u\rangle_{\widehat\Gamma},
	\label{eq:c11-complex-hypersingular-definition}
\end{equation}
\endgroup
for all $\widehat u,\widehat v\in H^{\frac12}(\widehat\Gamma)$.
The factor $J_{\text S,\br}^{-1}$ multiplying $\widehat v$ comes from
the negative-order pullback $\tau_\br$ in
$\widehat{\mathsf W}_\br^{(\kappa)}
=\tau_\br\mathsf W_{\Gamma_\br}^{(\kappa)}\tau_\br^{-1}$,
as defined by~\eqref{eq:c11-distribution-pullbacks}.

\begin{proposition}[Shape holomorphy of the hypersingular BIO]
\label{prop:c11-hypersingular-holomorphy}
Let $\mathfrak T$ satisfy Assumption~\ref{ass:c11-deformations}.
Then there exists $\delta>0$ such that the map
\begin{equation}
	\widehat{\mathcal W}^{(\kappa)}:
	\mathfrak T\to
	\mathscr L\bigl(H^{\frac12}(\widehat\Gamma),H^{-\frac12}(\widehat\Gamma)\bigr):
	\br\mapsto\widehat{\mathsf W}_\br^{(\kappa)},
\end{equation}
admits a bounded holomorphic extension to $\mathfrak T_\delta$.
\end{proposition}

\begin{proof}
Let $\delta>0$ be sufficiently small so that the extensions in
Proposition~\ref{prop:c11-singlelayer-holomorphy} and
Lemma~\ref{lem:c11-geometric-holomorphy} are defined on $\mathfrak T_\delta$.
Lemma~\ref{lem:c11-geometric-holomorphy} and the multiplier estimate
in Appendix~\ref{sec:proof_c11-jacobian-multiplier} show that the maps
$\mathfrak T\to
\mathscr L(H^{\frac12}(\widehat\Gamma),
           \boldsymbol H^{-\frac12}(\widehat\Gamma)):
\br\mapsto\mathsf C_\br$ and
$\mathfrak T\to
\mathscr L(H^{\frac12}(\widehat\Gamma),
           \boldsymbol H^{-\frac12}(\widehat\Gamma)):
\br\mapsto\mathsf N_\br$
admit bounded holomorphic extensions to $\mathfrak T_\delta$.
The map $\mathfrak T\to
\mathscr L(H^{\frac12}(\widehat\Gamma),H^{\frac12}(\widehat\Gamma)):
\br\mapsto\bigl(\widehat v\mapsto J_{\text S,\br}^{-1}\widehat v\bigr)$
admits a bounded holomorphic extension to $\mathfrak T_\delta$ as well.
For each $\br\in\mathfrak T_\delta$, we define
$\widehat{\mathsf W}_\br^{(\kappa)}$ by
\eqref{eq:c11-complex-hypersingular-definition}.
Since composition and the pairings are bounded multilinear operations,
the right-hand side defines a bounded holomorphic map with values in
the space of bilinear forms on
$H^{\frac12}(\widehat\Gamma)\times H^{\frac12}(\widehat\Gamma)$.
The identification of bounded bilinear forms with operators therefore shows
that the map
$\mathfrak T\ni\br\mapsto\widehat{\mathsf W}_\br^{(\kappa)}
\in\mathscr L(H^{\frac12}(\widehat\Gamma),H^{-\frac12}(\widehat\Gamma))$
admits a bounded holomorphic extension to $\mathfrak T_\delta$.
\end{proof}

\section{Shape Holomorphy of the Maxwell EFIE}
\label{sec:c11-efie}
We establish shape holomorphy of the EFIE operator on closed surfaces,
its solution map, and the electric far-field. We first introduce the
Maxwell trace spaces and the surface Piola transformation. The principal
ingredient is shape holomorphy of the single layer BIO, established in
Proposition~\ref{prop:c11-singlelayer-holomorphy}.

\subsection{Maxwell Trace Spaces and the Surface Piola Transformation}
\label{sec:c11-maxwell-traces}
For the following definitions and properties of Maxwell trace spaces,
we follow \cite[Sections~2--5]{BCS2002}.
Let $\text{D}\subset\IR^3$ be a bounded, simply connected $\mathscr C^{1,1}$ domain with
connected boundary $\Gamma$
and outward unit normal $\boldsymbol n$.
Since $\boldsymbol n\in\mathscr C^{0,1}(\Gamma;\mathbb R^3)$, the
multiplier estimate proved in Appendix~\ref{app:c11-jacobian}, applied
on $\Gamma$, shows that multiplication by its components is bounded on
$H^{-\frac12}(\Gamma)$.
Hence $\boldsymbol j\cdot\boldsymbol n\in H^{-\frac12}(\Gamma)$ for every
$\boldsymbol j\in\boldsymbol H^{-\frac12}(\Gamma)$.
For tangential fields, surface divergence is understood distributionally:
\begin{equation}
	\langle\operatorname{div}_\Gamma\boldsymbol j,\phi\rangle_{\Gamma}
	=-\langle\boldsymbol j,\nabla_\Gamma\phi\rangle_{\Gamma}.
\end{equation}
This identity holds for all $\phi\in\mathscr C^{1,1}(\Gamma)$.
The scalar surface curl is defined as
$\operatorname{curl}_\Gamma\boldsymbol\eta
\coloneqq\operatorname{div}_\Gamma(\boldsymbol\eta\times\boldsymbol n)$.
We define the tangential divergence and curl spaces by
\begin{equation}
	\begin{aligned}
	\boldsymbol H^{-\frac12}(\operatorname{div}_\Gamma,\Gamma)
	&\coloneqq\bigl\{\boldsymbol j\in\boldsymbol H^{-\frac12}(\Gamma):
	\boldsymbol j\cdot\boldsymbol n=0,\quad
	\operatorname{div}_\Gamma\boldsymbol j\in H^{-\frac12}(\Gamma)\bigr\},\\
	\boldsymbol H^{-\frac12}(\operatorname{curl}_\Gamma,\Gamma)
	&\coloneqq\bigl\{\boldsymbol\eta\in\boldsymbol H^{-\frac12}(\Gamma):
	\boldsymbol\eta\cdot\boldsymbol n=0,\quad
	\operatorname{curl}_\Gamma\boldsymbol\eta\in H^{-\frac12}(\Gamma)\bigr\}.
	\end{aligned}
\end{equation}
For the sake of simplicity, we write
$X(\Gamma)\coloneqq\boldsymbol H^{-\frac12}(\operatorname{div}_\Gamma,\Gamma)$
and
$Y(\Gamma)\coloneqq\boldsymbol H^{-\frac12}(\operatorname{curl}_\Gamma,\Gamma)$,
and equip these spaces with the norms
\begin{equation}
	\begin{aligned}
	\|\boldsymbol j\|_{X(\Gamma)}^2
	&\coloneqq\|\boldsymbol j\|_{\boldsymbol H^{-\frac12}(\Gamma)}^2
	+\|\operatorname{div}_\Gamma\boldsymbol j\|_{H^{-\frac12}(\Gamma)}^2,
	\quad \boldsymbol j\in X(\Gamma)
	\\
	\|\boldsymbol\eta\|_{Y(\Gamma)}^2
	&\coloneqq\|\boldsymbol\eta\|_{\boldsymbol H^{-\frac12}(\Gamma)}^2
	+\|\operatorname{curl}_\Gamma\boldsymbol\eta\|_{H^{-\frac12}(\Gamma)}^2,
	\quad
	\boldsymbol\eta\in Y(\Gamma).
	\end{aligned}
	\label{eq:c11-maxwell-trace-norms}
\end{equation}

We use the standard space $\boldsymbol H(\curl,\text{D})$, equipped
with its norm. For an open set $\Omega\subset\mathbb R^3$, the notation
$\boldsymbol H_{\mathrm{loc}}(\curl,\overline\Omega)$ uses the cutoff
convention of Subsection~\ref{sec:c11-scalar-calderon-setting}. For all
$\boldsymbol u\in\mathscr C^1(\overline{\text{D}};\mathbb C^3)$,
the tangential traces are
\begin{equation}
	\gamma_\times\boldsymbol u
	\coloneqq \boldsymbol n\times\boldsymbol u|_\Gamma
	\quad\text{and}\quad
	\gamma_T\boldsymbol u
	\coloneqq \boldsymbol n\times(\boldsymbol u|_\Gamma\times\boldsymbol n).
\end{equation}
For $\mathscr C^{1,1}$ surfaces, these componentwise definitions agree with
the intrinsic Maxwell trace spaces.
The traces extend to bounded surjections
$\gamma_\times:\boldsymbol H(\curl,\text{D})\to X(\Gamma)$ and
$\gamma_T:\boldsymbol H(\curl,\text{D})\to Y(\Gamma)$.
As is customary, the space $Y(\Gamma)$ is identified with $X(\Gamma)'$ through the
continuous duality pairing
$
	\langle\cdot,\cdot\rangle_{Y(\Gamma)\times X(\Gamma)}:
	Y(\Gamma)\times X(\Gamma)\to\mathbb C,
$
given by
\begin{equation}
	\dual{\boldsymbol\eta}{\boldsymbol j}{Y(\Gamma)\times X(\Gamma)}
	=\int_\Gamma\boldsymbol\eta(\mathbf{x})\cdot
	             \boldsymbol j(\mathbf{x})\,\dd s_{\mathbf{x}}
	\label{eq:c11-maxwell-pairing}
\end{equation}
for all $\boldsymbol\eta\in Y(\Gamma)\cap\mathscr C^{0,1}(\Gamma;\mathbb C^3)$
and $\boldsymbol j\in X(\Gamma)\cap\mathscr C^{0,1}(\Gamma;\mathbb C^3)$.
Under this identification, the dual norm inherited from $X(\Gamma)'$
is equivalent to the norm on $Y(\Gamma)$ defined in
\eqref{eq:c11-maxwell-trace-norms}.
For each $\br\in\mathfrak T$, we set
\begin{equation}
	\widehat X\coloneqq X(\widehat\Gamma)
	\quad\text{and}\quad
	X_\br\coloneqq X(\Gamma_\br).
\end{equation}

We use the classical Piola transformation
\cite[Section~III.1.3]{BrezziFortin1991} in its surface form
\cite[Section~4.2]{BonitoDemlowLicht2020}.
For each $\br\in\mathfrak T$, the surface Piola pullback
$\mathcal P_\br:X_\br\to\widehat X$ and its inverse are given by
\begin{equation}
	\begin{aligned}
	\mathcal P_\br\boldsymbol j
	=J_{\text S,\br}F_\br^{-1}(\tau_\br\boldsymbol j)
	\quad
	\text{and}
	\quad
	\tau_\br(\mathcal P_\br^{-1}\widehat{\boldsymbol j})
	=J_{\text S,\br}^{-1}F_\br\widehat{\boldsymbol j},
    \quad
    \forall \boldsymbol j\in X_\br,\widehat{\boldsymbol j}\in\widehat X.
	\end{aligned}
	\label{eq:c11-surface-piola}
\end{equation}

Recall that $F_\br=D_{\widehat\Gamma}\br$ is the tangential differential
defined in~\eqref{eq:c11-tangential-differential}.
Its tangential inverse satisfies $F_\br^{-1}F_\br=\widehat P$ and
$F_\br^{-1}\boldsymbol N_\br=0$, as in
\eqref{eq:c11-tangential-inverse-convention}.
The Piola pullback and its inverse are bounded isomorphisms,
and, for each $\br \in \mathfrak{T}$, $\mathcal{P}_\br$ satisfies
\begin{equation}
	\divGh(\mathcal P_\br\boldsymbol j)
	=
	J_{\text S,\br}\tau_\br
	    (\operatorname{div}_{\Gamma_\br}\boldsymbol j).
	\label{eq:c11-piola-divergence}
\end{equation}
The definition of $F_\br^{-1}$ gives
$\widehat{\boldsymbol n}\cdot\mathcal P_\br\boldsymbol j=0$
in $H^{-\frac12}(\widehat\Gamma)$ for all $\boldsymbol j\in X_\br$.
For all $\boldsymbol j\in X_\br$ and
$\widehat\phi\in\mathscr C^{1,1}(\widehat\Gamma)$, the tangential chain rule
and~\eqref{eq:c11-distribution-pullbacks} yield
\begin{equation*}
 \begin{aligned}
 \langle\operatorname{div}_{\widehat\Gamma}
       (\mathcal P_\br\boldsymbol j),\widehat\phi\rangle_{\widehat\Gamma}
 &=\langle\operatorname{div}_{\Gamma_\br}\boldsymbol j,
          \tau_\br^{-1}\widehat\phi\rangle_{\Gamma_\br}\\
 &=\langle J_{\text S,\br}\tau_\br
          \operatorname{div}_{\Gamma_\br}\boldsymbol j,
          \widehat\phi\rangle_{\widehat\Gamma}.
 \end{aligned}
\end{equation*}
Thus~\eqref{eq:c11-piola-divergence} holds in
$H^{-\frac12}(\widehat\Gamma)$.
Moreover, $F_\br$ and $F_\br^{-1}$ have entries in
$\mathscr C^{0,1}(\widehat\Gamma)$. Applying the multiplier estimate of
Appendix~\ref{app:c11-jacobian} to each entry shows that multiplication
by either matrix field is bounded on
$\boldsymbol H^{-\frac12}(\widehat\Gamma)$.
The divergence transformation~\eqref{eq:c11-piola-divergence} and
these matrix multiplier bounds, together with
Lemmas~\ref{lem:c11-jacobian-multiplier}
and~\ref{lem:c11-pullback-isomorphisms}, with the norms defined
in~\eqref{eq:c11-maxwell-trace-norms}, yield
\begin{equation}
 \|\mathcal P_\br\boldsymbol j\|_{\widehat X}
 \leq C_\br\|\boldsymbol j\|_{X_\br},
 \quad \forall\boldsymbol j\in X_\br.
\end{equation}
For the inverse formula in~\eqref{eq:c11-surface-piola}, multiplication
by $J_{\text S,\br}^{-1}F_\br$ and the inverse pullback bound
$\mathcal P_\br^{-1}\widehat{\boldsymbol j}$ in
$\boldsymbol H^{-\frac12}(\Gamma_\br)$. Applying the divergence calculation
to the inverse surface diffeomorphism
$\br^{-1}:\Gamma_\br\to\widehat\Gamma$ bounds its surface divergence in
$H^{-\frac12}(\Gamma_\br)$.
Consequently,
\begin{equation}
 \|\mathcal P_\br^{-1}\widehat{\boldsymbol j}\|_{X_\br}
 \leq \widehat C_\br\|\widehat{\boldsymbol j}\|_{\widehat X},
 \quad\forall\widehat{\boldsymbol j}\in\widehat X.
\end{equation}
Since the fields are tangential, the formulas in
\eqref{eq:c11-surface-piola} show that $\mathcal P_\br$ and
$\mathcal P_\br^{-1}$ are mutually inverse.

\subsection{EFIE Formulation of Scattering by a Perfectly Conducting Obstacle}
\label{sec:c11-closed-formulation}
In the following, the superscripts $+$ and $-$ denote exterior and interior
traces, respectively, on either $\Gamma_\br$ or $\widehat\Gamma$.
On each surface, both traces use the normal pointing out of the
corresponding bounded domain.
For smooth fields, the crossed tangential trace is
$\gamma_\times^\pm\boldsymbol E
=\boldsymbol n_\br\times\boldsymbol E|_{\Gamma_\br}$.

We consider scattering by a perfectly conducting obstacle. Fix $\br\in\mathfrak T$ and set
$\text{D}_\br^{\text{c}}\coloneqq\mathbb R^3\setminus\overline{\text{D}_\br}$.
Let $\boldsymbol E^{\mathrm{inc}}\in\mathscr C^\infty(\mathbb R^3;\mathbb C^3)$
be an incident electric field satisfying
\begin{equation}
	\curl\curl\boldsymbol E^{\mathrm{inc}}
	-\kappa^2\boldsymbol E^{\mathrm{inc}}=0
	\quad\text{in }\mathbb R^3.
\end{equation}
For example, consider the plane wave
\begin{equation}
	\boldsymbol E^{\mathrm{inc}}(\mathbf{x})
	\coloneqq\boldsymbol p\exp(\ii\kappa\boldsymbol d\cdot\mathbf{x}),
	\quad\forall\mathbf{x}\in\mathbb R^3,
	\label{eq:c11-plane-wave}
\end{equation}
where $\boldsymbol d\in\mathbb R^3$ with $\|\boldsymbol d\|=1$
is the propagation direction and
$\boldsymbol p\in\mathbb C^3\setminus\{0\}$ is the polarization,
with $\boldsymbol p\cdot\boldsymbol d=0$.
For each $\br\in\mathfrak T$, we seek a scattered field
$\boldsymbol E_\br^{\mathrm{sc}}$
in $\boldsymbol H_{\mathrm{loc}}(\curl,\overline{\text{D}_\br^{\mathrm c}})$ such that
\begin{equation}
	\begin{aligned}
	\curl\curl\boldsymbol E_\br^{\mathrm{sc}}-\kappa^2\boldsymbol E_\br^{\mathrm{sc}}
	&=0 &&\text{in }\text{D}_\br^{\text{c}},\\
	\gamma_T^+\boldsymbol E_\br^{\mathrm{sc}}
	&=-\gamma_T\boldsymbol E^{\mathrm{inc}}
	&&\text{in }Y(\Gamma_\br)=X_\br'.
	\end{aligned}
	\label{eq:c11-closed-scattering}
\end{equation}
The outgoing field satisfies the
Silver--M\"uller radiation condition
\begin{equation}
	\lim_{R\to\infty}R\bigl(
	\curl\boldsymbol E_\br^{\mathrm{sc}}(R\boldsymbol\theta)\times\boldsymbol\theta
	-\ii\kappa\boldsymbol E_\br^{\mathrm{sc}}(R\boldsymbol\theta)\bigr)=0,
	\quad \forall\boldsymbol\theta\in\mathbb S^2,
	\label{eq:c11-radiation}
\end{equation}
uniformly in $\boldsymbol\theta$, where
$\mathbb S^2\coloneqq\{\boldsymbol\theta\in\mathbb R^3:\|\boldsymbol\theta\|=1\}$.
The boundary condition implies that the total field
$\boldsymbol E_\br^{\mathrm{tot}}\coloneqq
\boldsymbol E^{\mathrm{inc}}+\boldsymbol E_\br^{\mathrm{sc}}$
has zero tangential electric trace. This exterior boundary value problem
is uniquely solvable for all $\kappa>0$
\cite[Sections~1 and~7]{BuffaHiptmair2003}.

For each $\br\in\mathfrak T$ and density
$\boldsymbol j\in X_\br$, the electric potential generated by
$\boldsymbol j$ is defined using the scalar Helmholtz single layer
potential in~\eqref{eq:c11-scalar-potentials} by
\begin{equation}
	\mathsf{SL}_{\Gamma_\br}^{(\kappa)}\boldsymbol j
	\coloneqq\ii\kappa\mathsf S_{\Gamma_\br}^{(\kappa)}\boldsymbol j
	-\frac{1}{\ii\kappa}\nabla\mathsf S_{\Gamma_\br}^{(\kappa)}
	    \operatorname{div}_{\Gamma_\br}\boldsymbol j.
	\label{eq:c11-electric-potential}
\end{equation}
Here $\mathsf S_{\Gamma_\br}^{(\kappa)}$ acts componentwise on the density
$\boldsymbol j$ and as a scalar operator on its surface divergence.
The gradient term ensures that the resulting electric field
$\boldsymbol E=\mathsf{SL}_{\Gamma_\br}^{(\kappa)}\boldsymbol j$ is
divergence-free away from $\Gamma_\br$. This field satisfies
$\curl\curl\boldsymbol E-\kappa^2\boldsymbol E=0$ in $\mathbb R^3\setminus\Gamma_\br$
and the Silver--M\"uller radiation condition in
$\mathbb R^3\setminus\overline{\text{D}_\br}$.

To obtain the EFIE, we take the
tangential trace of this field. The maps
$\mathsf{SL}_{\Gamma_\br}^{(\kappa)}:X_\br\to
\boldsymbol H(\curl,\text{D}_\br)$ and
$\mathsf{SL}_{\Gamma_\br}^{(\kappa)}:X_\br\to
\boldsymbol H_{\mathrm{loc}}(\curl,\mathbb R^3\setminus\text{D}_\br)$
are continuous and linear \cite[Theorem~5]{BuffaHiptmair2003}, so both tangential traces are
well defined in $Y(\Gamma_\br)$. Moreover, the interior and exterior traces
$\gamma_T^\pm\mathsf{SL}_{\Gamma_\br}^{(\kappa)}\boldsymbol j$
agree \cite[Theorem~7]{BuffaHiptmair2003}. Thus either trace gives the
same EFIE operator under the identification $Y(\Gamma_\br)=X_\br'$.

For each $\br\in\mathfrak T$, we define the bilinear form for the EFIE
$\mathsf a_{\Gamma_\br}^{(\kappa)}:X_\br\times X_\br\to\mathbb C$ by
\begin{equation}
	\begin{aligned}
	\mathsf a_{\Gamma_\br}^{(\kappa)}(\boldsymbol j,\boldsymbol v)
	&\coloneqq
	\ii\kappa\,
	\dual{\boldsymbol v}
	{\mathsf V_{\Gamma_\br}^{(\kappa)}\boldsymbol j}
	{\Gamma_\br}
	\\
	&\quad+
	\frac{1}{\ii\kappa}
	\dual{\operatorname{div}_{\Gamma_\br}\boldsymbol v}
	{\mathsf V_{\Gamma_\br}^{(\kappa)}\operatorname{div}_{\Gamma_\br}\boldsymbol j}
	{\Gamma_\br},
	\quad
	\forall \boldsymbol j,\boldsymbol v\in X_\br.
	\end{aligned}
	\label{eq:c11-physical-efie-form}
\end{equation}
The single layer BIO $\mathsf V_{\Gamma_\br}^{(\kappa)}$ acts componentwise in the
first term and on the scalar surface divergence in the second.
The single layer mapping property and the norms defined
in~\eqref{eq:c11-maxwell-trace-norms} imply that this bilinear form is
bounded.
For each $\br\in\mathfrak T$, it therefore induces the operator
$\mathsf A_{\Gamma_\br}^{(\kappa)}\in
 \mathscr L(X_\br,X_\br')$ by
\begin{equation}
	\dual{\mathsf A_{\Gamma_\br}^{(\kappa)}\boldsymbol j}{\boldsymbol v}
	{X_\br'\times X_\br}
	\coloneqq 
	\mathsf a_{\Gamma_\br}^{(\kappa)}(\boldsymbol j,\boldsymbol v),
	\quad \forall\boldsymbol j,\boldsymbol v\in X_\br.
\end{equation}
Surface integration by parts gives
\begin{equation}
	\mathsf A_{\Gamma_\br}^{(\kappa)}
	=\gamma_T^\pm\mathsf{SL}_{\Gamma_\br}^{(\kappa)}
	\in\mathscr L(X_\br,X_\br').
	\label{eq:c11-efie-trace-operator}
\end{equation}

The Stratton--Chu representation \cite[Theorem~6]{BuffaHiptmair2003}
and the boundary condition in~\eqref{eq:c11-closed-scattering} give
\begin{equation}
	\boldsymbol E_\br^{\mathrm{sc}}
	=\mathsf{SL}_{\Gamma_\br}^{(\kappa)}\boldsymbol j_\br
	\quad\text{in }\text{D}_\br^{\text{c}},
	\label{eq:c11-closed-representation}
	\noeqref{eq:c11-closed-representation}
\end{equation}
where $\boldsymbol j_\br\coloneqq(\ii\kappa)^{-1}
\gamma_\times^+\curl\boldsymbol E_\br^{\mathrm{tot}}\in X_\br$.
Taking the exterior tangential trace of this identity and using
\eqref{eq:c11-efie-trace-operator} yields the EFIE on $\Gamma_\br$:
for each $\br\in\mathfrak T$, we seek
$\boldsymbol j_\br\in X_\br$ such that
\begin{equation}
	\mathsf A_{\Gamma_\br}^{(\kappa)}\boldsymbol j_\br
	=\boldsymbol f_\br,
	\quad \boldsymbol f_\br\coloneqq-\gamma_T\boldsymbol E^{\mathrm{inc}}
	\in X_\br'.
	\label{eq:c11-closed-trace-efie}
\end{equation}
Equivalently,
\begin{equation}
	\mathsf a_{\Gamma_\br}^{(\kappa)}(\boldsymbol j_\br,\boldsymbol v)
	=\dual{\boldsymbol f_\br}{\boldsymbol v}{X_\br'\times X_\br},
	\quad \forall\boldsymbol v\in X_\br.
	\label{eq:c11-closed-variational-efie}
\end{equation}
Conversely, the electric potential associated with any solution of
\eqref{eq:c11-closed-trace-efie} satisfies the exterior scattering
problem. By uniqueness of that problem, every solution of the EFIE
represents the same unique scattered field, although the solution of
the EFIE need not be unique.

\subsection{Shape Holomorphy of the EFIE}
To study shape dependence on a fixed space, we use the surface Piola
map to transfer the EFIE from $X_\br$ to $\widehat X$.
For each $\br\in\mathfrak T$, the inverse map
$\mathcal P_\br^{-1}:\widehat X\to X_\br$ transports the trial and
test fields to $\Gamma_\br$. We therefore define the bilinear form
$\widehat{\mathsf a}_\br^{(\kappa)}:
\widehat X\times\widehat X\to\mathbb C$ by
\begin{equation}
	\widehat{\mathsf a}_\br^{(\kappa)}
	      (\widehat{\boldsymbol j},\widehat{\boldsymbol v})
	\coloneqq\mathsf a_{\Gamma_\br}^{(\kappa)}
	      (\mathcal P_\br^{-1}\widehat{\boldsymbol j},
	       \mathcal P_\br^{-1}\widehat{\boldsymbol v}),
	\quad\forall\widehat{\boldsymbol j},\widehat{\boldsymbol v}\in\widehat X.
\end{equation}
The boundedness of $\mathsf a_{\Gamma_\br}^{(\kappa)}$ and
$\mathcal P_\br^{-1}$ implies that this form is bounded on
$\widehat X\times\widehat X$. Its associated operator is thus
\begin{equation}
	\widehat{\mathsf A}_\br^{(\kappa)}
	\coloneqq 
	(\mathcal P_\br^{-1})'\mathsf A_{\Gamma_\br}^{(\kappa)}\mathcal P_\br^{-1}
	\in\mathscr L(\widehat X,\widehat X').
\end{equation}

Let $\widehat{\mathsf F}_\br$ denote multiplication by the tangential
differential:
\begin{equation}
	\widehat{\mathsf F}_\br:\widehat X\to\boldsymbol H^{-\frac{1}{2}}(\widehat\Gamma):
	\widehat{\boldsymbol j}\mapsto F_\br\widehat{\boldsymbol j}.
\end{equation}
The definition \eqref{eq:c11-tangential-differential} and the
multiplier estimate in the proof of
Lemma~\ref{lem:c11-jacobian-multiplier} give
\begin{equation}
	\|\widehat{\mathsf F}_\br\|_{\mathscr L(\widehat X,
	                            \boldsymbol H^{-\frac{1}{2}}(\widehat\Gamma))}
	\leq C\|\br\|_{\mathscr C^{1,1}(\widehat\Gamma;\IR^3)}.
\end{equation}
The same formula defines the bounded complex-linear map
$\mathscr C^{1,1}(\widehat\Gamma;\mathbb C^3)\to
\mathscr L(\widehat X,\boldsymbol H^{-\frac12}(\widehat\Gamma)):
\br\mapsto\widehat{\mathsf F}_\br$.
Recall that $\mathsf V_{\Gamma_\br}^{(\kappa)}$ is the scalar Helmholtz
single layer BIO on $\Gamma_\br$ defined in~\eqref{eq:c11-scalar-bios}.
The auxiliary reference BIO $\widehat{\mathsf U}_\br^{(\kappa)}$ is
defined by the kernel formula~\eqref{eq:c11-symmetric-layer} and satisfies
$\widehat{\mathsf U}_\br^{(\kappa)}
=\widehat{\mathsf V}_\br^{(\kappa)}J_{\text S,\br}^{-1}$.
We use the same notation for its componentwise action on vector fields.
For each $\br\in\mathfrak T$, the pullback formulas
\eqref{eq:c11-distribution-pullbacks} and
$\tau_\br\mathsf V_{\Gamma_\br}^{(\kappa)}
=\widehat{\mathsf U}_\br^{(\kappa)}J_{\text S,\br}\tau_\br$ give
\begin{equation}
	\langle g,\mathsf V_{\Gamma_\br}^{(\kappa)}f\rangle_{\Gamma_\br}
	=\langle J_{\text S,\br}\tau_\br g,
	\widehat{\mathsf U}_\br^{(\kappa)}
	    (J_{\text S,\br}\tau_\br f)\rangle_{\widehat\Gamma},
	\label{eq:c11-singlelayer-pairing-transport}
\end{equation}
for all $f,g\in H^{-\frac12}(\Gamma_\br)$.
For $\widehat{\boldsymbol j},\widehat{\boldsymbol v}\in\widehat X$,
set $\boldsymbol j=\mathcal P_\br^{-1}\widehat{\boldsymbol j}$
and $\boldsymbol v=\mathcal P_\br^{-1}\widehat{\boldsymbol v}$.
Applying~\eqref{eq:c11-singlelayer-pairing-transport} to the two terms
in~\eqref{eq:c11-physical-efie-form}, componentwise in the vector term,
and using~\eqref{eq:c11-surface-piola}
and~\eqref{eq:c11-piola-divergence} gives the representation
\begin{equation}
	\begin{aligned}
	\widehat{\mathsf a}_\br^{(\kappa)}
	      (\widehat{\boldsymbol j},\widehat{\boldsymbol v})
	&=\ii\kappa\,
	\dual{\widehat{\mathsf F}_\br\widehat{\boldsymbol v}}
	{\widehat{\mathsf U}_\br^{(\kappa)}
	      \widehat{\mathsf F}_\br\widehat{\boldsymbol j}}
	{\widehat\Gamma}\\
	&\quad+\tfrac{1}{\ii\kappa}
	\dual{\divGh\widehat{\boldsymbol v}}
	{\widehat{\mathsf U}_\br^{(\kappa)}\divGh\widehat{\boldsymbol j}}
	{\widehat\Gamma}.
	\end{aligned}
	\label{eq:c11-efie-factorization}
\end{equation}

\begin{theorem}[Shape holomorphy of the EFIE]
\label{thm:c11-efie-holomorphy}
Let $\mathfrak T$ satisfy Assumption~\ref{ass:c11-deformations}.
Then there exists $\delta>0$ such that the map
\begin{equation}
	\widehat{\mathcal A}^{(\kappa)}:
	\mathfrak T\to\mathscr L(\widehat X,\widehat X'):
	\br\mapsto\widehat{\mathsf A}_\br^{(\kappa)}
\end{equation}
admits a bounded holomorphic extension to $\mathfrak T_\delta$.
\end{theorem}

\begin{proof}
By Proposition~\ref{prop:c11-singlelayer-holomorphy}, the map
$\mathfrak T\ni\br\mapsto\widehat{\mathsf U}_\br^{(\kappa)}
\in\mathscr L(H^{-\frac12}(\widehat\Gamma),H^{\frac12}(\widehat\Gamma))$
admits a bounded holomorphic extension to $\mathfrak T_\delta$.
The multiplier map defined above is entire and bounded on bounded sets.
The duality pairing between $H^{-\frac12}(\widehat\Gamma)$ and
$H^{\frac12}(\widehat\Gamma)$, and its componentwise vector extension,
are bounded complex bilinear maps. Moreover, the reference surface
divergence $\divGh:\widehat X\to H^{-\frac12}(\widehat\Gamma)$
is bounded and complex linear.
Identifying bilinear forms with their induced operators,
\eqref{eq:c11-efie-factorization} therefore shows that the map
$\mathfrak T\to\mathscr L(\widehat X,\widehat X'):
\br\mapsto\widehat{\mathsf A}_\br^{(\kappa)}$
admits a bounded holomorphic extension to $\mathfrak T_\delta$.
\end{proof}

\subsection{Shape Holomorphy of the EFIE Solution}
\label{sec:c11-closed-solution}

Unlike the exterior boundary value problem, unique solvability of
this EFIE requires exclusion of interior electric Maxwell eigenvalues
\cite[Sections~6 and~7.2]{BuffaHiptmair2003}.
Set
\begin{equation}
	\boldsymbol H_0(\curl,\text{D}_\br)\coloneqq
	\{\boldsymbol E\in\boldsymbol H(\curl,\text{D}_\br):
	\gamma_T\boldsymbol E=0\}.
\end{equation}
\begin{assumption}[Nonresonance on closed surfaces]
\label{ass:c11-nonresonance}
For each $\br\in\mathfrak T$,
$\kappa^2$ is not an interior electric Maxwell eigenvalue of $\text{D}_\br$.
That is, the only $\boldsymbol E\in\boldsymbol H_0(\curl,\text{D}_\br)$
satisfying
\begin{equation}
	\int_{\text{D}_\br}\curl\boldsymbol E(\mathbf{x})\cdot
	    \overline{\curl\boldsymbol V(\mathbf{x})}\,\dd\mathbf{x}
	=\kappa^2\int_{\text{D}_\br}\boldsymbol E(\mathbf{x})\cdot
	    \overline{\boldsymbol V(\mathbf{x})}\,\dd\mathbf{x},
	\quad \forall\boldsymbol V\in\boldsymbol H_0(\curl,\text{D}_\br),
\end{equation}
is $\boldsymbol E=0$.
\end{assumption}

Set
$\widehat{\boldsymbol f}_\br\coloneqq
(\mathcal P_\br^{-1})'\boldsymbol f_\br$ and
$\widehat{\boldsymbol j}_\br\coloneqq\mathcal P_\br\boldsymbol j_\br$.
\Needspace{8\baselineskip}
\begin{corollary}[Shape holomorphy of the EFIE solution on closed surfaces]
\label{cor:c11-closed-solution-holomorphy}
Let $\mathfrak T$ satisfy
Assumptions~\ref{ass:c11-deformations} and~\ref{ass:c11-nonresonance}.
Then the following assertions hold.
\begin{itemize}
 \item[\normalfont{(i)}] There exists $\delta>0$ such that the map
 \begin{equation}
	\mathfrak T\to\mathscr L(\widehat X',\widehat X):
	\br\mapsto
	(\widehat{\mathsf A}_\br^{(\kappa)})^{-1}
 \end{equation}
 admits a bounded holomorphic extension to $\mathfrak T_\delta$.
 \item[\normalfont{(ii)}] For each $\br\in\mathfrak T$, let
 $\boldsymbol f_\br\in X_\br'$ and assume
 that $\mathfrak T\ni\br\mapsto\widehat{\boldsymbol f}_\br\in\widehat X'$
 admits a bounded holomorphic extension to $\mathfrak T_{\delta_f}$
 for some $\delta_f>0$. For each $\br\in\mathfrak T$, we seek
 $\boldsymbol j_\br\in X_\br$ such that
 \begin{equation}
 	\mathsf A_{\Gamma_\br}^{(\kappa)}\boldsymbol j_\br
	=\boldsymbol f_\br.
 \end{equation}
 This problem is well-posed and there exists $\delta\in(0,\delta_f]$
 such that the map
 \begin{equation}
	\mathfrak T\to\widehat X:\br\mapsto
	\widehat{\boldsymbol j}_\br
	=(\widehat{\mathsf A}_\br^{(\kappa)})^{-1}
	     \widehat{\boldsymbol f}_\br
 \end{equation}
 admits a bounded holomorphic extension to $\mathfrak T_\delta$.
\end{itemize}
\end{corollary}

\begin{proof}
Assumption~\ref{ass:c11-nonresonance} gives a bounded isomorphism
$\mathsf A_{\Gamma_\br}^{(\kappa)}:X_\br\to X_\br'$
\cite[Sections~6 and~7.2]{BuffaHiptmair2003}.
Since the Piola transformations are isomorphisms,
$\widehat{\mathsf A}_\br^{(\kappa)}$ is also a bounded isomorphism for each
$\br\in\mathfrak T$. Continuity of inversion and compactness of
$\mathfrak T$ imply that the inverses are uniformly bounded on
$\mathfrak T$. Let $\delta>0$ be such that the holomorphic operator
extension in Theorem~\ref{thm:c11-efie-holomorphy} is defined on
$\mathfrak T_\delta$. After decreasing $\delta$ if necessary, the perturbation estimate in
\cite[Proposition~4.20, item~(i)]{HS21} and continuity of the operator extension
imply that
$\widehat{\mathsf A}_\br^{(\kappa)}$ is invertible for every
$\br\in\mathfrak T_\delta$, with uniformly bounded inverses on
$\mathfrak T_\delta$.
Hence the map
$\mathfrak T\ni\br\mapsto(\widehat{\mathsf A}_\br^{(\kappa)})^{-1}
\in\mathscr L(\widehat X',\widehat X)$
admits a bounded holomorphic extension to $\mathfrak T_\delta$
by holomorphy of inversion
\cite[Proposition~4.20, item~(ii)]{HS21}, with
\begin{equation}
	M_\delta\coloneqq\sup_{\br\in\mathfrak T_\delta}
	\|(\widehat{\mathsf A}_\br^{(\kappa)})^{-1}\|_
	      {\mathscr L(\widehat X',\widehat X)}<\infty.
\end{equation}
For item~(ii), we choose $\delta\leq\delta_f$.
The bounded bilinear evaluation
$\mathscr L(\widehat X',\widehat X)\times\widehat X'\to\widehat X$
shows that
$\mathfrak T\ni\br\mapsto(\widehat{\mathsf A}_\br^{(\kappa)})^{-1}
\widehat{\boldsymbol f}_\br\in\widehat X$ admits a bounded holomorphic
extension to $\mathfrak T_\delta$
\cite[Proposition~4.20, item~(iii)]{HS21}, proving item~(ii), with
\begin{equation}
	\sup_{\br\in\mathfrak T_\delta}
	      \|\widehat{\boldsymbol j}_\br\|_{\widehat X}
	\leq M_\delta\sup_{\br\in\mathfrak T_\delta}
	      \|\widehat{\boldsymbol f}_\br\|_{\widehat X'}<\infty.\qedhere
\end{equation}
\end{proof}

For the plane wave~\eqref{eq:c11-plane-wave}, the transported
right-hand side $\widehat{\boldsymbol f}_\br\in\widehat X'$ satisfies
\begin{equation}
	\dual{\widehat{\boldsymbol f}_\br}{\widehat{\boldsymbol v}}
	     {\widehat X'\times\widehat X}
	=-\dual{\widehat{\mathsf F}_\br\widehat{\boldsymbol v}}
	       {\boldsymbol p\exp(\ii\kappa\boldsymbol d\cdot\br)}
	       {\widehat\Gamma},
	\quad\forall\widehat{\boldsymbol v}\in\widehat X.
	\label{eq:c11-plane-wave-reference-data}
\end{equation}
Indeed, \eqref{eq:c11-plane-wave-reference-data} follows from
\eqref{eq:c11-surface-change-of-variables} and
\eqref{eq:c11-surface-piola}.
For complex-valued boundary transformations, we define the extension
by the same formula.
The map
$\mathscr C^{1,1}(\widehat\Gamma;\mathbb C^3)\ni\br\mapsto
\boldsymbol p\exp(\ii\kappa\boldsymbol d\cdot\br)
\in\mathscr C^{0,1}(\widehat\Gamma;\mathbb C^3)$
is entire and bounded on bounded sets.
The bounded linear dependence of
$\widehat{\mathsf F}_\br$ on $\br$ and the bounded complex bilinear duality
between $\boldsymbol H^{-\frac12}(\widehat\Gamma)$ and
$\boldsymbol H^{\frac12}(\widehat\Gamma)$
therefore imply that the map
$\mathfrak T\ni\br\mapsto\widehat{\boldsymbol f}_\br\in\widehat X'$
admits a bounded holomorphic extension to $\mathfrak T_\delta$.
By Corollary~\ref{cor:c11-closed-solution-holomorphy}, item~(ii),
the corresponding map
$\mathfrak T\to\widehat X:\br\mapsto\widehat{\boldsymbol j}_\br$
therefore admits a bounded holomorphic extension to a common
$\mathfrak T_\delta$.

\subsection{Shape Holomorphy of the Far-Field}
\label{sec:c11-closed-far-field}

For each $\br\in\mathfrak T$, let $\boldsymbol j_\br\in X_\br$
be the unique solution of the EFIE~\eqref{eq:c11-closed-trace-efie}
under the hypotheses of Corollary~\ref{cor:c11-closed-solution-holomorphy},
item~(ii).
The corresponding outgoing electric field is
$\boldsymbol E_\br^{\mathrm{sc}}\coloneqq
\mathsf{SL}_{\Gamma_\br}^{(\kappa)}\boldsymbol j_\br$
in $\text{D}_\br^{\text{c}}$.
The electric far-field pattern $\boldsymbol E_\br^\infty$
of $\boldsymbol E_\br^{\mathrm{sc}}$
is defined by the asymptotic expansion
\begin{equation}
	\boldsymbol E_\br^{\mathrm{sc}}(\rho\boldsymbol\theta)
	=\frac{\exp(\ii\kappa\rho)}{\rho}
	 \bigl(\boldsymbol E_\br^\infty(\boldsymbol\theta)
	       +\mathcal O(\rho^{-1})\bigr),
	\quad \rho\to\infty,
	\quad\forall\boldsymbol\theta\in\mathbb S^2.
	\label{eq:c11-far-field-asymptotic}
\end{equation}
The expansion is uniform in $\boldsymbol\theta\in\mathbb S^2$, and
$\boldsymbol E_\br^\infty\in\mathscr C^0(\mathbb S^2;\mathbb C^3)$
satisfies $\boldsymbol\theta\cdot\boldsymbol E_\br^\infty(\boldsymbol\theta)=0$
for all $\boldsymbol\theta\in\mathbb S^2$
\cite[Section~4.2]{ColtonKress2013}.

With the normalizations in~\eqref{eq:c11-electric-potential}
and~\eqref{eq:c11-far-field-asymptotic}, the far-field formula
\cite[Eq.~(2.10)]{CostabelLeLouer2012ShapeII} reads
\begin{equation}
	\boldsymbol E_\br^\infty(\boldsymbol\theta)
	=\frac{\ii\kappa}{4\pi}
	 (\mathsf I_3-\boldsymbol\theta\boldsymbol\theta^\top)
	 \int_{\Gamma_\br}
	 \exp(-\ii\kappa\boldsymbol\theta\cdot\mathbf y)
	 \boldsymbol j_\br(\mathbf y)\,\dd s_{\mathbf y},
	\quad\forall\boldsymbol\theta\in\mathbb S^2.
	\label{eq:c11-physical-far-field}
\end{equation}
The integral is understood as the componentwise
$H^{-\frac12}(\Gamma_\br)\times H^{\frac12}(\Gamma_\br)$ duality pairing.

For each $\br\in\mathfrak T$, we define the far-field operator on the
reference space
$\widehat{\mathsf F}_{\br,\infty}^{(\kappa)}:
\widehat X\to\mathscr C^0(\mathbb S^2;\mathbb C^3)$ by
\begin{equation}
	(\widehat{\mathsf F}_{\br,\infty}^{(\kappa)}
	       \widehat{\boldsymbol j})(\boldsymbol\theta)
	\coloneqq\frac{\ii\kappa}{4\pi}
	 (\mathsf I_3-\boldsymbol\theta\boldsymbol\theta^\top)
	 \int_{\widehat\Gamma}
	 \exp(-\ii\kappa\boldsymbol\theta\cdot\br(\widehat{\mathbf y}))
	 F_\br(\widehat{\mathbf y})\widehat{\boldsymbol j}(\widehat{\mathbf y})
	 \,\dd s_{\widehat{\mathbf y}}.
	\label{eq:c11-reference-far-field}
\end{equation}
This definition holds for all $\widehat{\boldsymbol j}\in\widehat X$
and all $\boldsymbol\theta\in\mathbb S^2$.
The integral is understood componentwise in the
$H^{-\frac12}(\widehat\Gamma)\times H^{\frac12}(\widehat\Gamma)$ pairing,
with $F_\br\widehat{\boldsymbol j}
=\widehat{\mathsf F}_\br\widehat{\boldsymbol j}$ in the first argument.
The Piola formula~\eqref{eq:c11-surface-piola} and the pullback
identity~\eqref{eq:c11-distribution-pullbacks} give
\begin{equation}
	\boldsymbol E_\br^\infty
	=\widehat{\mathsf F}_{\br,\infty}^{(\kappa)}
	   \widehat{\boldsymbol j}_\br,
	\quad\forall\br\in\mathfrak T.
	\label{eq:c11-far-field-composition}
\end{equation}
Equation~\eqref{eq:c11-reference-far-field} defines
$\widehat{\mathsf F}_{\br,\infty}^{(\kappa)}$
for all $\br\in\mathscr C^{1,1}(\widehat\Gamma;\mathbb C^3)$.

\Needspace{8\baselineskip}
\begin{corollary}[Shape holomorphy of the electric far-field]
\label{cor:c11-closed-far-field-holomorphy}
Assume the setting of
Corollary~\ref{cor:c11-closed-solution-holomorphy}, item~(ii), and set
$\boldsymbol E_\br^\infty\coloneqq
\widehat{\mathsf F}_{\br,\infty}^{(\kappa)}
\mathcal P_\br\boldsymbol j_\br$.
Then there exists $\delta\in(0,\delta_f]$ such that the map
\begin{equation}
	\mathfrak T\to\mathscr C^0(\mathbb S^2;\mathbb C^3):
	\br\mapsto\boldsymbol E_\br^\infty
\end{equation}
admits a bounded holomorphic extension to $\mathfrak T_\delta$.
In particular, the conclusion holds for the plane-wave data in
Subsection~\ref{sec:c11-closed-solution}.
\end{corollary}

\begin{proof}
The map
$\mathscr C^{0,1}(\widehat\Gamma;\mathbb C^3)\to
\mathscr C^0(\mathbb S^2;\mathscr C^{0,1}(\widehat\Gamma)):
\br\mapsto\bigl(\boldsymbol\theta\mapsto
\exp(-\ii\kappa\boldsymbol\theta\cdot\br(\cdot))\bigr)$
is entire
and bounded on bounded sets
\cite[Sections~5 and~13]{Mujica1986}.
The bounded embedding
$\mathscr C^{0,1}(\widehat\Gamma)\hookrightarrow
H^{\frac12}(\widehat\Gamma)$, the bounded linear dependence of
$\widehat{\mathsf F}_\br$ on $\br$, and the bounded bilinear duality
between $H^{-\frac12}(\widehat\Gamma)$ and $H^{\frac12}(\widehat\Gamma)$,
applied componentwise, show that
\eqref{eq:c11-reference-far-field} defines the entire map
$\mathscr C^{1,1}(\widehat\Gamma;\mathbb C^3)\to
\mathscr L(\widehat X,\mathscr C^0(\mathbb S^2;\mathbb C^3)):
\br\mapsto\widehat{\mathsf F}_{\br,\infty}^{(\kappa)}$,
bounded on bounded sets.
By~\eqref{eq:c11-far-field-composition},
Corollary~\ref{cor:c11-closed-solution-holomorphy}, and
\cite[Proposition~4.20, item~(iii)]{HS21}, the map
$\mathfrak T\ni\br\mapsto\boldsymbol E_\br^\infty
\in\mathscr C^0(\mathbb S^2;\mathbb C^3)$ admits a bounded holomorphic
extension to $\mathfrak T_\delta$, with $0<\delta\leq\delta_f$
chosen within the neighborhood of the solution extension.
\end{proof}

Following the treatment of squared far-field norms in
\cite[Section~5.3]{DolzHenriquez2024} and
\cite[Appendix~A]{DolzHenriquez2026FullyDiscrete},
we define, for each fixed $\boldsymbol\theta\in\mathbb S^2$,
$\Phi_{\boldsymbol\theta}:\mathfrak T\to\mathbb R$
and $\Phi:\mathfrak T\to\mathbb R$ by
\begin{equation}
	\Phi_{\boldsymbol\theta}(\br)
	\coloneqq\|\boldsymbol E_\br^\infty(\boldsymbol\theta)\|_{\mathbb C^3}^2
	\quad\text{and}\quad
	\Phi(\br)
	\coloneqq\|\boldsymbol E_\br^\infty\|_{\boldsymbol L^2(\mathbb S^2)}^2,
	\quad\forall\br\in\mathfrak T.
\end{equation}
Let $\boldsymbol E^\infty$ be the holomorphic far-field extension
provided by Corollary~\ref{cor:c11-closed-far-field-holomorphy}, and set
\begin{equation}
	\Phi_{\boldsymbol\theta}^{\mathbb C}(\br)
	\coloneqq\boldsymbol E^\infty(\br)(\boldsymbol\theta)\cdot
	 \overline{\boldsymbol E^\infty(\overline{\br})(\boldsymbol\theta)}
	\quad\text{and}\quad
	\Phi^{\mathbb C}(\br)
	\coloneqq\int_{\mathbb S^2}
	 \Phi_{\boldsymbol\theta}^{\mathbb C}(\br)\,\dd s_{\boldsymbol\theta}.
	\label{eq:c11-squared-far-field-extension}
\end{equation}

\begin{corollary}
\label{cor:c11-squared-far-field-holomorphy}
Let the assumptions of
Corollary~\ref{cor:c11-closed-far-field-holomorphy} hold.
The maps $\Phi_{\boldsymbol\theta}:\mathfrak T\to\mathbb R$ and
$\Phi:\mathfrak T\to\mathbb R$ admit bounded holomorphic extensions
$\Phi_{\boldsymbol\theta}^{\mathbb C}:\mathfrak T_\delta\to\mathbb C$ and
$\Phi^{\mathbb C}:\mathfrak T_\delta\to\mathbb C$, respectively,
with $\delta$ as in
Corollary~\ref{cor:c11-closed-far-field-holomorphy}.
In particular, for
$U_\delta\coloneqq\mathfrak T_\delta\cap
\mathscr C^{1,1}(\widehat\Gamma;\mathbb R^3)$, the maps
$\Phi_{\boldsymbol\theta}^{\mathbb C}|_{U_\delta},
\Phi^{\mathbb C}|_{U_\delta}:U_\delta\to\mathbb R$
are real analytic extensions of $\Phi_{\boldsymbol\theta}$ and $\Phi$,
respectively.
\end{corollary}

\begin{proof}
The set $\mathfrak T$ consists of real maps and its complex neighborhood
$\mathfrak T_\delta$ is invariant under conjugation.
Conjugation invariance makes the expressions
in~\eqref{eq:c11-squared-far-field-extension} well-defined.
Conjugating the local power series of $\boldsymbol E^\infty$ shows
that the map $\mathfrak T_\delta\to\mathscr C^0(\mathbb S^2;\mathbb C^3):
\br\mapsto\overline{\boldsymbol E^\infty(\overline{\br})}$ is holomorphic.
The pointwise dot product is bounded bilinear and integration over
$\mathbb S^2$ is bounded linear, so both extensions are holomorphic.
For real $\br\in\mathfrak T_\delta$, we have $\overline{\br}=\br$,
so both extensions are real-valued and equal the corresponding squared
norms. On $\mathfrak T$, they therefore agree with
$\Phi_{\boldsymbol\theta}$ and $\Phi$, respectively.
Writing
$M_{\infty,\delta}\coloneqq\sup_{\br\in\mathfrak T_\delta}
\|\boldsymbol E^\infty(\br)\|_{\mathscr C^0(\mathbb S^2;\mathbb C^3)}<\infty$,
we obtain
\begin{equation}
	\sup_{\br\in\mathfrak T_\delta}
	|\Phi_{\boldsymbol\theta}^{\mathbb C}(\br)|\leq M_{\infty,\delta}^2
	\quad\text{and}\quad
	\sup_{\br\in\mathfrak T_\delta}
	|\Phi^{\mathbb C}(\br)|\leq4\pi M_{\infty,\delta}^2.
	\qedhere
\end{equation}
\end{proof}

\begin{remark}
In \cite[Theorem~7.3]{EscapilJerez2026EFIE}, far-field holomorphy is
proved for countably parametrized affine deformations under
interior electric nonresonance, in
$\mathscr C^m(\mathbb S^2;\mathbb C^3)$ for every $m\geq0$.
Corollary~\ref{cor:c11-closed-far-field-holomorphy}
instead treats the deformation itself in
$\mathscr C^{1,1}(\widehat\Gamma;\mathbb C^3)$ and yields a bounded
holomorphic extension
$\mathfrak T_\delta\to\mathscr C^0(\mathbb S^2;\mathbb C^3)$.
The lack of holomorphy noted in
\cite[Remark~7.3]{EscapilJerez2026EFIE} concerns direct evaluation of
squared norms at complex parameters. Conjugation does not prevent the
construction of holomorphic functions that coincide with the squared
norms for real boundary transformations, as established by
Corollary~\ref{cor:c11-squared-far-field-holomorphy} for the functions
defined in~\eqref{eq:c11-squared-far-field-extension}.
\end{remark}

\section{Shape Holomorphy of the Maxwell EFIE on Screens}
\label{sec:c11-screens}

We consider perfectly conducting screens obtained by deforming a
fixed patch of the reference boundary. Extending densities by zero and
restricting the resulting functionals to screen test fields transfers
the EFIE operator and its shape holomorphy from
Section~\ref{sec:c11-efie} to screens. In addition to
Assumption~\ref{ass:c11-deformations}, we impose the following condition.

\begin{assumption}[Admissible screens]
\label{ass:c11-screen}
There exists a relatively open, simply connected $\mathscr C^{1,1}$ subdomain
$\widehat{\Gamma}^s \subsetneq \widehat{\Gamma}$ with nonempty relative boundary
$\partial_{\widehat{\Gamma}}\widehat{\Gamma}^s$ such that
$\widehat{\Gamma}\setminus\overline{\widehat{\Gamma}^s}$ is nonempty and simply
connected and $\Gamma^s_\br = \br(\widehat{\Gamma}^s)$, for each
$\br \in \mathfrak{T}$.
\end{assumption}

\subsection{Sobolev Spaces on Screens}
\label{sec:c11-screen-scalar}

For each $S\in\{\Gamma_\br^s,\widehat\Gamma^s\}$, we use the standard
screen Sobolev space $H^{\frac12}(S)$, equipped with the quotient
norm induced by restriction from its supporting boundary. We set
$\widetilde H^{-\frac12}(S)\coloneqq H^{\frac12}(S)'$
with $L^2(S)$ as pivot.

We also use the vector-valued spaces
$\boldsymbol H^{\frac12}(S)\coloneqq (H^{\frac12}(S))^3$ and
$\widetilde{\boldsymbol H}^{-\frac12}(S)
\coloneqq(\widetilde H^{-\frac12}(S))^3$.
Let $\mathsf R_\br^{(D)}:H^{\frac12}(\Gamma_\br)\to
H^{\frac12}(\Gamma_\br^s)$ denote the restriction operator to $\Gamma_\br^s$.
Its dual operator
$\mathsf E_\br^{(N)}\coloneqq(\mathsf R_\br^{(D)})':
\widetilde H^{-\frac12}(\Gamma_\br^s)\to
H^{-\frac12}(\Gamma_\br)$ is the extension by zero,
defined by
\begin{equation}
	\langle\mathsf E_\br^{(N)}f,v\rangle_{\Gamma_\br}
	=\langle f,\mathsf R_\br^{(D)}v\rangle_{\Gamma_\br^s},
\end{equation}
for all $f\in\widetilde H^{-\frac12}(\Gamma_\br^s)$ and
$v\in H^{\frac12}(\Gamma_\br)$.
On the reference boundary, we denote by
$\widehat{\mathsf R}^{(D)}:H^{\frac12}(\widehat\Gamma)\to
H^{\frac12}(\widehat\Gamma^s)$ and
$\widehat{\mathsf E}^{(N)}\coloneqq(\widehat{\mathsf R}^{(D)})':
\widetilde H^{-\frac12}(\widehat\Gamma^s)\to
H^{-\frac12}(\widehat\Gamma)$ the corresponding operators for restriction
and extension by zero on the reference boundary $\widehat\Gamma$
and the reference screen $\widehat\Gamma^s$.
The restriction $\mathsf R_\br^{(D)}$ is a bounded surjection and its
dual $\mathsf E_\br^{(N)}$ is an isometric injection, as are their
reference counterparts.
These restriction and zero-extension operators act componentwise on
vector fields, with the same notation.

We use the duality
$\langle\cdot,\cdot\rangle_{S}:
\widetilde H^{-\frac12}(S)\times H^{\frac12}(S)\to\mathbb C$.
For regular functions, it agrees with
$\langle f,\phi\rangle_{S}=\int_S f\phi\,\dd s$,
without complex conjugation.
For vector fields, the pairing is the sum of the scalar pairings over
the three components.

On $H^{\frac12}(\Gamma_\br^s)$ and its reference counterpart, we define
the pullback and its inverse by
$\tau_\br^s f\coloneqq f\circ\br|_{\widehat\Gamma^s}$ and
$(\tau_\br^s)^{-1}\widehat f
\coloneqq\widehat f\circ\br^{-1}|_{\Gamma_\br^s}$.
For all $f\in\widetilde H^{-\frac12}(\Gamma_\br^s)$ and
$\widehat f\in\widetilde H^{-\frac12}(\widehat\Gamma^s)$, we define
\begin{equation}
	\begin{aligned}
	\langle\tau_\br^s f,\widehat\phi\rangle_{\widehat\Gamma^s}
	&\coloneqq\langle f,(\tau_\br^s)^{-1}
		(J_{\text S,\br}^{-1}\widehat\phi)\rangle_{\Gamma_\br^s},\\
	\langle(\tau_\br^s)^{-1}\widehat f,\phi\rangle_{\Gamma_\br^s}
	&\coloneqq\langle\widehat f,J_{\text S,\br}\tau_\br^s\phi
		\rangle_{\widehat\Gamma^s},
	\end{aligned}
	\label{eq:c11-screen-distribution-pullbacks}
\end{equation}
for all $\widehat\phi\in H^{\frac12}(\widehat\Gamma^s)$ and
$\phi\in H^{\frac12}(\Gamma_\br^s)$, respectively.
The Jacobian factors are restricted to $\widehat\Gamma^s$.

\begin{lemma}
\label{lem:c11-screen-scalar-operators}
Let $\mathfrak T$ satisfy
Assumptions~\ref{ass:c11-deformations} and~\ref{ass:c11-screen}.
For each $\br\in\mathfrak T$, the following assertions hold.
\begin{itemize}
 \item[\normalfont{(i)}] For
 $S\in\{\Gamma_\br^s,\widehat\Gamma^s\}$ with supporting boundary
 $\Gamma$, multiplication by $m|_S$, $m\in\mathscr C^{0,1}(\Gamma)$,
 is bounded on $H^{\frac12}(S)$ and
 $\widetilde H^{-\frac12}(S)$, and commutes with restriction
 and extension by zero, respectively.
 \item[\normalfont{(ii)}] The pullbacks
 $\tau_\br^s:H^{\frac12}(\Gamma_\br^s)\to H^{\frac12}(\widehat\Gamma^s)$
 and
 $\tau_\br^s:\widetilde H^{-\frac12}(\Gamma_\br^s)\to
 \widetilde H^{-\frac12}(\widehat\Gamma^s)$
 are bounded isomorphisms with inverse $(\tau_\br^s)^{-1}$.
 They satisfy
 \begin{equation}
	\tau_\br^s\mathsf R_\br^{(D)}
	=\widehat{\mathsf R}^{(D)}\tau_\br
	\quad\text{and}\quad
	\tau_\br\mathsf E_\br^{(N)}
	=\widehat{\mathsf E}^{(N)}\tau_\br^s.
	\label{eq:c11-screen-scalar-transport}
 \end{equation}
\end{itemize}
\end{lemma}

\begin{proof}
The proof is given in Appendix~\ref{app:c11-screen-scalar}.
\end{proof}

On the reference screen, these properties hold for multiplication by
$J_{\text S,\br}$, by $J_{\text S,\br}^{-1}$, and by each entry of
$F_\br$. The screen single layer BIO is
\begin{equation}
	\mathsf V_{\Gamma_\br^s}^{(\kappa)}
	\coloneqq \mathsf R_\br^{(D)}\mathsf V_{\Gamma_\br}^{(\kappa)}
	      \mathsf E_\br^{(N)}
	\in\mathscr L\bigl(\widetilde H^{-\frac{1}{2}}(\Gamma_\br^s),
	                   H^{\frac{1}{2}}(\Gamma_\br^s)\bigr).
	\label{eq:c11-screen-singlelayer}
\end{equation}

\subsection{Maxwell Trace Spaces and the Piola Transformation on Screens}
For each $\br\in\mathfrak T$, we define the Maxwell trace space
$X_\br^s\coloneqq\widetilde{\boldsymbol H}^{-\frac{1}{2}}
(\operatorname{div}_{\Gamma_\br^s},\Gamma_\br^s)$
as the space of fields
$\boldsymbol j\in\widetilde{\boldsymbol H}^{-\frac{1}{2}}(\Gamma_\br^s)$
with zero normal component,
$\bigl(\mathsf R_\br^{(D)}\boldsymbol n_\br\bigr)
 \cdot\boldsymbol j=0$
in $\widetilde H^{-\frac12}(\Gamma_\br^s)$,
for which there exists $f\in\widetilde H^{-\frac{1}{2}}(\Gamma_\br^s)$
such that
\begin{equation}
	\langle f,\mathsf R_\br^{(D)}\phi\rangle_{\Gamma_\br^s}
	=-\langle\boldsymbol j,
	         \mathsf R_\br^{(D)}
	         \nabla_{\Gamma_\br}\phi\rangle_{\Gamma_\br^s},
\end{equation}
for all $\phi\in\mathscr C^{1,1}(\Gamma_\br)$.
The density $f$ is uniquely determined by this identity.
We define the surface divergence of $\boldsymbol j$ by
$\operatorname{div}_{\Gamma_\br^s}\boldsymbol j\coloneqq f$.
We equip $X_\br^s$ with the norm
\begin{equation}
	\|\boldsymbol j\|_{X_\br^s}^2
	\coloneqq \|\boldsymbol j\|_{\widetilde{\boldsymbol H}^{-\frac{1}{2}}(\Gamma_\br^s)}^2
	+\|\operatorname{div}_{\Gamma_\br^s}\boldsymbol j\|_
	      {\widetilde H^{-\frac{1}{2}}(\Gamma_\br^s)}^2,
	\quad \forall\boldsymbol j\in X_\br^s.
\end{equation}
The reference space
$\widehat X^s\coloneqq\widetilde{\boldsymbol H}^{-\frac{1}{2}}
(\operatorname{div}_{\widehat\Gamma^s},\widehat\Gamma^s)$
and its norm are defined analogously.
Since the tangential projector is bounded on
$\boldsymbol H^{\pm\frac12}$ for $\mathscr C^{1,1}$ surfaces and restrictions
of smooth functions are dense in $H^{\frac12}(\Gamma_\br^s)$, this
definition agrees with
\cite[Section~2.3, Definition~1]{BuffaChristiansen2003}.

For every $\phi\in\mathscr C^{1,1}(\Gamma_\br)$,
\begin{equation*}
 \begin{aligned}
 \bigl\langle\operatorname{div}_{\Gamma_\br}
      (\mathsf E_\br^{(N)}\boldsymbol j),\phi\bigr\rangle_{\Gamma_\br}
 &=-\bigl\langle\boldsymbol j,
      \mathsf R_\br^{(D)}\nabla_{\Gamma_\br}\phi
      \bigr\rangle_{\Gamma_\br^s}\\
 &=\bigl\langle\mathsf E_\br^{(N)}
      \operatorname{div}_{\Gamma_\br^s}\boldsymbol j,
      \phi\bigr\rangle_{\Gamma_\br}.
 \end{aligned}
\end{equation*}
Therefore,
\begin{equation}
	\operatorname{div}_{\Gamma_\br}
	\bigl(\mathsf E_\br^{(N)}\boldsymbol j\bigr)
	=\mathsf E_\br^{(N)}
	      \operatorname{div}_{\Gamma_\br^s}\boldsymbol j,
	\label{eq:c11-screen-zero-extension-divergence}
\end{equation}
in $H^{-\frac{1}{2}}(\Gamma_\br)$.

Extension by zero preserves the zero normal component.
The scalar extension bound and
\eqref{eq:c11-screen-zero-extension-divergence} therefore show that
it is bounded in the Maxwell norm. The same holds on the reference
screen.
For each $\br\in\mathfrak T$, let
$\mathsf E_{\br,X}:X_\br^s\to X_\br$ denote extension by zero
and define its dual restriction operator
$\mathsf R_{\br,X}\coloneqq \mathsf E_{\br,X}':
X_\br'\to X_\br^{s\prime}$.
The corresponding reference operators are denoted by
$\widehat{\mathsf E}_X$ and $\widehat{\mathsf R}_X$.

Using the surface Piola formula~\eqref{eq:c11-surface-piola}, we define
for each $\br\in\mathfrak T$ the map
$\mathcal P_\br^s:X_\br^s\to
\widetilde{\boldsymbol H}^{-\frac12}(\widehat\Gamma^s)$
by
\begin{equation}
	\mathcal P_\br^s\boldsymbol j
	\coloneqq
	\widehat{\mathsf R}^{(D)}(J_{\text S,\br}F_\br^{-1})
	      \tau_\br^s\boldsymbol j,
	\quad \forall\boldsymbol j\in X_\br^s.
\end{equation}
Here $\widehat{\mathsf R}^{(D)}$ acts entrywise on the coefficient matrix.
The next lemma shows that the screen Piola map is an isomorphism and
commutes with extension by zero. This compatibility yields the transported
screen EFIE from the closed surface operator.

\begin{lemma}
\label{lem:c11-screen-extension}
Let $\mathfrak T$ satisfy
Assumptions~\ref{ass:c11-deformations} and~\ref{ass:c11-screen}.
For each $\br\in\mathfrak T$, the maps
 \begin{equation}
 	\mathcal P_\br^s\in\mathscr L_{\text{\normalfont iso}}(X_\br^s,\widehat X^s)
 	\quad\text{and}\quad
 	(\mathcal P_\br^s)^{-1}\in\mathscr L_{\text{\normalfont iso}}(\widehat X^s,X_\br^s)
 \end{equation}
 satisfy
 \begin{equation}
 	\widehat{\mathsf E}_X\mathcal P_\br^s
 	=\mathcal P_\br\mathsf E_{\br,X}.
 	\label{eq:c11-screen-piola-extension}
 \end{equation}
\end{lemma}

\begin{proof}
For each $\boldsymbol j\in X_\br^s$,
\eqref{eq:c11-screen-scalar-transport} and
Lemma~\ref{lem:c11-screen-scalar-operators}, item~(i), give
\begin{equation}
	\widehat{\mathsf E}^{(N)}
	    \mathcal P_\br^s\boldsymbol j
	=\mathcal P_\br\mathsf E_{\br,X}\boldsymbol j.
	\label{eq:c11-screen-piola-components}
\end{equation}
Since $\mathcal P_\br:X_\br\to\widehat X$, the field
$\mathcal P_\br\mathsf E_{\br,X}\boldsymbol j$ belongs to
$\widehat X\subset\boldsymbol H^{-\frac12}(\widehat\Gamma)$.
In particular,
$\widehat{\boldsymbol n}\cdot\mathcal P_\br\mathsf E_{\br,X}\boldsymbol j=0$
in $H^{-\frac12}(\widehat\Gamma)$.
By~\eqref{eq:c11-piola-divergence} and
\eqref{eq:c11-screen-zero-extension-divergence}, its surface divergence
on $\widehat\Gamma$ is the zero extension of
$\widehat{\mathsf R}^{(D)}(J_{\text S,\br})\tau_\br^s
\operatorname{div}_{\Gamma_\br^s}\boldsymbol j$.
Since $\widehat{\mathsf E}^{(N)}$ is injective and commutes with multiplication,
$\widehat{\mathsf R}^{(D)}(\widehat{\boldsymbol n})
 \cdot\mathcal P_\br^s\boldsymbol j=0$
in $\widetilde H^{-\frac12}(\widehat\Gamma^s)$.
For all $\widehat\phi\in\mathscr C^{1,1}(\widehat\Gamma)$, the formula
for $\operatorname{div}_{\widehat\Gamma}
(\mathcal P_\br\mathsf E_{\br,X}\boldsymbol j)$ and the duality defining
$\widehat{\mathsf E}^{(N)}$ give the variational identity
\begin{equation*}
 \bigl\langle\mathcal P_\br^s\boldsymbol j,
       \widehat{\mathsf R}^{(D)}
          \nabla_{\widehat\Gamma}\widehat\phi\bigr\rangle_{\widehat\Gamma^s}
 ={}
 -\bigl\langle
       \widehat{\mathsf R}^{(D)}(J_{\text S,\br})\tau_\br^s
          \operatorname{div}_{\Gamma_\br^s}\boldsymbol j,
       \widehat{\mathsf R}^{(D)}\widehat\phi
       \bigr\rangle_{\widehat\Gamma^s}.
\end{equation*}
The variational identity shows that
$\mathcal P_\br^s\boldsymbol j\in\widehat X^s$ and
\begin{equation}
	\operatorname{div}_{\widehat\Gamma^s}
	    (\mathcal P_\br^s\boldsymbol j)
	=\widehat{\mathsf R}^{(D)}(J_{\text S,\br})\tau_\br^s
	    \operatorname{div}_{\Gamma_\br^s}\boldsymbol j.
	\label{eq:c11-screen-piola-divergence}
\end{equation}
The field formula for $\mathcal P_\br^s$,
\eqref{eq:c11-screen-piola-divergence}, and the multiplier and
pullback estimates in Lemma~\ref{lem:c11-screen-scalar-operators},
items~(i) and~(ii),
give boundedness from $X_\br^s$ to $\widehat X^s$.

Applying the same divergence calculation to the inverse surface
diffeomorphism $\br^{-1}:\Gamma_\br^s\to\widehat\Gamma^s$, together with
the screen multiplier and pullback estimates, gives the bounded map
\begin{equation}
	(\mathcal P_\br^s)^{-1}\widehat{\boldsymbol j}
	=(\tau_\br^s)^{-1}
	  \bigl(\widehat{\mathsf R}^{(D)}(J_{\text S,\br}^{-1}F_\br)
	         \widehat{\boldsymbol j}\bigr)\in X_\br^s,
	\quad\forall\widehat{\boldsymbol j}\in\widehat X^s.
\end{equation}
The tangential-projector identities
$F_\br^{-1}F_\br=\widehat P$ and
$F_\br F_\br^{-1}=\mathsf I_3-\boldsymbol N_\br\boldsymbol N_\br^\top$,
together with tangency and injectivity of $\tau_\br^s$, show that these
maps are mutually inverse.
Since $\widehat{\mathsf E}_X$ acts component-wise as
$\widehat{\mathsf E}^{(N)}$,
\eqref{eq:c11-screen-piola-components} is
\eqref{eq:c11-screen-piola-extension}, completing the proof.
\end{proof}

\subsection{Shape Holomorphy of the EFIE on Screens}
For each $\br\in\mathfrak T$, we define the bilinear form
$\mathsf a_{\Gamma_\br^s}^{(\kappa)}:
X_\br^s\times X_\br^s\to\mathbb C$
for the EFIE on a screen as 
\begin{equation}
	\mathsf a_{\Gamma_\br^s}^{(\kappa)}(\boldsymbol j,\boldsymbol v)
	\coloneqq\ii\kappa\,
	\dual{\boldsymbol v}
	     {\mathsf V_{\Gamma_\br^s}^{(\kappa)}\boldsymbol j}
	     {\Gamma_\br^s}
	+\frac{1}{\ii\kappa}\,
	\dual{\operatorname{div}_{\Gamma_\br^s}\boldsymbol v}
	     {\mathsf V_{\Gamma_\br^s}^{(\kappa)}
	      \operatorname{div}_{\Gamma_\br^s}\boldsymbol j}
	     {\Gamma_\br^s},
\end{equation}
for all $\boldsymbol j,\boldsymbol v\in X_\br^s$.
The pairings are those fixed in Subsection~\ref{sec:c11-screen-scalar},
and the single layer BIO is defined in~\eqref{eq:c11-screen-singlelayer}.
Furthermore, for each $\br \in \mathfrak{T}$, this bilinear form induces the operator
$\mathsf A_{\Gamma_\br^s}^{(\kappa)}:X_\br^s\to X_\br^{s\prime}$
through
\begin{equation}
	\dual{\mathsf A_{\Gamma_\br^s}^{(\kappa)}\boldsymbol j}
	     {\boldsymbol v}{X_\br^{s\prime}\times X_\br^s}
	\coloneqq\mathsf a_{\Gamma_\br^s}^{(\kappa)}(\boldsymbol j,\boldsymbol v),
	\quad\forall\boldsymbol j,\boldsymbol v\in X_\br^s.
\end{equation}
The definition~\eqref{eq:c11-screen-singlelayer} and the divergence
identity~\eqref{eq:c11-screen-zero-extension-divergence} give
\begin{equation}
	\mathsf A_{\Gamma_\br^s}^{(\kappa)}
	=\mathsf R_{\br,X}\mathsf A_{\Gamma_\br}^{(\kappa)}
	      \mathsf E_{\br,X}.
	\label{eq:c11-screen-efie-compression}
\end{equation}
Indeed, for all $f,g\in\widetilde H^{-\frac{1}{2}}(\Gamma_\br^s)$
and each $\br \in \mathfrak{T}$,
$\mathsf E_\br^{(N)}=(\mathsf R_\br^{(D)})'$ and the
single layer formula~\eqref{eq:c11-screen-singlelayer} yield
\begin{equation}
	\langle g,\mathsf V_{\Gamma_\br^s}^{(\kappa)}f\rangle_{\Gamma_\br^s}
	=\langle\mathsf E_\br^{(N)}g,
	      \mathsf V_{\Gamma_\br}^{(\kappa)}\mathsf E_\br^{(N)}f
	      \rangle_{\Gamma_\br}.
\end{equation}
Applying this identity to each vector component and to the scalar
divergence term, using~\eqref{eq:c11-screen-zero-extension-divergence},
proves~\eqref{eq:c11-screen-efie-compression}.
The resulting screen EFIE agrees with the formulation in
\cite[Section~3.3, Eq.~(73)]{BuffaChristiansen2003}
up to the factor $\ii\kappa$.

For each $\br\in\mathfrak T$, the transported screen EFIE operator is
\begin{equation}
	\widehat{\mathsf A}_{\br,s}^{(\kappa)}
	\coloneqq ((\mathcal P_\br^s)^{-1})'
	    \mathsf A_{\Gamma_\br^s}^{(\kappa)}(\mathcal P_\br^s)^{-1}
	\in\mathscr L(\widehat X^s,\widehat X^{s\prime}).
\end{equation}

\begin{theorem}[Shape holomorphy of the screen EFIE]
\label{thm:c11-screen-holomorphy}
Let $\mathfrak T$ satisfy
Assumptions~\ref{ass:c11-deformations} and~\ref{ass:c11-screen}.
Then there exists $\delta>0$ such that the map
\begin{equation}
	\widehat{\mathcal A}_s^{(\kappa)}:
	\mathfrak T\to\mathscr L(\widehat X^s,\widehat X^{s\prime}):
	\br\mapsto\widehat{\mathsf A}_{\br,s}^{(\kappa)}
\end{equation}
admits a bounded holomorphic extension to $\mathfrak T_\delta$.
\end{theorem}

\begin{proof}
Recalling~\eqref{eq:c11-screen-piola-extension}
and~\eqref{eq:c11-screen-efie-compression}, we obtain
\begin{equation}\label{eq:c11-screen-fixed-compression}
	\widehat{\mathsf A}_{\br,s}^{(\kappa)}
	=
	\widehat{\mathsf R}_X
	\widehat{\mathsf A}_\br^{(\kappa)}
	\widehat{\mathsf E}_X.
\end{equation}
The reference extension and restriction operators are bounded and
independent of $\br$.
The map
$\mathscr L(\widehat X,\widehat X')\ni\mathsf{T}\mapsto
\widehat{\mathsf R}_X\mathsf{T}\widehat{\mathsf E}_X
\in\mathscr L(\widehat X^s,\widehat X^{s\prime})$
is bounded and complex linear.
Thus the map
$\mathfrak T\ni\br\mapsto\widehat{\mathsf A}_{\br,s}^{(\kappa)}
\in\mathscr L(\widehat X^s,\widehat X^{s\prime})$
admits a bounded holomorphic extension to $\mathfrak T_\delta$
by Theorem~\ref{thm:c11-efie-holomorphy}.
\end{proof}

\subsection{Shape Holomorphy of the EFIE Solution on a Screen}
\label{sec:c11-screen-solution}

Let $\boldsymbol E^{\mathrm{inc}}$ be an incident electric field as in
Subsection~\ref{sec:c11-closed-formulation}.
For each $\br\in\mathfrak T$, consider scattering by the screen
$\Gamma_\br^s$.
Set $\text{D}_\br^{\text{c},s}\coloneqq
\mathbb R^3\setminus\overline{\Gamma_\br^s}$.
The traces $\gamma_T^\pm$ and $\gamma_\times^\pm$ on $\Gamma_\br$
are taken from the exterior and interior of $\text{D}_\br$, respectively,
using the same normal $\boldsymbol n_\br$ on both sides.
For fields with these one-sided tangential traces, we define the
electric traces on the screen as
\begin{equation}
	\gamma_{T,s}^\pm\boldsymbol u
	\coloneqq\mathsf R_{\br,X}(\gamma_T^\pm\boldsymbol u)
	\in X_\br^{s\prime}.
\end{equation}
For the smooth incident field, the two traces agree and are denoted by
$\gamma_{T,s}\boldsymbol E^{\mathrm{inc}}$.
For each $\br\in\mathfrak T$, we seek
$\boldsymbol E_{\br,s}^{\mathrm{sc}}\in
\boldsymbol H_{\mathrm{loc}}(\curl,\mathbb R^3)$ such that
\begin{equation}
	\begin{aligned}
	\curl\curl\boldsymbol E_{\br,s}^{\mathrm{sc}}
	-\kappa^2\boldsymbol E_{\br,s}^{\mathrm{sc}}
	&=0 &&\text{in }\text{D}_\br^{\text{c},s},\\
	\gamma_{T,s}^+\boldsymbol E_{\br,s}^{\mathrm{sc}}
	=\gamma_{T,s}^-\boldsymbol E_{\br,s}^{\mathrm{sc}}
	&=-\gamma_{T,s}\boldsymbol E^{\mathrm{inc}}
	&&\text{in }X_\br^{s\prime},
	\end{aligned}
	\label{eq:c11-screen-scattering}
\end{equation}
together with the Silver--M\"uller radiation condition
\eqref{eq:c11-radiation}, with $\boldsymbol E_{\br,s}^{\mathrm{sc}}$
in place of $\boldsymbol E_\br^{\mathrm{sc}}$
\cite[Section~3.1]{BuffaChristiansen2003}.

For all $\mu\in\widetilde H^{-\frac12}(\Gamma_\br^s)$, the scalar
Helmholtz single layer potential is defined by
\begin{equation}
	(\mathsf S_{\Gamma_\br^s}^{(\kappa)}\mu)(\mathbf{x})
	\coloneqq\int_{\Gamma_\br^s}G^{(\kappa)}(\mathbf{x},\mathbf{y})
	                 \mu(\mathbf{y})\,\dd s_{\mathbf{y}},
	\quad 
	\mathbf{x}\in\text{D}_\br^{\text{c},s},
\end{equation}
where the integral is understood as the
$\widetilde H^{-\frac12}(\Gamma_\br^s)\times H^{\frac12}(\Gamma_\br^s)$
duality pairing. For all $\boldsymbol j\in X_\br^s$, the electric
potential generated by $\boldsymbol j$ is
\begin{equation}
	\mathsf{SL}_{\Gamma_\br^s}^{(\kappa)}\boldsymbol j
	\coloneqq\ii\kappa\mathsf S_{\Gamma_\br^s}^{(\kappa)}
	                 \boldsymbol j
	-\frac{1}{\ii\kappa}\nabla\mathsf S_{\Gamma_\br^s}^{(\kappa)}
	                 \operatorname{div}_{\Gamma_\br^s}\boldsymbol j.
	\label{eq:c11-screen-electric-potential}
	\noeqref{eq:c11-screen-electric-potential}
\end{equation}
Here $\mathsf S_{\Gamma_\br^s}^{(\kappa)}$ acts componentwise on the
density and as a scalar operator on its surface divergence.

By \cite[Section~3.2]{BuffaChristiansen2003}, the map
$\mathsf{SL}_{\Gamma_\br^s}^{(\kappa)}:X_\br^s\to
\boldsymbol H_{\mathrm{loc}}(\curl,\mathbb R^3)$ is continuous and linear.
For every $\boldsymbol j\in X_\br^s$, the resulting field satisfies
the Maxwell equation in $\text{D}_\br^{\text{c},s}$ and the
Silver--M\"uller radiation condition.

Zero extension leaves the scalar single layer potential unchanged, so
\eqref{eq:c11-screen-zero-extension-divergence} gives
$\mathsf{SL}_{\Gamma_\br^s}^{(\kappa)}\boldsymbol j
=\mathsf{SL}_{\Gamma_\br}^{(\kappa)}\mathsf E_{\br,X}\boldsymbol j$
in $\mathbb R^3\setminus\Gamma_\br$ for all $\boldsymbol j\in X_\br^s$.
Taking either restricted tangential trace and applying
\eqref{eq:c11-efie-trace-operator} and
\eqref{eq:c11-screen-efie-compression} yields
\begin{equation}
	\gamma_{T,s}^\pm\mathsf{SL}_{\Gamma_\br^s}^{(\kappa)}
	=\mathsf R_{\br,X}\mathsf A_{\Gamma_\br}^{(\kappa)}
	  \mathsf E_{\br,X}
	=\mathsf A_{\Gamma_\br^s}^{(\kappa)}
	\in\mathscr L(X_\br^s,X_\br^{s\prime}).
\end{equation}

Let $\boldsymbol E_{\br,s}^{\mathrm{sc}}$ solve the screen scattering
problem~\eqref{eq:c11-screen-scattering}. The screen Stratton--Chu
representation \cite[Theorem~3.2]{BuffaChristiansen2003} gives
\begin{equation}
	\boldsymbol E_{\br,s}^{\mathrm{sc}}
	=\mathsf{SL}_{\Gamma_\br^s}^{(\kappa)}\boldsymbol j_\br^s
	\quad\text{in }\text{D}_\br^{\text{c},s},
	\label{eq:c11-screen-representation}
	\noeqref{eq:c11-screen-representation}
\end{equation}
where $\boldsymbol j_\br^s\in X_\br^s$ is determined by
\begin{equation*}
	\mathsf E_{\br,X}\boldsymbol j_\br^s
	=\frac{1}{\ii\kappa}
	\bigl(\gamma_\times^+\curl\boldsymbol E_{\br,s}^{\mathrm{sc}}
	      -\gamma_\times^-\curl\boldsymbol E_{\br,s}^{\mathrm{sc}}\bigr)
	\quad\text{in }X_\br.
\end{equation*}
Taking either tangential trace on the screen and using the boundary
condition in~\eqref{eq:c11-screen-scattering} yields the screen EFIE:
for each $\br\in\mathfrak T$, we seek
$\boldsymbol j_\br^s\in X_\br^s$ such that
\begin{equation}
	\mathsf A_{\Gamma_\br^s}^{(\kappa)}\boldsymbol j_\br^s
	=\boldsymbol f_\br^s,
	\quad\boldsymbol f_\br^s\coloneqq-\gamma_{T,s}\boldsymbol E^{\mathrm{inc}}
	\in X_\br^{s\prime}.
	\label{eq:c11-screen-trace-efie}
\end{equation}
Equivalently,
\begin{equation}
	\mathsf a_{\Gamma_\br^s}^{(\kappa)}(\boldsymbol j_\br^s,\boldsymbol v)
	=\dual{\boldsymbol f_\br^s}{\boldsymbol v}{X_\br^{s\prime}\times X_\br^s},
	\quad \forall\boldsymbol v\in X_\br^s.
	\label{eq:c11-screen-variational-efie}
\end{equation}
Conversely, the electric potential associated with any solution of
\eqref{eq:c11-screen-trace-efie} satisfies the screen scattering problem.

Since $\Gamma_\br\setminus\overline{\Gamma_\br^s}$ is a nonempty
relatively open aperture connecting $\text{D}_\br$ with its exterior,
the screen complement $\text{D}_\br^{\text{c},s}$ is connected.
By \cite[Corollary~3.6]{BuffaChristiansen2003},
$\mathsf A_{\Gamma_\br^s}^{(\kappa)}:X_\br^s\to X_\br^{s\prime}$
is a bounded isomorphism, and the scattering problem is
uniquely solvable for all $\kappa>0$.

To study shape dependence, we transport the data and solution of the
EFIE to the reference screen by setting
$\widehat{\boldsymbol f}_\br^s\coloneqq
((\mathcal P_\br^s)^{-1})'\boldsymbol f_\br^s$ and
$\widehat{\boldsymbol j}_\br^s\coloneqq
\mathcal P_\br^s\boldsymbol j_\br^s$.
\Needspace{8\baselineskip}
\begin{corollary}[Shape holomorphy of the screen EFIE solution]
\label{cor:c11-screen-solution-holomorphy}
Let $\mathfrak T$ satisfy
Assumptions~\ref{ass:c11-deformations} and~\ref{ass:c11-screen}.
Then the following assertions hold.
\begin{itemize}
 \item[\normalfont{(i)}] There exists $\delta>0$ such that the map
 \begin{equation}
 	\mathfrak T\to\mathscr L(\widehat X^{s\prime},\widehat X^s):
 	\br\mapsto(\widehat{\mathsf A}_{\br,s}^{(\kappa)})^{-1}
 \end{equation}
 admits a bounded holomorphic extension to $\mathfrak T_\delta$.
 \item[\normalfont{(ii)}] For each $\br\in\mathfrak T$, let
 $\boldsymbol f_\br^s\in X_\br^{s\prime}$
 and assume that $\mathfrak T\ni\br\mapsto\widehat{\boldsymbol f}_\br^s
 \in\widehat X^{s\prime}$ admits a bounded holomorphic extension to
 $\mathfrak T_{\delta_f}$ for some $\delta_f>0$.
 For each $\br\in\mathfrak T$, we seek
 $\boldsymbol j_\br^s\in X_\br^s$ such that
 \begin{equation}
 	\mathsf A_{\Gamma_\br^s}^{(\kappa)}\boldsymbol j_\br^s
	=\boldsymbol f_\br^s.
 \end{equation}
 This problem is well-posed. There exists
 $\delta\in(0,\delta_f]$ such that the map
 \begin{equation}
 	\mathfrak T\to\widehat X^s:\br\mapsto\widehat{\boldsymbol j}_\br^s
 	=(\widehat{\mathsf A}_{\br,s}^{(\kappa)})^{-1}
 	     \widehat{\boldsymbol f}_\br^s
 \end{equation}
 admits a bounded holomorphic extension to $\mathfrak T_\delta$.
\end{itemize}
\end{corollary}

\begin{proof}
The bounded isomorphism
$\mathsf A_{\Gamma_\br^s}^{(\kappa)}:X_\br^s\to X_\br^{s\prime}$
from \cite[Corollary~3.6]{BuffaChristiansen2003} and
Lemma~\ref{lem:c11-screen-extension} imply that
$\widehat{\mathsf A}_{\br,s}^{(\kappa)}$ is boundedly invertible for each
$\br\in\mathfrak T$.
Theorem~\ref{thm:c11-screen-holomorphy} and the compactness and
holomorphic inversion argument used in the proof of
Corollary~\ref{cor:c11-closed-solution-holomorphy}
\cite[Proposition~4.20, items~(i) and~(ii)]{HS21}
give item~(i), with
\begin{equation}
	M_{\delta,s}\coloneqq\sup_{\br\in\mathfrak T_\delta}
	\|(\widehat{\mathsf A}_{\br,s}^{(\kappa)})^{-1}\|_
	      {\mathscr L(\widehat X^{s\prime},\widehat X^s)}<\infty.
\end{equation}
For item~(ii), choose $\delta\leq\delta_f$ within the inverse extension
neighborhood. As in that proof, the bounded bilinear evaluation
$\mathscr L(\widehat X^{s\prime},\widehat X^s)\times
\widehat X^{s\prime}\to\widehat X^s$ gives the asserted bounded
holomorphic solution map
\cite[Proposition~4.20, item~(iii)]{HS21}, with
\begin{equation}
	\sup_{\br\in\mathfrak T_\delta}
	      \|\widehat{\boldsymbol j}_\br^s\|_{\widehat X^s}
	\leq M_{\delta,s}\sup_{\br\in\mathfrak T_\delta}
	      \|\widehat{\boldsymbol f}_\br^s\|_{\widehat X^{s\prime}}
	<\infty.\qedhere
\end{equation}
\end{proof}

We now apply Corollary~\ref{cor:c11-screen-solution-holomorphy} to the
plane wave~\eqref{eq:c11-plane-wave}. The data transported to the
reference screen satisfy
\begin{equation}
	\dual{\widehat{\boldsymbol f}_\br^s}{\widehat{\boldsymbol v}}
	     {\widehat X^{s\prime}\times\widehat X^s}
	=-\dual{\widehat{\mathsf F}_\br\widehat{\mathsf E}_X\widehat{\boldsymbol v}}
	       {\boldsymbol p\exp(\ii\kappa\boldsymbol d\cdot\br)}
	       {\widehat\Gamma},
	\quad\forall\widehat{\boldsymbol v}\in\widehat X^s.
	\label{eq:c11-screen-plane-wave-data}
\end{equation}
Indeed, \eqref{eq:c11-screen-piola-extension} gives
$\widehat{\boldsymbol f}_\br^s
=\widehat{\mathsf R}_X\widehat{\boldsymbol f}_\br$.
Thus \eqref{eq:c11-screen-plane-wave-data} follows from
\eqref{eq:c11-plane-wave-reference-data}.
Since $\widehat{\mathsf R}_X$ is bounded and complex linear,
the holomorphic extension of the plane-wave data established in
Subsection~\ref{sec:c11-closed-solution} yields a bounded holomorphic
extension of the map
$\mathfrak T\ni\br\mapsto\widehat{\boldsymbol f}_\br^s
\in\widehat X^{s\prime}$ to $\mathfrak T_\delta$ for some $\delta>0$.

After decreasing $\delta$ if necessary,
Corollary~\ref{cor:c11-screen-solution-holomorphy}, item~(ii), therefore
shows that the map $\mathfrak T\to\widehat X^s:$
$\br\mapsto\widehat{\boldsymbol j}_\br^s$ admits a bounded holomorphic
extension to $\mathfrak T_\delta$.

\section{Concluding Remarks}
\label{sec:c11-conclusions}

We have established shape holomorphy of the complete
Helmholtz Calder\'on projector on closed $\mathscr C^{1,1}$ surfaces
and of the Maxwell EFIE on closed surfaces and screens of the same
regularity.
Under the stated data and solvability assumptions, the transported
EFIE solution maps admit bounded holomorphic extensions.
For closed surfaces, the electric far-field also admits a bounded
holomorphic extension.
Any fixed bounded linear functional of the transported solution or,
for closed surfaces, of the electric far-field also admits a bounded
holomorphic extension.

For any well-posed boundary integral formulation of a Helmholtz
problem, the transported shape-to-solution map inherits the shape
holomorphy of the corresponding BIO when the transported data depend
holomorphically on the deformation. This follows from holomorphy of
operator inversion \cite[Proposition~4.20]{HS21}.
This argument has been studied extensively for Laplace and Helmholtz
problems in \cite{Henriquez2021,HS21}, for multiple open arcs in
\cite{PHJ24}, and for parametric Helmholtz formulations in
\cite{DolzHenriquez2024}.

Since the holomorphy results are formulated directly in terms of the
boundary transformations, they can be applied to different
parametrizations without repeating the operator analysis.
Their use is therefore not restricted to affine parametrizations,
as in \cite[Section~2.2 and Theorem~6.1]{EscapilJerez2026EFIE},
but also includes nonlinear, possibly nonanalytic parametrizations
of Gevrey regularity, as in
\cite[Section~2]{DjurdjevacKaarniojaSchillingsZepernick2026}.
If the range of the parametrization lies in a compact family of
admissible transformations and its parametric derivative bounds hold
in the $\mathscr C^{1,1}(\widehat\Gamma;\mathbb R^3)$ norm,
composition with the holomorphic operator maps preserves the Gevrey order
\cite[Section~2.4]{HarbrechtSchmidlinSchwab2024}.
Under the stated solvability assumptions, the EFIE solution maps
inherit the same order when the transported data satisfy the
corresponding Gevrey bounds.

Parametric holomorphy obtained by composing the operator and solution
maps with holomorphic parametrizations of the boundary transformations
provides an analytic basis for sparse polynomial and sparse grid approximation
\cite{ChkifaCohenSchwab2014},
higher-order quasi-Monte Carlo integration \cite{DickLeGiaSchwab2016},
neural network approximation \cite{SchwabZech2019}, and reduced
basis methods \cite{CohenDeVore2016}. Dimension-independent rates
additionally require quantitative summability and stability.

Extensions of these results to Helmholtz and Maxwell problems on
closed surfaces and screens of Lipschitz regularity, with precise
control of the dependence of the domain of holomorphy on the
wavenumber $\kappa$, are currently underway.

\section*{Acknowledgments}

The author's work is funded by the Deutsche Forschungsgemeinschaft
(DFG, German Research Foundation) -- Project-ID 258734477 -- SFB 1173
and the Austrian Science Fund (FWF) under project I6667-N.
Funding was also received from the European Research Council (ERC)
under the Horizon Europe research and innovation programme of the
European Union (grant agreement No.~101125225, project ``New Frontiers
in Optimal Adaptivity'', PI: Michael Feischl).

\section*{Declaration of Generative AI Use}
Generative AI was used to assist with editing this manuscript and
reviewing its mathematical arguments. 
All proofs have been checked by the author, who takes full
responsibility for their correctness.

\appendix

\section{Surface Jacobians and Multiplication Operators}
\label{app:c11-jacobian}

\subsection{Proof of Lemma~\ref{lem:c11-jacobian-multiplier}}
\label{sec:proof_c11-jacobian-multiplier}
Fix $\br\in\mathfrak T$ and set
$\boldsymbol a_\br\coloneqq\operatorname{cof}(F_\br)\widehat{\boldsymbol n}$.
By~\eqref{eq:c11-surface-jacobian},
$J_{\text S,\br}=\|\boldsymbol a_\br\|>0$.
Both $F_\br$ and $\widehat{\boldsymbol n}$ are Lipschitz on $\widehat\Gamma$, so
$\boldsymbol a_\br\in\mathscr C^{0,1}(\widehat\Gamma;\IR^3)$.
The inequality $\bigl|\|\boldsymbol a\|-\|\boldsymbol b\|\bigr|
\leq\|\boldsymbol a-\boldsymbol b\|$ gives
\begin{equation}
	\|J_{\text S,\br}\|_{\mathscr C^0(\widehat\Gamma)}
	\leq\|\boldsymbol a_\br\|_{\mathscr C^0(\widehat\Gamma;\mathbb R^3)}
	\quad\text{and}\quad
	|J_{\text S,\br}|_{\mathscr C^{0,1}(\widehat\Gamma)}
	\leq|\boldsymbol a_\br|_{\mathscr C^{0,1}(\widehat\Gamma;\mathbb R^3)}.
\end{equation}
Compactness of $\widehat\Gamma$ gives
$c_\br\coloneqq \min_{\widehat\Gamma}J_{\text S,\br}>0$. In particular,
\begin{equation}
	\|J_{\text S,\br}^{-1}\|_{\mathscr C^0(\widehat\Gamma)}\leq c_\br^{-1}
	\quad\text{and}\quad
	|J_{\text S,\br}^{-1}|_{\mathscr C^{0,1}(\widehat\Gamma)}
	\leq c_\br^{-2}|J_{\text S,\br}|_{\mathscr C^{0,1}(\widehat\Gamma)}.
\end{equation}
For all $m\in\mathscr C^{0,1}(\widehat\Gamma)$ and
$u\in\mathscr C^{0,\alpha}(\widehat\Gamma)$,
the H\"older product estimate gives
\begin{equation}
	\|mu\|_{\mathscr C^{0,\alpha}(\widehat\Gamma)}
	\leq C_{\alpha,\widehat\Gamma}
	\|m\|_{\mathscr C^{0,1}(\widehat\Gamma)}
	\|u\|_{\mathscr C^{0,\alpha}(\widehat\Gamma)}.
\end{equation}
Multiplication by $m$ is also bounded on
$L^2(\widehat\Gamma)$ and $H^1(\widehat\Gamma)$, with
\begin{equation}
	\|mu\|_{L^2(\widehat\Gamma)}
	\leq\|m\|_{\mathscr C^0(\widehat\Gamma)}\|u\|_{L^2(\widehat\Gamma)}
	\quad\text{and}\quad
	\|mu\|_{H^1(\widehat\Gamma)}
	\leq C_{\widehat\Gamma}\|m\|_{\mathscr C^{0,1}(\widehat\Gamma)}
	       \|u\|_{H^1(\widehat\Gamma)}.
\end{equation}
By interpolation \cite[Theorem~B.11]{McLean2000},
we obtain, for $0<s<1$ and all $u\in H^s(\widehat\Gamma)$,
\begin{equation}
	\|mu\|_{H^s(\widehat\Gamma)}
	\leq C_{s,\widehat\Gamma}\|m\|_{\mathscr C^{0,1}(\widehat\Gamma)}
	            \|u\|_{H^s(\widehat\Gamma)}.
\end{equation}
For $0<s\leq1$ and all $g\in H^{-s}(\widehat\Gamma)$, we define multiplication
by duality as
\begin{equation}
	\dual{mg}{\phi}{\widehat\Gamma}
	\coloneqq \dual{g}{m\phi}{\widehat\Gamma},
	\quad \forall\phi\in H^s(\widehat\Gamma).
\end{equation}
Multiplication on $H^{-s}(\widehat\Gamma)$ is therefore the
dual of multiplication on $H^s(\widehat\Gamma)$.
It has the same operator norm, so the preceding multiplier estimate
also gives the required bound on $H^{-s}(\widehat\Gamma)$.
Taking $m=J_{\text S,\br}$ and $m=J_{\text S,\br}^{-1}$
proves the asserted multiplier bounds.

\subsection{Proof of Lemma~\ref{lem:c11-geometric-holomorphy}}
\label{sec:proof_c11-geometric-holomorphy}
The map $\mathscr C^{1,1}(\widehat\Gamma;\mathbb C^3)\ni\br\mapsto
F_\br\in\mathscr C^{0,1}(\widehat\Gamma;\mathbb C^{3\times3})$
is bounded and complex linear. Hence
$\boldsymbol a_\br\coloneqq
\operatorname{cof}(F_\br)\widehat{\boldsymbol n}$
depends polynomially on $\br$ with values in
$\mathscr C^{0,1}(\widehat\Gamma;\mathbb C^3)$.
Fix $\br_0\in\mathfrak T$. For complex-valued $\br$ near $\br_0$, set
\begin{equation}
\begin{aligned}
 J_{\text S,\br}
 &\coloneqq\sqrt{\boldsymbol a_\br\cdot\boldsymbol a_\br},
 \\
 F_\br^{-1}
 &\coloneqq
 (F_\br^\top F_\br+
 \widehat{\boldsymbol n}\widehat{\boldsymbol n}^\top)^{-1}
 F_\br^\top,
\end{aligned}
\label{eq:c11-complex-geometric-coefficients}
\end{equation}
and set $J_{\text S,\br}^{-1}\coloneqq(J_{\text S,\br})^{-1}$.
We also set
\begin{equation}
 \boldsymbol N_\br\coloneqq
 \frac{(\mathsf I_3-F_\br F_\br^{-1})\boldsymbol N_{\br_0}}
 {\sqrt{\boldsymbol N_{\br_0}\cdot
 (\mathsf I_3-F_\br F_\br^{-1})\boldsymbol N_{\br_0}}}.
\end{equation}
The dot products are bilinear, and the transposes do not involve
complex conjugation. Both square roots are chosen positive at $\br_0$.

At $\br_0$, the first square-root argument is $J_{\text S,\br_0}^2$,
which has a positive minimum on $\widehat\Gamma$, and the second is
identically one. The matrix to be inverted is uniformly positive
definite. Square roots, the matrix inverse, and the reciprocal of
$J_{\text S,\br}$ therefore depend holomorphically on the coefficients
near their values at $\br_0$.
Pointwise multiplication is continuous in
$\mathscr C^{0,1}(\widehat\Gamma)$ and
$\mathscr C^{0,1}(\widehat\Gamma;\mathbb C^{3\times3})$.
The formulas are holomorphic on a sufficiently small ball in
$\mathscr C^{1,1}(\widehat\Gamma;\mathbb C^3)$.
Since $\br_0\in\mathfrak T$, we may shrink this ball so that every
$\br$ in its intersection with
$\mathscr C^{1,1}(\widehat\Gamma;\mathbb R^3)$ is a
$\mathscr C^{1,1}$ diffeomorphism from $\widehat\Gamma$ onto
$\br(\widehat\Gamma)$.
For such transformations, the formulas give the surface Jacobian and
tangential inverse; the normal formula agrees with
$\boldsymbol n_\br\circ\br$, since
$(\boldsymbol n_\br\circ\br)\cdot\boldsymbol N_{\br_0}>0$
by continuity.

As in the proof of Proposition~\ref{prop:c11-singlelayer-holomorphy},
\cite[Theorem~5.34]{DudleyNorvaisa2011} shows that the extensions agree
on the overlaps. Combining a finite covering argument with the
compactness of $\mathfrak T$, we obtain bounded holomorphic extensions
of $J_{\text S,\br}$, $J_{\text S,\br}^{-1}$, $F_\br^{-1}$, and
$\boldsymbol N_\br$ to a common $\mathfrak T_\delta$, for some
$\delta>0$.

\section{Pullback Operators}
\label{app:c11-pullbacks}

\begin{proof}[Proof of Lemma~\ref{lem:c11-pullback-isomorphisms}]
Fix $\br\in\mathfrak T$.
By change of variables and the positive minimum and finite maximum
of $J_{\text S,\br}$ on $\widehat\Gamma$, we obtain
\begin{equation}
	\|\tau_\br f\|_{L^2(\widehat\Gamma)}
	\leq C_\br\|f\|_{L^2(\Gamma_\br)}
	\quad\text{and}\quad
	\|\tau_\br^{-1}\widehat f\|_{L^2(\Gamma_\br)}
	\leq C_\br\|\widehat f\|_{L^2(\widehat\Gamma)},
\end{equation}
for all $f\in L^2(\Gamma_\br)$ and
$\widehat f\in L^2(\widehat\Gamma)$.
Recalling the tangential inverse $F_\br^{-1}$,
the tangential chain rule gives, for all
$f\in\mathscr C^1(\Gamma_\br)$ and
$\widehat f\in\mathscr C^1(\widehat\Gamma)$,
\begin{equation}
	\nabla_{\widehat\Gamma}(\tau_\br f)
	=F_\br^\top\tau_\br(\nabla_{\Gamma_\br}f)
	\quad\text{and}\quad
	\tau_\br\bigl(\nabla_{\Gamma_\br}
	                   (\tau_\br^{-1}\widehat f)\bigr)
	=F_\br^{-\top}\nabla_{\widehat\Gamma}\widehat f.
\end{equation}
Using these chain rules, the boundedness of $F_\br$ and $F_\br^{-1}$,
and the preceding estimates, we obtain
\begin{equation}
	\|\tau_\br f\|_{H^1(\widehat\Gamma)}
	\leq C_\br\|f\|_{H^1(\Gamma_\br)}
	\quad\text{and}\quad
	\|\tau_\br^{-1}\widehat f\|_{H^1(\Gamma_\br)}
	\leq C_\br\|\widehat f\|_{H^1(\widehat\Gamma)}.
\end{equation}
By density, these identities and bounds hold for all
$f\in H^1(\Gamma_\br)$ and
$\widehat f\in H^1(\widehat\Gamma)$.
By interpolation \cite[Theorem~B.11]{McLean2000}, these bounds give
boundedness of $\tau_\br$ and
$\tau_\br^{-1}$ between $H^s(\Gamma_\br)$ and
$H^s(\widehat\Gamma)$ for $0<s<1$.
The identities $\tau_\br^{-1}\tau_\br f=f$ and
$\tau_\br\tau_\br^{-1}\widehat f=\widehat f$ hold for
$f\in H^s(\Gamma_\br)$ and $\widehat f\in H^s(\widehat\Gamma)$,
since these spaces are contained in $L^2(\Gamma_\br)$ and
$L^2(\widehat\Gamma)$, respectively.

For $f\in\mathscr C^{0,\alpha}(\Gamma_\br)$,
\begin{equation}
	|f\circ\br|_{\mathscr C^{0,\alpha}(\widehat\Gamma)}
	\leq\operatorname{Lip}(\br)^\alpha
	      |f|_{\mathscr C^{0,\alpha}(\Gamma_\br)}.
\end{equation}
The seminorm estimate gives boundedness of $\tau_\br$ from
$\mathscr C^{0,\alpha}(\Gamma_\br)$ to
$\mathscr C^{0,\alpha}(\widehat\Gamma)$.
The two surface gradient chain rules give boundedness from
$\mathscr C^{1,\alpha}(\Gamma_\br)$ to
$\mathscr C^{1,\alpha}(\widehat\Gamma)$, using the Lipschitz coefficients
$F_\br$ and $F_\br^{-1}$.
Applying these estimates with $\br^{-1}$ and the boundaries
interchanged gives the bounds for $\tau_\br^{-1}$.

For $-1\leq s<0$, consider the maps
$\widehat\phi\mapsto\tau_\br^{-1}
(J_{\text S,\br}^{-1}\widehat\phi)$ from
$H^{-s}(\widehat\Gamma)$ to $H^{-s}(\Gamma_\br)$ and
$\phi\mapsto J_{\text S,\br}\tau_\br\phi$ in the opposite direction.
Composing these two maps cancels the Jacobian factors and the
pullbacks, so each composition is the identity.
Their boundedness between $H^{-s}(\widehat\Gamma)$ and
$H^{-s}(\Gamma_\br)$ follows from the
positive-order pullback bounds and
Lemma~\ref{lem:c11-jacobian-multiplier}.
By~\eqref{eq:c11-distribution-pullbacks}, $\tau_\br$ and $\tau_\br^{-1}$
on $H^s(\Gamma_\br)$ and $H^s(\widehat\Gamma)$
are the duals of these maps. They are therefore bounded and
mutually inverse.
All constants may depend on $\br$,
but not on the functions to which the operators are applied.
\end{proof}

\section{Proof of Lemma~\ref{lem:c11-screen-scalar-operators}}
\label{app:c11-screen-scalar}

Fix $S\in\{\Gamma_\br^s,\widehat\Gamma^s\}$ with supporting boundary
$\Gamma$, and omit hats and deformation subscripts temporarily.

For all $u\in H^{\frac12}(S)$ and every extension
$v\in H^{\frac12}(\Gamma)$ with $\mathsf R^{(D)}v=u$, the multiplier
estimate of Appendix~\ref{app:c11-jacobian}, applied on $\Gamma$, gives
\begin{equation}
	\|(m|_S)u\|_{H^{\frac12}(S)}
	\leq\|mv\|_{H^{\frac12}(\Gamma)}
	\leq C_\Gamma\|m\|_{\mathscr C^{0,1}(\Gamma)}
		\|v\|_{H^{\frac12}(\Gamma)}.
\end{equation}
Taking the infimum over all such $v$ proves the multiplier bound
on $H^{\frac12}(S)$.
Restriction gives
$\mathsf R^{(D)}(mv)=(m|_S)\mathsf R^{(D)}v$
for all $v\in H^{\frac12}(\Gamma)$.
The dual multiplier is multiplication by the same function $m$
and has the same operator norm. Taking duals, we obtain
\begin{equation}
	\mathsf E^{(N)}\bigl((m|_S)f\bigr)
	=m\,\mathsf E^{(N)}f,
	\quad\forall f\in\widetilde H^{-\frac12}(S),
\end{equation}
which proves item~(i).

Composition with $\br$ gives
$\tau_\br^s\mathsf R_\br^{(D)}=\widehat{\mathsf R}^{(D)}\tau_\br$.
In particular, the screen pullback is independent of the chosen
extension to the supporting boundary.
For all $f\in H^{\frac12}(\Gamma_\br^s)$, the quotient norm gives
\begin{equation}
	\|\tau_\br^s f\|_{H^{\frac12}(\widehat\Gamma^s)}
	\leq\inf_{\substack{v\in H^{\frac12}(\Gamma_\br)\\
		\mathsf R_\br^{(D)}v=f}}
		\|\tau_\br v\|_{H^{\frac12}(\widehat\Gamma)}
	\leq C_\br\|f\|_{H^{\frac12}(\Gamma_\br^s)}.
\end{equation}
Interchanging $\Gamma_\br^s$ and $\widehat\Gamma^s$ and replacing
$\br$ by $\br^{-1}$ gives
$\|(\tau_\br^s)^{-1}\widehat f\|_{H^{\frac12}(\Gamma_\br^s)}
\leq C_\br\|\widehat f\|_{H^{\frac12}(\widehat\Gamma^s)}$
for all $\widehat f\in H^{\frac12}(\widehat\Gamma^s)$.
On the negative-order spaces,
\eqref{eq:c11-screen-distribution-pullbacks} identifies
$\tau_\br^s$ and $(\tau_\br^s)^{-1}$ with the duals of
the operators $(\tau_\br^s)^{-1}J_{\text S,\br}^{-1}$ and
$J_{\text S,\br}\tau_\br^s$, respectively.
These operators act from $H^{\frac12}(\widehat\Gamma^s)$ to
$H^{\frac12}(\Gamma_\br^s)$ and in the opposite direction, respectively.
The Jacobian factors and pullbacks cancel in both compositions,
so the two operators are inverses.
The bounds on $\tau_\br^s$ and $(\tau_\br^s)^{-1}$ just proved,
together with the multiplier bound in item~(i), give their
boundedness between $H^{\frac12}(\widehat\Gamma^s)$ and
$H^{\frac12}(\Gamma_\br^s)$.
Hence $\tau_\br^s$ and $(\tau_\br^s)^{-1}$ are bounded and mutually
inverse between $\widetilde H^{-\frac12}(\Gamma_\br^s)$ and
$\widetilde H^{-\frac12}(\widehat\Gamma^s)$.

Finally, for all $f\in\widetilde H^{-\frac12}(\Gamma_\br^s)$ and
$\widehat\phi\in H^{\frac12}(\widehat\Gamma)$, the duality formulas
\eqref{eq:c11-distribution-pullbacks} and
\eqref{eq:c11-screen-distribution-pullbacks},
the identity $\tau_\br^s\mathsf R_\br^{(D)}
=\widehat{\mathsf R}^{(D)}\tau_\br$ on $H^{\frac12}(\Gamma_\br)$,
and commutation of restriction with multiplication by
$J_{\text S,\br}^{-1}$ give
\begin{equation}
	\begin{aligned}
	\langle\tau_\br\mathsf E_\br^{(N)}f,\widehat\phi\rangle_{\widehat\Gamma}
	&=\bigl\langle f,(\tau_\br^s)^{-1}
		\widehat{\mathsf R}^{(D)}
		(J_{\text S,\br}^{-1}\widehat\phi)\bigr\rangle_{\Gamma_\br^s}\\
	&=\langle\tau_\br^s f,\widehat{\mathsf R}^{(D)}\widehat\phi
		\rangle_{\widehat\Gamma^s}\\
	&=\langle\widehat{\mathsf E}^{(N)}\tau_\br^s f,\widehat\phi
		\rangle_{\widehat\Gamma}.
	\end{aligned}
\end{equation}
This proves $\tau_\br\mathsf E_\br^{(N)}
=\widehat{\mathsf E}^{(N)}\tau_\br^s$ in
\eqref{eq:c11-screen-scalar-transport}, completing item~(ii).

\section{Lax's Lemma}
\label{app:c11-operator-extension}

We recall Lax's lemma in the form given in
\cite[Theorem~4.13]{Kress2014}, used in
Lemma~\ref{lem:c11-sobolev-extension}.

\begin{theorem}
\label{thm:c11-operator-extension}
Let $X$ and $Y$ be complex normed spaces with inner products
$(\cdot,\cdot)_X$ and $(\cdot,\cdot)_Y$. Let
$\|u\|_{X,0}\coloneqq(u,u)_X^{1/2}$ and
$\|v\|_{Y,0}\coloneqq(v,v)_Y^{1/2}$, and assume that
\begin{equation}
	\|u\|_{X,0}\leq c_X\|u\|_X
	\quad\text{and}\quad
	\|v\|_{Y,0}\leq c_Y\|v\|_Y,
	\quad\forall u\in X,\ \forall v\in Y,
\end{equation}
for some $c_X,c_Y>0$.
Suppose that $A\in\mathscr L(X,Y)$ and
$B\in\mathscr L(Y,X)$ satisfy
\begin{equation}
	(Au,v)_Y=(u,Bv)_X,
	\quad\forall u\in X,\ \forall v\in Y.
\end{equation}
Then
\begin{equation}
	\|Au\|_{Y,0}^2
	\leq\|A\|_{\mathscr L(X,Y)}
	     \|B\|_{\mathscr L(Y,X)}\|u\|_{X,0}^2,
	\quad\forall u\in X.
\end{equation}
Consequently, $A$ has a unique bounded extension between the
Hilbert completions of $X$ and $Y$ in their inner product norms,
with operator norm at most
$\bigl(\|A\|_{\mathscr L(X,Y)}
\|B\|_{\mathscr L(Y,X)}\bigr)^{1/2}$.
\end{theorem}

\bibliographystyle{amsplain}
\bibliography{references}
\end{document}